\documentclass[11pt]{amsart}

\usepackage{amscd,latexsym,amsthm,amsfonts,amssymb,amsmath,amsxtra,eucal}%,IEEEtrantools}%mathdots}
\usepackage[all]{xy}

\usepackage[pdftex,bookmarks,colorlinks,pdfmenubar]{hyperref}

\usepackage{verbatim}
\usepackage{mathabx}
\usepackage{mathrsfs}
\usepackage{bbm}
\usepackage{relsize}
\usepackage[bbgreekl]{mathbbol}
\usepackage{amsfonts}

\DeclareMathOperator{\CP}{\mathcal{P}}

\DeclareSymbolFontAlphabet{\mathbb}{AMSb}
\DeclareSymbolFontAlphabet{\mathbbl}{bbold}

\renewcommand\theequation{\thesection.\arabic{equation}}

\newcommand{\BC}{{\mathbb {C}}}

\newcommand{\BZ}{{\mathbb {Z}}}

\newcommand{\Qlb}{\overline{\mathbb{Q}}_\ell}

\newcommand{\Ad}{{\mathrm{Ad}}}

\newcommand{\Ind}{{\mathrm{Ind}}}

\newcommand{\rank}{{\mathrm{rank}}}

\newcommand{\Spec}{{\mathrm{Spec}}}

\newcommand{\sgn}{{\mathrm{sgn}}}

\newcommand{\wt}{\widetilde}

\newcommand{\bs}{\backslash}

\newtheorem{thm}{Theorem}[section]
\newtheorem{cor}[thm]{Corollary}
\newtheorem{lem}[thm]{Lemma}
\newtheorem{prop}[thm]{Proposition}

\newtheorem {ques/conj}[thm]{Question/Conjecture}

\newtheorem{defn}[thm]{Definition}
\newtheorem{rmk}[thm]{Remark}

\usepackage[numbers]{natbib}

\newcommand{\Irr}{{\rm Irr}}

\newcommand{\select}[1]{{\it{#1}}}
\begin{document}
\renewcommand{\theequation}{\arabic{equation}}
\numberwithin{equation}{section}

\begin{comment}

\date{\today}

\author[Fang Shi]{Fang Shi}

\address{School of Mathematical Sciences, Xiamen University, Xiamen 361005, Fujian, P.R. China}

\email{shif@xmu.edu.cn}

\subjclass[2010]{Primary  20C33}

\maketitle

\begin{abstract}
We completely and explicitly classify the irreducible representations of finite general linear groups $\rm{GL}_{2n}(\mathbb{F}_q)$ and finite unitary groups $\rm{U}_{2n}(\mathbb{F}_q)$ that admit a Shalika model. Using Lusztig’s Jordan decomposition of irreducible representations, we establish a multiplicity formula for Shalika multiplicities in terms of the Weyl group of the centralizer of a semisimple element in the dual group. We also derive connections to symplectic models, linear models, and relative Langlands duality.
\end{abstract}

\section{Introduction and Main Results}
~

\end{comment}

\date{\today}

\title[Shalika Models for Finite General Linear and Unitary Groups]{Shalika models for finite general linear and unitary groups}

\author[Fang Shi]{Fang Shi}

\address{School of Mathematical Sciences, Xiamen University, Xiamen 361005, Fujian, P.R. China}

\email{11935007@zju.edu.cn}

 \author{Zhicheng Wang}
 \address{Department of Mathematics, Jilin University, Changchun 130012, Jilin, P. R. China}
 \email{wangzhicheng@jlu.edu.sg}


\maketitle

\begin{abstract}
We completely and explicitly classify the irreducible representations of finite general linear groups ${\rm{GL}}_{2\frak{n}}(\mathbb{F}_q)$ and finite unitary groups ${\rm{U}}_{2\frak{n}}(\mathbb{F}_q)$ that admit a Shalika model. Using Lusztig's Jordan decomposition of irreducible representations, we establish a multiplicity formula for Shalika multiplicities in terms of the Weyl group of the centralizer of a semisimple element in the dual group. We also derive connections to symplectic models, linear models, and relative Langlands duality.
\end{abstract}

\section{Introduction}

Let $\mathbb{F}_q$ be a finite field, and let $\mathrm k$ be an algebraic closure of $\mathbb{F}_q$.
Let $G = \mathrm{GL}_{2\mathfrak{n}}$ be the general linear group over $\mathrm k$, endowed with a Frobenius endomorphism $F$ that is either split or non-split, so that the fixed-point group $G^F$ is either a finite general linear group $\mathrm{GL}_{2\mathfrak{n}}(\mathbb{F}_q)$ or a finite unitary group $\mathrm{U}_{2\mathfrak{n}}(\mathbb{F}_q)$. For any $F$-stable algebraic subgroup $H \subset G$ and irreducible representations $\pi \in \mathrm{Irr}(G^F)$ and $\pi' \in \mathrm{Irr}(H^F)$, we define the multiplicity scalar
\[
\langle \pi, \pi' \rangle_{H^F} := \dim_{\overline{\mathbb{Q}}_\ell} \mathrm{Hom}_{H^F}(\pi, \pi').
\]
By abuse of notation, we denote by $\pi$ the character of the irreducible representation $\pi$ when no confusion arises.

Let $V$ be the standard $2\mathfrak{n}$-dimensional representation of $G$.
Fix a basis $\mathbb{B} = \{v_i\}_{1 \leq i \leq 2\mathfrak{n}}$ of $V$, and set
\[
W = \mathrm{span}\{v_1, \dots, v_{\mathfrak {n}}\}, \quad W' = \mathrm{span}\{v_{\mathfrak{n}+1}, \dots, v_{2\mathfrak{n}}\}.
\]
Let $\mathrm{P} \subset G$ be the parabolic subgroup stabilizing the partial flag $0 \subset W \subset V$, with Levi subgroup $L$ and unipotent radical $N$ (see Definition \ref{def-various} for details).
We denote by $\mathrm S \subset \mathrm{P}$ the stabilizer of the isomorphism
$
i_W : W \stackrel{\sim}{\longrightarrow} V/W, \quad v_i \mapsto v_{i+\mathfrak{n}}.
$
Then the natural identification $V/W \cong W'$ induces an isomorphism $i_{W,W'}:W \longrightarrow W'$. We define a morphism
$\mathrm c : N \longrightarrow \mathbb{G}_a$
by sending $n \in N$ to the trace of the composition
\[
W' \xrightarrow{n - \mathrm{id}_V} W \xrightarrow{i_{W,W'}} W'.
\]
The subgroup $L \cap \mathrm S$ normalizes $N$ and preserves $\mathrm c$, so $\mathrm c$ extends to a morphism
\[\mathfrak{c} : \mathrm S \longrightarrow \mathbb{G}_a\]
that is trivial on $L \cap \mathrm S$.
Fix \emph{once and for all} a nontrivial additive character
$\psi : \mathbb{F}_q \longrightarrow \overline{\mathbb{Q}}_\ell^\times.$
We define a character $\psi_{\mathrm S} \in \mathrm{Irr}(\mathrm{S}^F)$ by
\[
\psi_{\mathrm S}(s) = \psi(\mathfrak{c}(s)).
\]
(Note that $\psi_{\mathrm{S}}$ here coincides with $\psi_{\mathrm{S}}^{(1)}$ 
in Definition~\ref{def-psiS}.)

\begin{defn}
 For $\pi \in \mathrm{Irr}(G^F)$, the Shalika multiplicity is defined as
\[
S(\pi) := \langle \pi, \psi_{\mathrm S} \rangle_{\mathrm {S}^F}.
\]
We say that an irreducible representation $\pi\in {\rm{Irr}}(G^F)$ has a \textit{Shalika model} if
 \[
S(\pi) \geq 1.
\]
\end{defn}

The Shalika model was first introduced in \cite{JS}
in order to characterize the pole at $s=1$ of the partial exterior square $L$-function
$L^S(s,\pi,\Lambda^2)$ attached to an irreducible cuspidal automorphic representation $\pi$
of $\mathrm{GL}_{2\mathfrak{n}}(\mathbb{A})$. For general linear groups over non-archimedean local fields. The multiplicity-freeness of Shalika models was established by Jacquet and Rallis \cite{JR}. In \cite{BW}, Beuzart-Plessis and Wan investigated the generalized Shalika model for ${\rm GL}_2(D)$ where $D$ is a central simple algebra over a local field.

In the finite field setting, Prasad computes a certain twisted Jacquet model for cuspidal representations in \cite{P}. In particular, his result yields the multiplicity formula for cuspidal representations with respect to Shalika models. The multiplicity-freeness of Shalika models was proved by Nien \cite{Nien}. However, the problem of explicitly classifying all representations admitting a Shalika model has remained open until now. In the present paper, we provide a complete and explicit answer to this question for both finite general linear groups and finite unitary groups.

\subsection{Main results}
To state our main results, we first briefly recall Lusztig’s parametrization of the irreducible representations.
Let $G^*$ be the dual group of $G$ in the sense of \cite[Definition 5.21]{DL}.
By abuse of notation, we also denote the Frobenius endomorphism of $G^*$ again by $F$.
Lusztig’s Jordan decomposition (see, e.g., \cite[Section 2]{LS} and \cite[Theorem 5.2]{Gec}) partitions $\Irr(G^F)$ into disjoint Lusztig series, indexed by the semisimple conjugacy classes in $G^{*F}$ (the finite group of $F$-fixed points of $G^*$):
\begin{equation}
\Irr(G^F)=\coprod_{(s)}\mathcal{E}(G,s),
\end{equation}
where $(s)$ runs over the $G^{*F}$-conjugacy classes of semisimple elements in $G^{*F}$,
and $\mathcal{E}(G,s)$ is the \textit{Lusztig series} attached to $s$.

The centralizer $G^*_s:= C_{G^*}(s)$ is a direct product of general linear groups. Let ${\mathbb W}_s$ be the canonical Weyl group of $G^*_s$ in the sense of Deligne-Lusztig \cite[Section 1.1]{DL}, which we briefly review in subsection \ref{subsec-canonical-rep}. 
We denote by ${\mathbb W}_s^\vee$ the set of isomorphism classes of irreducible representations of ${\mathbb W}_s$.
For each $\rho\in ({\mathbb W}_s^\vee)^F$, we denote by $\pi_{s,\rho}\in\mathcal{E}(G,s)$  the corresponding representation, characterized by the identity
\[
(-1)^{\sigma(G)+\sigma(G^*_s)} \pi_{s,\rho}
= |{\mathbb W}_s|^{-1}\sum_{w\in {\mathbb W}_s} {\rm Tr}\bigl(w\mathfrak{F},\widetilde{\rho}\bigr) R_{T_w^*,s},
\]
where:
\begin{itemize}
    \item $\sigma(G)$ denotes the $\mathbb{F}_q$-rank of $G$,
    \item  $\widetilde{\rho}$ is a representation of $\widetilde{{\mathbb W}}_s:={\mathbb W}_s \rtimes \langle \frak{F} \rangle
$ lifting $\rho$ as in Theorem~\ref{thm-G*s-unipotente-rep} (which rephrases \cite[Theorem 2.2]{LS}),
\item  $T^*_w$ is the $F$-stable maximal torus of $G^*_s$ corresponding to $w$ in the sense of \cite[Corollary 1.14]{DL},
\item $R_{T_w^*,s}$ is the Deligne–Lusztig character of $G^F$ associated to the pair $(s,T^*_w)$.
\end{itemize}
   Note that every irreducible representation of $G^F$ arises in this way. We remark here that the extended Weyl group $\widetilde{{\mathbb W}}_s$ records the natural Frobenius action on $\mathbb W_s$ (see Definition \ref{def-ext-caonical-weyl}).

\begin{defn}
We say that a semisimple element $s \in G^{*F}$ is \textit{symplectic} if the multiplicity of each eigenvalue $a$ of $s$ equals that of $a^{-1}$, and the multiplicities of the eigenvalue $\pm1$  are even. 
In other words, there exists a subgroup $H^* \cong \mathrm{Sp}_{2\mathfrak{n}}$ of $G^*$ such that $s\in H^{*F}$. (We fix one such $H^*$ for each symplectic $s\in G^{*F
}$.)
For a symplectic element $s$, set $H_s^*:=C_{H^*}(s)$ (the centralizer), and let ${\mathbb W}_s^{\mathrm{Sp}}$ denote the Weyl group of $H_s^*$ in the sense of Deligne-Lusztig \cite[Section 1.1]{DL} (see Proposition \ref{pro-mul-one-induced} and Remark \ref{618} for details; note that ${\mathbb W}_s^\mathrm{Sp}$ can be identified with the stabilizer group ${\mathrm {St}}_x$ introduced in Proposition \ref{pro-mul-one-induced} up to conjugation).
\end{defn}

\begin{rmk}
We can identify ${\mathbb W}_s^\mathrm{Sp}$ as a subgroup of ${\mathbb W}_s$.
\end{rmk}

The following theorem reformulates Theorem \ref{thm-general-irreducible-main}.
\begin{thm}\label{main1}
Let $\pi_{s,\rho}\in \mathcal{E}(G,s)$. We have
\begin{equation}\label{eq1}
S(\pi_{s,\rho})=\left\{
\begin{array}{ll}
  \left\langle \rho, {\rm sgn}\otimes {\rm{Ind}}_{{\mathbb W}_{s}^{\rm{Sp}}}^{{\mathbb W}_s}{  1}\right\rangle_{{\mathbb W}_s}   & \textrm{ if $s$ is symplectic;} \\
    0 & \textrm{otherwise,} 
\end{array}
\right.
\end{equation}
where $\rm{sgn}$ is the sign character of ${\mathbb W}_s$ and $1$ is the trivial character of ${{\mathbb W}_s^{\rm Sp}}$.
\end{thm}

\begin{rmk}
 By the multiplicity-freeness of Shalika models \cite{Nien}, the right-hand side of \eqref{eq1} is at most $1$.
We also give an independent proof of this result (see Remark \ref{rm-mul-one-induced} and Corollary \ref{shalika-one}).
\end{rmk}

In the local setting, the Shalika model is related to the linear model via the theta correspondence. In particular, the multiplicity-freeness of Shalika models was proved by the fact that linear models are multiplicity-free, and the existence of a Shalika model for a representation of $\mathrm{GL}_{2\mathfrak{n}}$ implies the existence of a linear model (see \cite{JR}).
However, this does not always hold over finite fields.
The linear model in the finite field case has been explicitly described in \cite{H}, and is not multiplicity-free.
The difference arises from the unipotent parts of the representations.
According to \cite[Theorem~3.1.1 and 3.3.1]{H} and Theorem~\ref{main1}, we obtain the following corollary. 
\begin{cor} Assume that ${\mathbb F}_q$ is a finite field of odd characteristic.
    \begin{enumerate}
        \item [(1)] If $\pi\in {\rm{Irr}}(G^F)$ has a Shalika model, then 
        \[
        \left\langle \pi, {\rm Ind}^{G^F}_{G^{\prime F}} 1\right\rangle_{G^F}\ne 0,
        \]
        where 
        \[
        G^{\prime F}\cong 
        \left\{
        \begin{array}{ll}
         {\rm GL}_{\frak{n}}(\mathbb{F}_q)\times  {\rm GL}_{\frak{n}}(\mathbb{F}_q)   &  \textrm{ if $F$ is split;}\\
            {\rm U}_{\frak{n}}(\mathbb{F}_q)\times  {\rm U }_{\frak{n}}(\mathbb{F}_q)   &  \textrm{ if $F$ is non-split.}\\
        \end{array}
        \right.
        \]
        \item [(2)] Assume that $\pi\in \mathcal{E}(G,s)$, where the semisimple element $s$ has no eigenvalue $1$. Then
        \[
        S(\pi)=\left\langle \pi, {\rm Ind}^{G^F}_{G^{\prime F}} 1\right\rangle_{G^F}.
        \]
    \end{enumerate}
\end{cor}
\begin{rmk}
The above corollary relies on the assumption that $\mathbb{F}_q$ has odd characteristic $q$, as \cite{H} does, whereas Theorem \ref{main1} in this paper does not require this assumption.
\end{rmk}

 In \cite{BZSV,SV}, Ben-Zvi, Sakellaridis, and Venkatesh proposed a relative Langlands duality for hyperspherical varieties, which is expressed as:
\[
\mathbf{G} \circlearrowright \mathbf{X} \longleftrightarrow \mathbf{X}^\vee \circlearrowleft \mathbf{G}^\vee,
\]
where:
\begin{itemize}
\item $\mathbf{G}$ is an arbitrary reductive group defined over a number field,
\item $\mathbf{G}^\vee$ is the Langlands dual group of $\mathbf{G}$,
\item $\mathbf{X}$ is a hyperspherical $\mathbf{G}$-variety,
\item $\mathbf{X}^\vee$ is a hyperspherical $\mathbf{G}^\vee$-variety.
\end{itemize}
\begin{comment}
\(\mathbf{G}\) is a arbitrary reductive group defined over a number field, \(\mathbf{G}^\vee\) denotes the Langlands dual group of \(\mathbf{G}\), and \(\mathcal{X}\) and \(\mathcal{X}^\vee\) are hyperspherical \(\mathbf{G}\)-variety and \(\mathbf{G}^\vee\)-variety, respectively.
\end{comment}
In the conjecture of the relative Langlands program, a representation has nonvanishing period integrals on one side if and only if its Arthur parameter can be factored through the other side; furthermore, the period integrals on one side are associated with the L-functions on the other side. Within the framework of the relative Langlands program, the hyperspherical varieties induced by the Shalika model and the restriction from $\mathrm{GL}_{2\frak{n}}$ to $\mathrm{Sp}_{2\frak{n}}$ are mutually dual \cite[Example 4.5.1]{BZSV}.
A cuspidal automorphic representation \(\pi\) of \(\operatorname{GL}_{2\mathfrak{n}}\) admits a Shalika model if and only if the Langlands parameter of \(\pi\) factors through the dual group \({\operatorname{SO}}^\vee_{2\mathfrak{n}+1} = \operatorname{Sp}_{2\mathfrak{n}}(\mathbb{C})\) \cite{CKPSS,GRS,JiS1, JiS2}. This places the classical Shalika characterization as a central instance of the duality for spherical pairs. In the case of finite fields, we can also observe a similar phenomenon as follows.

We say that an irreducible representation $\pi$ has a symplectic model if
\[
Sp(\pi):=\left\langle \pi , {\rm Ind}^{G^F}_{H^F} 1\right\rangle_{G^F}\ne 0,
\]
 where $H\cong \rm{Sp}_{2\frak{n}}$ is a subgroup of $G$. By combining Theorem \ref{main1} with Henderson's result  \cite[Theorem 2.1.1]{H} on symplectic model, and applying the Macdonald correspondence \cite{A,Ma}—a finite-field analog of the local Langlands correspondence for general linear groups—one can observe similar phenomena.
To be more precise, for an irreducible {\it non-unipotent} representation \(\pi\) (i.e., \(\pi\in \mathcal{E}({\rm GL}_{2\frak{n}},s)\) where \(s\) has no eigenvalues \(\pm 1\)), we have the following:
\begin{enumerate}
    \item \(\pi\) has a Shalika model if and only if \(\mathcal{M}_{2\mathfrak{n}}(\pi)\) factors through \({\rm Sp}_{2\frak{n}}(\mathbb{C})\)—that is, the restriction of $\mathcal{M}_{2\mathfrak{n}}(\pi)$ on the inertia subgroup is a symplectic representation—where \(\mathcal{M}_{2\frak{n}}\) is the Macdonald correspondence defined in \cite[Section 2.4]{A};
    \item \(\pi\) has a symplectic model if and only if \(\mathcal{M}_{2\frak{n}}(\pi)\) factors through \( K\times \left( \begin{smallmatrix} 0 & 1_\frak{n} \\ 0 &0 \end{smallmatrix} \right)\), where  \(1_\frak{n}\) is the $\frak{n}\times \frak{n}$ identity matrix,  $K:=\{{\rm diag}(g,g)\ |\ g\in {\rm GL}_{\frak{n}}(\mathbb{C}) \}$. Here ``factors through'' means that
the nilpotent part of the image of the inertia subgroup under $\mathcal{M}_{2\frak{n}}(\pi)$
has the form
$\left(\begin{smallmatrix} g & 1_\frak{n} \\ 0 & g \end{smallmatrix}\right)$ with $g$ nilpotent up to conjugation,
and the semisimple part is contained in $K$.
\end{enumerate}

In \cite{FGT}, Finkelberg, Ginzburg, and Travkin formulate a geometric conjecture realizing the relative Langlands duality as an isomorphism of Borel--Moore homology groups:
\begin{equation}\label{FGT1}
H_{\mathrm{top}}^{\mathrm{BM}}(\mathrm{St}_{\mathbf{X}}, \mathbb{C}) \cong H_{\mathrm{top}}^{\mathrm{BM}}(\mathrm{St}_{\mathbf{X}^\vee}, \mathbb{C})
\end{equation}
in the category of \({\mathbb W}_\mathbf{G} = {\mathbb W}_{\mathbf{G}^\vee}\)-modules, with \({\mathbb W}_\mathbf{G}\) and \({\mathbb W}_{\mathbf{G}^\vee}\) the Weyl groups of \(\mathbf{G}\) and \(\mathbf{G}^\vee\), where $\mathrm{St}_{\mathbf{X}}$ is defined in \cite[Section 1.4.1]{MQYZ} and \cite[Section 1.1]{FGT}.
For spherical varieties (over finite fields), Ma, Qiu, Zou, and Yun relate the Borel–Moore homology groups in \eqref{FGT1} to a \({\mathbb W}_\mathbf{G}\)-model arising from the unipotent principal series part of a certain $\mathbf{G}$-model associated with \(\mathbf{X}\). 

The realization of the relative Langlands duality in \cite{FGT} implies a duality relation between the multiplicity problems. In the Theta-Gan-Gross-Prasad setting over finite fields, a similar multiplicity duality phenomenon has been observed \cite{W}. In the Shalika-symplectic setting, we can also observe a similar phenomenon.

For quadratic unipotent representations --- that is, representations occurring in $\mathcal{E}(G,s)$ where $s$ has eigenvalues $\pm 1$ only --- the duality between the Shalika model and the symplectic model is established in a manner analogous to the Theta-Gan-Gross-Prasad duality \cite[Corollary~1.9]{W}.  
\begin{thm}\label{un-dual}
Assume that $\mathbb{F}_q$ is a finite field of odd characteristic. Let $\pi$ be an irreducible quadratic unipotent representation of $G^F$. Then $\pi$ has a Shalika model if and only if the Alvis--Curtis dual of $\pi$ \cite{Al,C} admits a symplectic model. In particular, for a unipotent representation $\pi=\pi_{1,\rho}$, where $1$ denote the identity element in $G^*$, and $\rho$ is an irreducible representation of the Weyl group ${\mathbb W}={\mathbb W}_1$ of $G^*$, we have
\[
S(\pi_{1,\rho})=Sp( \pi_{1,\sgn\otimes \rho})=\left\langle \rho, {\rm sgn}\otimes {\rm{Ind}}_{{\mathbb W}^{\rm{Sp}}}^{{\mathbb W}}{  1}\right\rangle_{{\mathbb W}},
\]
where ${\mathbb W}^{\rm{Sp}}$ is the group ${\mathbb W}^{\rm{Sp}}_s$ from Theorem \ref{main1} with $s=1$, and it is isomorphic to the Weyl group of $\rm{Sp}_{2\frak{n}}$. 
\end{thm}
\begin{rmk}
The above theorem relies on the assumption that $\mathbb{F}_q$ has odd characteristic $q$, as does the work in \cite[Theorem 2.1.1, Theorem 2.2.1]{H}.
\end{rmk}

%%\begin{rmk}
%%It is possible to state Conjecture \eqref{FGT1} for the Shalika and %%symplectic cases by extending the work of \cite{MQYZ} to the Shalika and %%symplectic settings, combined with Theorem \ref{un-dual}.
%%\end{rmk}

\subsection{Methods and organization of this paper} First, we note that the multiplicity $S(\cdot)$ of Shalika models can be extended linearly to  virtual characters. Hence it makes sense to speak of the multiplicity $S(R_{T,\theta})$  with respect to the Deligne-Lusztig character $R_{T,\theta}$. 

Section \ref{sec-doublecoset}, \ref{sec-geo-formulation}, \ref{sec-computational-approach} are dedicated to computing the multiplicity of Shalika models  with respect to  Deligne-Lusztig characters.

Section \ref{sec-doublecoset} essentially involves linear algebra techniques. In Section \ref{sec-doublecoset}, we study the double coset space associated with the Shalika subgroup $\mathrm S$ and a Borel subgroup of $G$. Geometrically, this reveals certain properties of $B$-stabilizer on the spherical variety $G/\mathrm S$ for a Borel subgroup $B$ of $G$.

We observe that Deligne-Lusztig characters can be informally viewed as characters of ``parabolic induction" from $F$-stable maximal tori. Combining the theory of character sheaves with the Grothendieck trace formula, it is routine to express the multiplicity in question in terms of Frobenius traces on certain cohomology spaces, as presented in Proposition \ref{pro-cohointerpretation-pre}. Using Deligne's theory of weight for $\Qlb$-sheaves, we deduce that the top cohomology space controls the multiplicity of Shalika models for Deligne-Lusztig characters, as shown in Proposition \ref{pro-cohointerpret-main}. Such an interpretation is not surprising: the underlying geometric object (which depends on the choice of a Borel subgroup $B$) of this top cohomology space has a natural partition into equidimensional irreducible subschemes indexed by the set of $B$-orbits on the spherical variety $G/\mathrm S$. Thus, Proposition \ref{pro-cohointerpret-main} can be viewed as a cohomological incarnation of the Mackey formula. The drawback of the cohomology interpretation  Proposition \ref{pro-cohointerpret-main} is that the choice of a Borel subgroup for such a geometric model is never canonical, and we cannot capture the Frobenius action directly. In the language of characters sheaves, the Frobenius action arises from the intermediate extension, which is technically hard to track.

To remedy this, we adopt a computational approach in Section \ref{sec-computational-approach} to trace the Frobenius action. This approach is systematically used by Reeder in \cite{R} to study the restriction problem for $(\mathrm {SO}_{2n+1},\mathrm {SO}_{2n})$. In \cite{LMS} and \cite{Shi}, some variants of this method are used to compute Fourier-Jacobi models and periods with respect to spherical varieties. As Theorem \ref{thm-passing-to-semisimple} shows that the ``leading terms" of the multiplicity are concentrated in certain loci within the set of semisimple elements, we establish the summation formula for such leading terms in Theorem \ref{thm-summation-main}. Although a priori complicated, this formula captures the Frobenius actions explicitly.
Combining the geometric observation from Section \ref{sec-geo-formulation} with the formula in Theorem \ref{thm-summation-main}, we ultimately obtain the multiplicity formula in Theorem \ref{thm-mul-refined} for Shalika models with respect to Deligne-Lusztig characters. Again, we may view this formula as a variant of the Mackey formula, with the additional advantage of eliminating the dependence on the choice of a Borel subgroup containing a given $F$-stable maximal torus.

Finally, in Section \ref{sec-mul-general-irreducible}, we compute the multiplicity of Shalika model for an arbitrary irreducible representation of $G^F$. After some technical preliminaries, our computation reduces to the result from \cite{LS}: every irreducible representation of $G^F$ can be expressed as a linear combination of Deligne-Lusztig representations, as reviewed in Subsection \ref{subsec-review-LS} and \ref{subsec-Lusztigmap}. Our main result is Theorem \ref{thm-general-irreducible-main}.

We remark that our calculation of the Shalika model multiplicities applies simultaneously to unitary groups and general linear groups. From our perspective, the main discrepancy between these two algebraic groups is the geometric Frobenius actions they equip, which we spell out in Section \ref{subsec-intro-Fro}.

\subsection{Convention}

Fix an algebraic closure $\mathrm k$ of $\mathbb F_q$.

Throughout this paper, the algebraic group $G$  is the one introduced in Section \ref{sec-doublecoset}; except for Subsection \ref{subsec-DL} and Subsection \ref{subsec-charactershea}, where $G$ denotes an arbitrary connected reductive group over $\mathrm k$ that is defined over $\mathbb F_q$.

By an algebraic scheme $X$ (over $\mathbb F_q$ or $\mathrm k$), we mean that it is separated and of finite type. 

For a linear operator $a$ on a vector space $V$, let $\mathrm {Tr}(a,V)$ denote the trace.

\subsection*{Acknowledgement} The authors would like to thank the organizers of the Seminar on Harmonic Analysis on Algebraic Groups and the Seminar on the Relative Langlands Program at Tianyuan Mathematics Research Center, Kunming, in August 2025. This paper was inspired by discussions and feedback from these two seminars. The authors also wish to thank Chuijia Wang for his insightful and valuable discussions. 
We acknowledge generous support provided by the National Natural Science Foundation of PR China (No. 12571012) and the Jilin Provincial Postdoctoral Foundation. Fang Shi is supported by the Postdoctoral Fellowship Program of CPSF (No. GZC20252027).

\section{Preliminaries: double coset space and Frobenius endomorphism}\label{sec-doublecoset}We introduce the notation adopted in this section.
We fix an algebraic closure $\mathrm k$ of $\mathbb F_q$. Let $V$ be a $2\mathfrak{n}$-dimensional vector space over $\mathbb F_q$ with basis $\mathbb B=\{v_i\}_{1\leq i\leq 2\mathfrak n}$. Let $G_0:=\mathrm {GL}(V)$, which we view as a reductive group over $\mathbb F_q$. Let $G$ be the pullback of $G_0$ to $\mathrm k$, which we view as a  group over $\mathrm k$ endowed with the geometric Frobenius $F_{\mathrm {GL}}:G\to G$. For a maximal torus $X$ in an algebraic group $Y$, let $W(X,Y):=N_{Y}(X)/C_Y(X)$ denote its Weyl group (where $N_Y(X)$ denotes the normalizer of $X$ in $Y$, and $C_Y(X)$ denotes the centralizer of $X$ in $Y$). We view these Weyl groups as finite abstract groups. % when the corresponding groups are defined over $\mathbb F_q$,  the Frobenius $F$ acts on such Weyl groups naturally.
\subsection{The double coset space $\mathrm B^{op}\backslash G/\mathrm S$}
\begin{defn}\label{def-various}
We define the following (algebraic) subgroup of $G$. (In what follows, we abuse the notation by denoting the scalar extension of $V$ to $\mathrm k$ again by $V$ and view $V$ as an algebraic representation of $G$.)
    \begin{itemize}
        \item $\mathrm B$ the Borel subgroup of $G$ stabilizes the maximal flag $V_0\subset \ldots \subset V_{2\mathfrak n}$, where $V_i$ is the span of $\{v_j\}_{1\leq j\leq i}$.
        \item $\mathrm P$ the parabolic subgroup of $G$ stabilizes the flag
        $0\subset W\subset V$, where $W=V_{\mathfrak{n}}$ is the span of $\{v_i\}_{1\leq i\leq \mathfrak  n}$.
        \item $N$ the unipotent radical of $\mathrm P$.
        \item $\mathrm  S$ the subgroup of $\mathrm P$ stabilizes the isomorphism $i_W:W\xrightarrow[]{\sim} V/W$ given by $v_i\mapsto v_{i+\mathfrak n}$ for $1\leq i\leq \mathfrak n$. (Here we abuse the notation by denoting the image of $v_{i+\mathfrak n}$ in $V/W$ again by $v_{i+\mathfrak n}$.)
        \item $L$ the subgroup of $\mathrm P$ stabilizes the decomposition $V=W\oplus W'$, where $W'$ is the span of $\{v_i\}_{\mathfrak n+1\leq i\leq 2\mathfrak n}$.
        \item $\mathrm T$ the maximal torus of $G$ with the set of  eigenvectors $\mathbb B$.
        \item $\mathrm B^{op}$ the opposite Borel subgroup of $\mathrm B$ with respect to the maximal torus $\mathrm T$.
    \end{itemize}
\end{defn}
Note that $L$ is a Levi subgroup of $\mathrm P$. Let $q_\mathrm P:\mathrm P\to L$ be the natural quotient map. 

Note that we have an identification $i_{W,W'}:W\to W'$ given by the composite $W\xrightarrow[]{i_W}V/W\xrightarrow[]{\sim} W'$, where the later isomorphic arrow is the inverse of the natural isomorphism $W'\xrightarrow[]{\sim} V/W$ induced by the quotient. Coordinatewisely, the map $i_{W,W'}$ is given by $v_i\mapsto v_{i+\mathfrak{n}}$ for $1\leq i\leq \mathfrak{n}$.
This in turn gives rise to an identification $\mathrm s_{W,W'}:\mathrm {GL}(W)\xrightarrow[]{\sim}\mathrm {GL}(W') $. 

The group $L\cap \mathrm S$ consists of pairs $(g,\mathrm s_{W,W'}(g))\in \mathrm {GL}(W)\times \mathrm {GL}(W')\cong L$. 

We have the following well-known lemma.
\begin{lem}\label{lem-BPdoublecoset}
    The double coset space $\mathrm B^{op}(\mathrm k)\backslash G(\mathrm k)/\mathrm P(\mathrm k)$ is in natural bijection with $W(\mathrm T,G)/W(\mathrm T,P)$.
\end{lem}

\begin{cor}\label{cor-interborel}
    Fix any $g\in G(\mathrm k)$. Then $q_\mathrm P(g\mathrm{B}g^{-1}\cap \mathrm P)$ is a Borel subgroup of $L$.
\end{cor}
\begin{proof}
    This follows from Lemma \ref{lem-BPdoublecoset} and a straightforward computation.
\end{proof}
\begin{rmk}\label{rm-classifydoublecoset}
     The double coset space $\mathrm B^{op}\backslash G/\mathrm S$ naturally classifies the equivalence classes of  the following triples $(F,F',\mathrm i)$:
    \begin{itemize}
        \item[(i)] A maximal flag $F:0\subset \mathcal V_1\subset \ldots \subset \mathcal V_{2\mathfrak{n}}=V$;
        \item[(ii)] A flag $F':0\subset \mathcal W\subset V$ with $\dim W=\mathfrak n$;
        \item[(iii)] An isomorphism $\mathrm i:\mathcal W\xrightarrow[]{\sim}V/\mathcal W$. 
    \end{itemize}
    (Here, the data corresponding to $\bar e$  is explicitly given by Definition \ref{def-bijBS} with respect to the identity element in $W(\mathrm T,G)$, where $\bar e$ is the double coset $\mathrm B^{op}\mathrm S$.)
\end{rmk}

\begin{defn}\label{def-torusU}
    Let $U$ be the subtorus of $\mathrm T$ such that:
\begin{itemize}
    \item It is isomorphic to $\mathbb G_m^{\mathfrak n}$ (the $\mathfrak n$-fold copy of $\mathbb G_m$), of which   we denote by $s=(s_1,\ldots,s_{\mathfrak n})$ a typical element;
    \item %$s\cdot v_i=s_{i} v_i$, if $i\leq \mathfrak n$; and $s\cdot v_j=s_{j-\mathfrak n} v_j$ if $\mathfrak n+1\leq j\leq 2\mathfrak n$.
    The action on the basis vectors is given by:
    $$s \cdot v_i = s_i v_i \text{ for } 1 \leq i \leq \mathfrak{n},$$
    $$s \cdot v_j = s_{j-n} v_j \text{ for } \mathfrak{n}+1 \leq j \leq 2\mathfrak{n}.$$
\end{itemize}
\end{defn}

Note that $U$ is a maximal torus of $L\cap \mathrm S$. And we may embed $W(U,L\cap \mathrm S)$ into $W(\mathrm T,G)$ in a natural way. (We may identify $W(\mathrm T,G)$ with the symmetric group $S_{2\mathfrak n}$. Under this identification, the group $W(U, L \cap \mathrm{S})$ corresponds to the subgroup $S_{\mathfrak{n}}$ embedded diagonally in $S_{2\mathfrak{n}}$.) %then $W(U,L\cap \mathrm S)$ is identified with the group $S_{\mathfrak n}$ embedded in a diagonal way

\begin{defn}\label{def-bijBS}
    For an element $w\in W(\mathrm T,G)/ W(U,L\cap \mathrm S)$, we construct a triple $\mathcal T_w=(F,F',\mathrm i)$ as follows:
    \begin{itemize}
        \item $F$ is the maximal flag  stabilized by $\mathrm B^{op}$;
        \item $F'$ is the flag  $0\subset \wt w \cdot W\subset V$;
        \item $\mathrm i$ is the isomorphism given by $\wt w\cdot v_i \mapsto \wt w\cdot v_{i+\mathfrak n}$ for $1\leq i\leq \mathfrak n$. 
    \end{itemize}Here, we choose a representative $\wt w\in G(\mathrm k)$ of $w$; we abuse the notation by denoting the image of $\wt  w\cdot v_{i+\mathfrak n}$ in $V/(\wt w\cdot W)$ again by $\wt w\cdot v_{i+\mathfrak n}$.
    (We easily verify that $\mathcal T_{ws}=\mathcal T_{w}$ for $w\in W(\mathrm T,G)$ and $s\in W(U,L\cap \mathrm S)$.)
\end{defn}

\begin{lem}\label{lem-BSdoublecoset}
 The assignment $w\mapsto \mathcal T_w$ exhibits a bijection between $W(\mathrm T,G)/W(U,L\cap \mathrm S)$ and  $\mathrm B^{op}(\mathrm k)\backslash G(\mathrm k)/\mathrm S(\mathrm k)$. 
\end{lem}
\begin{proof}
    By Bruhat decomposition, we have
    $$
    G(\mathrm k)=\bigcup_{w\in W(\mathrm T,G)/W(\mathrm T,\mathrm {P})}\mathrm B^{op}(\mathrm k) w \mathrm P(\mathrm k).
    $$

Fix $w\in W(\mathrm T,G)$ in the rest of this proof. By the above decomposition,
a typical element in $\mathrm B^{op}(\mathrm k) w \mathrm P(\mathrm k)$ can be written as $bwln$ for $b\in \mathrm B$, $l\in L$ and $n\in N$. 
We note that $L\cong \mathrm {GL}(W)\times \mathrm {GL}(W')$.
Also, a typical elements in $\mathrm B^{op}(\mathrm k) w \mathrm P(\mathrm k)$ can be written as $b'wl_Ws$, for $b'\in \mathrm B^{op}$, $l_W\in \mathrm {GL}(W)$ and $s\in \mathrm S$.

Let $B_w:=q_\mathrm P(w^{-1}\mathrm B^{op}w\cap \mathrm P)$ be the Borel subgroup of $L$ introduced in Corollary \ref{cor-interborel}. We have the decomposition $B_w=B_1\times B_2$, for $B_1$ ($B_2$, resp.) a Borel subgroup of $\mathrm {GL}(W)$ ($\mathrm {GL}(W')$, resp.). Moreover, we see that $T_1$ (resp., $T_2$) is a maximal torus of $B_1$ (resp., $B_2$), where $T_1$ (resp., $T_2$) is the subtorus of $B_1$ (resp., $B_2$) possessing $\{v_i\}_{1\leq i\leq n}$ (resp., $\{v_i\}_{\mathfrak n+1\leq i\leq 2\mathfrak n}$) as eigenvectors. Note that we have $\mathrm T=T_1\times T_2$ and $\mathrm s_{W,W'}(T_1)=T_2$.

Suppose that $l_1\in B_1l_2\mathrm s_{W,W'}^{-1}(B_2)$ for $l_1,l_2\in \mathrm {GL}(W)$.
%(We fix $l_1$ and $l_2$ in this paragraph.)
We see that $\mathrm B^{op}(k)wl_1\mathrm S(\mathrm k)=\mathrm B^{op}(k)wl_2\mathrm S(\mathrm k)$ by the former paragraph. Hence by Bruhat decomposition of $\mathrm{GL}(W)$ we have $\mathrm B^{op}(\mathrm k)wl_1\mathrm S(\mathrm k)=\mathrm B^{op}(\mathrm k)ww'\mathrm S(\mathrm k)$ for some $w'\in W(T_1,\mathrm {GL}(W))$. 
This also shows that the assignment we formulate yields a surjection from $W(\mathrm T,G)/W(U,L)$ onto $\mathrm B^{op}(\mathrm k)\backslash G(\mathrm k)/\mathrm S(\mathrm k)$.

Before preceding, we introduce some notation. Let $\mathrm P^{op}$ be the parabolic subgroup of $G$ stabilizes the flag $0\subset W'\subset V$. Let $N^{op}$ be the unipotent radical of $\mathrm P^{op}$.

Suppose that we have $\mathrm B^{op}(\mathrm k) w \mathrm S(\mathrm k)=\mathrm B^{op}(\mathrm k) w_2 \mathrm S(\mathrm k)$, for some $w_2\in W(\mathrm T,G)$. We fix $w_2$ in the rest of this proof. It remains to show $w_2\in wW(U,L\cap\mathrm S)$.

%By Bruhat decomposition, each element $g\in \mathrm B^{op}(\mathrm k) w \mathrm S(\mathrm k)$ can be uniquely written as $g=u_{\mathrm B^{op}}mn$ for $m\in w B_w(L\cap \mathrm S)$, $n\in N$ and $u_{\mathrm B^{op}}\in \mathrm B^{op}\cap w N^{op} w^{-1}$. 
  Let $$K:=\{x\in \mathrm B^{op}(\mathrm k) w\mathrm S(\mathrm k):\text{$x$ stabilizes $\mathrm T$ under conjugation}\}.$$  
Fix $k\in K$ in this paragraph. Then we have $k=u_km_k n_k$, for some $m_k\in w B_w(L\cap \mathrm S)$, $n_k\in N$ and $u_k\in \mathrm B^{op}\cap w N^{op} w^{-1}$. Note that such a decomposition is unique. By Bruhat decomposition, each element $g\in \mathrm B^{op}(\mathrm k)w \mathrm P$ can be uniquely written as $g=u_g m_g n_g$ for some $m_g\in L$, $n_g\in N$ and $u_g\in \mathrm B^{op}\cap w N^{op} w^{-1}$.
By the definition of $K$, for each $t\in \mathrm T$, there exists $t'\in \mathrm T$ such that $tk=kt'$, indicating $tu_kt^{-1} tm_k n_k=u_k m_k t' (t')^{-1}n_kt'$. Note that $tm_k,m_kt'\in L$.  We see (by the uniqueness of the decomposition) that $tu_kt^{-1}=u_k$, $tm_k=m_kt'$ and $(t')^{-1}n_kt'=n_k$. 
Hence $u_k$ is normalized by $\mathrm T$, manifesting $u_k=e$. A similar argument shows that $n_k=e$.  Therefore, we have $K\subset wB_w(L\cap \mathrm S)$.

Note that $N_{L\cap S}(U)=L\cap\mathrm S\cap N_G(\mathrm T)$. Using the decomposition $L\cong \mathrm {GL}(W)\times \mathrm {GL}(W')$, we easily verify that $\left(wB_w(L\cap \mathrm S)\right)\cap N_G(\mathrm T)=w\mathrm TN_{L\cap \mathrm S}(U) $.
 Consequently, the image in $W(\mathrm T,G)$ of the set $K$  is precisely the coset $wW(U,L\cap \mathrm S)$, yielding $w_2\in wW(U,L\cap\mathrm S)$. %This shows that our assignment induces an injection from  $W(T,G)/W(U,L\cap \mathrm S)$ to $\mathrm B^{op}(\mathrm k)\backslash G(\mathrm k)/\mathrm S(\mathrm k)$.

This completes the proof.
\end{proof}

\begin{rmk}
    A immediate consequence of Lemma \ref{lem-BSdoublecoset} is that $G/\mathrm S$ is a spherical variety with respect to left-$G$-action.
\end{rmk}

Let $\mathrm c: N\to \mathbb G_a$ denote the morphism sending $n\in N$ (viewed as an endomorphism of $V$) to the trace of the composite $W'\xrightarrow[]{n-\mathrm{id}_{V}}W\xrightarrow[]{\sim}W'$, where the latter isomorphic arrow is $i_{W,W'}$.
Note that $L\cap \mathrm S$ normalize $N$ and fixes the morphism $\mathrm c$. We see that $\mathrm c$ extends to a morphism $\mathfrak c :\mathrm S\to \mathbb G_a$ such that the restriction to $L\cap \mathrm S$ is trivial.

Let $\mathrm {St}(U)$ denote
$$
\{w\in W(\mathrm T,G):w\cdot U=U\}.
$$

Notice that $\mathrm {St}(U)$ is isomorphic to the Weyl group of $\mathrm {Sp}_{2\mathfrak n}$. Informally, we have $W(U,L\cap \mathrm S)\ltimes (S_2)^{\mathfrak n}\cong \mathrm {St}(U)$, where $S_2$ denotes the symmetric group of rank $2$.

\begin{prop}\label{pro-bijec-phi}
  We have a  quotient map $ \mathrm B^{op}(\mathrm k)\backslash G(\mathrm k)/\mathrm S(\mathrm k)\cong W(\mathrm T,G)/W(U,L\cap \mathrm S)\twoheadrightarrow W(\mathrm T,G)/\mathrm {St}(U) $   restricting to a bijection between 
  $$\Phi:=\{v\in \mathrm B^{op}(\mathrm k)\backslash G(\mathrm k)/\mathrm S(\mathrm k) : \text{ $v^{-1}_0\mathrm B^{op}v_0\cap N$ has trivial image under $\mathrm c$} \}$$
  and $W(\mathrm T,G)/\mathrm {St}(U)$. Moreover, for $w\in W(\mathrm T,G)$, the group $w Uw^{-1}$ is a maximal torus of $\mathrm B^{op}\cap w \mathrm S w^{-1}$; for  $v\in \Phi$,  the centralizer of $v_0 Uv_0^{-1}$ in $\mathrm B^{op}\cap v_0 \mathrm S v_0^{-1}$ is $v_0 Uv_0^{-1}$ itself.
  (In this proposition, for $v\in \mathrm B^{op}(\mathrm k)\backslash G(\mathrm k)/\mathrm S(\mathrm k)$, we denote some/any representative of $v$ by $v_0$.)
\end{prop}
\begin{proof}The displayed quotient map is induced by Lemma \ref{lem-BSdoublecoset}.
   The rest is due to elementary linear algebra, as we now elaborate.

   We may identify $W(\mathrm T,G)$ with the permutation group $S_{2\mathfrak n}$ acting on $\mathbb B$. 

   Fix $w\in W(\mathrm T,G)$ in this paragraph. Under the identification in the former paragraph, we may write $w\cdot v_{i}=v_{w(i)}$ for $1\leq i\leq 2\mathfrak n$.
   We see that $w^{-1}\mathrm B^{op}w$ is the group stabilizes the flag $w^{-1} F_{\mathrm B^{op}}$, where $F_{\mathrm B^{op}}$ is the maximal flag stabilized by $\mathrm B^{op}$. For $1\leq i \leq \mathfrak n$, we define $\mathbb G_i$ to be the algebraic subgroup of $G$, characterized by the following:\begin{itemize}
       \item[(1)] It fixes $v_j$ for  $j\neq i+\mathfrak n$;
       \item[(2)] It maps $v_{i+\mathfrak n}$ to $v_{i+\mathfrak n}+xv_i$, where $x$ varies in $\mathbb G_a$.\;
       \item[(3)] The assignment $\mathbb G_i\to \mathbb G_a$ that maps the variable $x$ in (2) to $x\in \mathbb G_a$ is an isomorphism. 
   \end{itemize}   It is straightforward to see:\begin{itemize}
       \item  $w^{-1}\mathrm B^{op}w\cap N$ has trivial image under $\mathrm c$ if and only if $w\mathbb G_iw^{-1}$ is not contained in $\mathrm B^{op}$ for every $1\leq i\leq \mathfrak n$.
   \end{itemize}
   Or equivalently:\begin{itemize}
       \item  $w^{-1}\mathrm B^{op}w\cap N$ has trivial image under $\mathrm c$ if and only if $w(i)\leq w(i+\mathfrak n)$ for every $1\leq i\leq \mathfrak n$.
   \end{itemize}

  Hence,  the quotient map that we construct restricts to a bijection between $\Phi$ and $W(\mathrm T,G)/\mathrm {St}(U)$. Fix $w\in W(\mathrm T,G)$ in the rest of this paragraph. Since $w Uw^{-1}$ is a maximal torus of $w \mathrm S w^{-1}$, it is likewise a maximal torus of $\mathrm B^{op}\cap w \mathrm S w^{-1}$. The multiplication exhibits an isomorphism from  $w U w^{-1}\times \prod_{1\leq i\leq \mathfrak n} w\mathbb G_iw^{-1}$ to the centralizer $C_{w\mathrm Sw^{-1}}(w U w^{-1})$. If we further assume that $w\in \Phi$, we see immediately from this isomorphism that $C_{\mathrm B^{op}\cap w\mathrm Sw^{-1}}(w U w^{-1})=w U w^{-1}$. These prove the last sentence of this proposition by obvious symmetry.

   This completes the proof.
\end{proof}

\begin{rmk}\label{rm-phi-selfcen}For $w\in \mathrm B^{op}(\mathrm k)\backslash G(\mathrm k)/\mathrm S(\mathrm k)$, let $w_0\in W(\mathrm T,G)/W(U,L\cap \mathrm S)$ be the image of $w$ under the bijective map $\mathrm B^{op}(\mathrm k)\backslash G(\mathrm k)/\mathrm S(\mathrm k)\xrightarrow[]{\sim} W(\mathrm T,G)/W(U,L\cap \mathrm S)$ introduced in Proposition \ref{pro-bijec-phi}. Note that $U$ is a maximal torus of $w_0^{-1}\mathrm B^{op}w_0\cap \mathrm S$.
    By the proof of the above proposition, we may rephrase $\Phi$ as 
    $$
  \{  w\in \mathrm B^{op}(\mathrm k)\backslash G(\mathrm k)/\mathrm S(\mathrm k):C_{w_0^{-1}\mathrm B^{op}w_0\cap \mathrm S}(U)=U\} ,
    $$
    where  $C_{w_0^{-1}\mathrm B^{op}w_0\cap \mathrm S}(U)$ denotes the corresponding centralizer.
    Unraveling the definition, we may also rephrase $\Phi$ as
    $$
     \{  w\in \mathrm B^{op}(\mathrm k)\backslash G(\mathrm k)/\mathrm S(\mathrm k):\text{ $w_0^{-1}\mathrm B^{op}w_0\cap \mathrm S$ has trivial image under $\mathfrak c$}\}. 
    $$
    Fix $w\in \Phi$, we also verify that $C_{w_0^{-1}\mathrm B^{op}w_0\cap \mathrm S}(s)=U$ for a general point $s\in U(\mathrm k)$. Consequently, we have
    $$
    \Phi=\{w\in \mathrm B^{op}(\mathrm k)\backslash G(\mathrm k)/\mathrm S(\mathrm k):\text{ a general element in $w_0^{-1}\mathrm B^{op}w_0\cap \mathrm S$ is semisimple}\}.
    $$
\end{rmk}

\begin{rmk}\label{rm-ker-connected-unipo}
Fix $w\in \mathrm B^{op}(\mathrm k)\backslash G(\mathrm k)/\mathrm S(\mathrm k)$.
Let $$\mathrm{c}_w:w^{-1}\mathrm B^{op}w\cap N\to \mathbb G_a$$
be the restriction of $\mathrm c$.
Then the kernel of $\mathrm c_w$ is a connected unipotent group.
This claim follows from the classification of double coset space Lemma \ref{lem-BSdoublecoset} and a straightforward computation.
\end{rmk}

\subsection{Frobenius endomorphisms on $G$}\label{subsec-intro-Fro}
We introduce another Frobenius endomorphism $F_{\mathrm U}$ on $G$, which arises from an outer form of $G$ defined over $\mathbb F_q$. 

We need some notations. Set
\[\mathrm {E_{2\mathfrak{n}}}:= \begin{pmatrix} &1_{\mathfrak n}\\-1_{\mathfrak n}&\end{pmatrix},\]
where $1_{\mathfrak{n}}$ denotes the identity matrix of rank $\mathfrak{n}$. 

Using the basis $\mathbb B=\{v_i\}_{1\leq i \leq 2\mathfrak{n}}$ of $V$, we identify $G$ with $\mathrm{GL}_{2\mathfrak n}$ over $\mathrm k$, which can be identified with the variety of invertible matrices of rank $2\mathfrak{n}$. In the remainder of this subsection, we fix this natural identification. 

Under this identification, the Frobenius $F_{\mathrm {GL}}$ sends $(a_{ij})_{1\leq i,j \leq 2\mathfrak{n}}$ to $(a^q_{ij})_{1\leq i, j\leq 2\mathfrak{n}}$. Let $\mathfrak{I}:G\to G$ be the involution given by
$$
g\mapsto \mathrm {E}_{2\mathfrak{n}}(^tg)^{-1} \mathrm{E}_{2\mathfrak{n}}^{-1},
$$
where we identify $g$ with the matrix representing it, and $^tg$ denotes the transpose of $g$.

\begin{defn}\label{def-Fro-Unitary}Let the notation be as above.
    We define $F_{\mathrm U}=\mathfrak{I}\circ F_{\mathrm {GL}}$ to be the composition.
\end{defn}

 \begin{rmk}
     The algebraic group $G$ endowed with the Frobenius $F_{\mathrm U}$ is isomorphic to the  unitary group of rank $2\mathfrak{n}$. Note that our construction does not exclude the case that the characteristic of $\mathbb F_q$ is $2$.
 \end{rmk}

\begin{rmk}
    It is straightforward to see that the algebraic subgroups $\mathrm {P},~N,~\mathrm{S},~L,~\mathrm T$ of $G$ defined in Definition \ref{def-various} and the group $U$ defined in Definition \ref{def-torusU} are both $F_{\mathrm{GL}}$-stable and $F_{\mathrm U}$-stable. Moreover, the morphism $\mathfrak{c}:\mathrm S\to \mathbb G_a$ is defined over $\mathbb F_q$ with respect to both Frobenius structures.
\end{rmk}

\section{A Geometrical Formulation}\label{sec-geo-formulation}
In this section, we reformulate the Shalika periods of Deligne-Lusztig characters in geometrical terms. Such an interpretation identifies the involved multiplicities  with the trace of Frobenius on specific cohomology spaces, while leaving the action of the Frobenius mysterious. We will overcome this problem in the next section, using a computational approach. On a first reading, the reader is advised to read Definition \ref{def-mathfrakS}, Theorem \ref{thm-passing-to-semisimple} and Remark \ref{rm-passing-to-semisimple}, then skip this section.

We introduce the notation adopted throughout this section.
For a scheme $X_0$ over $\mathbb F_q$, denote by $X$  its pullback to $\mathrm k$. For Weil sheaves $\mathscr L$ on $X$, let $K^n_{\mathscr L}$ denote the function $X^{F^n}\to \Qlb$ given by ``trace of Frobenius" (also known as the sheaf-function correspondence in Grothendieck's sense) for positive integers $n$. More generally, for a complex $\mathscr K$ of Weil sheaves on $X$, let $K^n_{\mathscr K}$ denote the function $\sum_{i\in \mathbb Z}(-1)^iK^n_{\mathscr H^{i}(\mathscr K)}$. For a morphism $f:X\to Y$ of schemes, all functors $f^!,f_!,f_*,f^*$ are understood in the derived sense.

Throughout this section, the reductive group $G$ is as introduced in Section \ref{sec-doublecoset}, and $F$ denotes either $F_{\mathrm {GL}}$ or $F_{\mathrm U}$ introduced in Subsection \ref{subsec-intro-Fro}; except in Subsection \ref{subsec-DL} and \ref{subsec-charactershea}, where $G$ refers to an arbitrary connected reductive group over $\mathrm k$ that is defined over $\mathbb F_q$.

\subsection{Deligne-Lusztig characters}\label{subsec-DL}
In this subsection, $G_0$ is an arbitrary  reductive group over $\mathbb F_q$. Let $G$ be the pullback of $G_0$ to $\mathrm k$. 
For an $F$-stable maximal torus $T$ of $G$ and a character $\chi:T^F\to \Qlb^\times$, the virtual character $R_{T}^{\chi}$ of $G^F$ is defined in \cite{DL}.

\begin{rmk}
  Fix an $F$-stable maximal torus $T$ of $G$ and a character $\chi:T^F\to \Qlb^\times$.  When we say that $R_T^\chi$ is a Deligne-Lusztig character of $G^F$, we implicitly specify the reductive group $G_0$ over the finite field $\mathbb F_q$ and the field extension $\mathbb F_q\subset \mathrm k$. Hence we also say that $R_T^\chi$ is a Deligne-Lusztig character with respect to the pair $(\mathbb F_q\subset \mathrm k,G_0)$.
\end{rmk}

For a positive integer $n$,
    we say that a virtual character $\rho$ of $G^{F^n}$ is a Deligne-Lusztig character if it is a Deligne-Lusztig character with respect to the pair $(\mathbb F_{q^n}\subset \mathrm k,G_0\times_{\mathbb F_q}\mathbb F_{q^n} )$. (To ease the notation, we also say that $\rho$ is a Deligne-Lusztig character of $G^{F^n}$.)

  Fix a positive integer $n$ in this paragraph.  For an $F^n$-stable maximal torus $T$ of $G$  and a character $\chi:T^{F^n}\to \Qlb^\times$, we denote the corresponding Deligne-Lusztig character of $G^{F^n}$ by $R_{T,\chi}^{(n)}$. If $n=1$, we also denote $R_{T,\chi}:=R_{T,\chi}^{(1)}=R_T^\chi$ for convenience.

\subsection{Sheaves on tori}\label{subsec-sheavesontori}In this subsection, $T_0$ is a torus over $\mathbb F_q$. Let $T$ be the pullback of $T_0$ to $\mathrm k$.
    We have a homomorphism $\mathrm l^T:T\to T$ defined by $t\mapsto F(t)\cdot t^{-1}$. 
    We verify that $\mathrm l^T$ is  finite \'etale and compatible with the geometric Frobenius endomorphism $F$. We have the decomposition of Weil sheaves $$
   \mathrm l^T_* \Qlb = \mathrm l^T_! \Qlb \cong \bigoplus_{\eta\in (T^F)^\vee} \mathscr L_\eta^T,
    $$
    where $\mathscr L_\eta^T$ is the $1$-rank local system on $T$ satisfying the following property: (We also denote $\mathscr L_\eta=\mathscr L_\eta^T$ if it causes no confusion.)
    \begin{itemize}
        \item For any positive integer $n$, the corresponding function $K_{\mathscr L _\eta}^n: T^{F^n}\to \Qlb$  is explicitly given by the assignment
        $$
        T^{F^n}\ni t\mapsto \eta \circ \mathrm N_T^n(t),
        $$
        where $\mathrm N_T^n:T^{F^n}\to T^F$ is the norm map given by $T^{F^n}\ni t\mapsto t\cdot F(t) \cdot \ldots \cdot F^{n-1}(t)$.
    \end{itemize}
\begin{comment}
 Note that each $\mathscr L_\eta$ is indeed a local system introduced in \cite[1.3 and 1.6]{L5} endowed with a ``natural normalization'' in the sense of \select{loc. cit.}. Consequently, we see that $\{\mathscr L_\eta:\eta\in (T^F)^\vee\}$ gives a set of representatives of $\mathfrak L(T)^F$ (introduced in Remark \ref{rm-kummersheaf}). (We have $(\mathscr L_\eta)^{\otimes |T^F|}\cong \Qlb$, and the integer $|T^F|$ is coprime to $p$, the radical of $q$.)
\end{comment}
\subsection{character sheaves}\label{subsec-charactershea}In this subsection, $G_0$ is an arbitrary connected reductive group over $\mathbb F_q$. Let $G$ be the pullback of $G_0$ to $\mathrm k$.
For a Borel pair $T\subset B$ of $G$ with $F$-stable $T$, set $\mathrm d^B_T:B\to T$ to be the map witnessing $T$ as the reductive quotient and providing a section for the inclusion $T\subset B$.
We introduce the following commutative diagram and fix the notation.
\begin{equation}\label{diag-charshf}
\xymatrix{
&T_{reg}\ar[d]_{j_T}&\wt{G}_{rss}\ar[l]_{\rho_{rss}}\ar[d]^{\wt j}\ar[r]^{\pi_{rss}}&G_{rss}\ar[d]_j\\
&T &\wt{G}\ar[l]_{\rho}\ar[r]^{\pi}&G
}    
\end{equation}
where \begin{itemize}
    \item $T_{reg}:=\{t\in T: (C_G(t))^\circ=T\}$;
    \item $G_{rss}=\bigcup_{h\in G}h T_{reg}h^{-1}$;
    \item $\wt G_{rss}:=\{(g,hT)\in G_{rss}\times G/T:h^{-1}gh\in T_{reg}\}$;
    \item $\wt G:=\{(g,hB)\in G\times G/B: h^{-1}gh\in B\}$;
    \item $\rho_{rss}$ sends $(g,hT)\in \wt G_{rss}$ to $h^{-1}gh$;
    \item $\pi_{rss}$ sends $(g, hT)\in \wt G_{rss}$ to $g$;
    \item $\rho$ sends $(g,hB)\in \wt G$ to $\mathrm d^B_T(hgh^{-1})$;
    \item $\pi$ sends $(g,hB)\in \wt G$ to $g$;
    \item $\wt j$ sends $(g,hT)\in \wt G_{rss}$ to $(g,hB)\in \wt G$;
    \item $j_T$ and $j$ are natural inclusions.
    \end{itemize}

\begin{rmk}
     For convenience, we also denote $\wt G$ by $\wt G_{T,B}$ to emphasize the Borel pair $T\subset B$.
\end{rmk}

\begin{defn}\label{def-scRTchi}
    For a character $\eta:T^F\to \Qlb^\times$, recall the sheaf $\mathscr L_\eta$ introduced in Subsection \ref{subsec-sheavesontori}. We set $$\mathscr R_{T,\eta}:= \left(j_{!*}(\pi_{rss})_!\rho^*_{rss}j_T^* \mathscr L_\eta [\dim G]\right)[-\dim G],$$ where $j_{!*}$ is the intermediate extension along $j$. 
\end{defn}
We note that the sheaf $\mathscr L_\eta$ and the upper row of the Diagram (\ref{diag-charshf}) is indeed defined over $\mathbb F_q$. Hence $\mathscr R_{T,\eta}$ has a Weil structure $\phi_{T,\eta}: F^* \mathscr R_{T,\eta}\xrightarrow{\sim} \mathscr R_{T,\eta}$ arising from the intermediate extension $j_{!*}$.

\begin{prop}\label{prop-charactersheaf-main}
    The following statements hold:
    \begin{itemize}
        \item [(i)] Fix a positive integer $n$.  The function $K^n_{\mathscr R_{T,\eta}}$ corresponding to $\mathscr R_{T,\eta}$ is the Deligne-Lusztig character $R_{T,\eta\circ \mathrm N_T^n}^{(n)}:G^{F^n}\to \Qlb$ (see Subsection \ref{subsec-DL} for the notation), where $\mathrm N_T^n:T^{F^n}\to T^F$ is the norm map;
        \item [(ii)] There is a canonical isomorphism $\mathscr R_{T,\eta}\cong \pi_! \rho^* \mathscr L_\eta$ of $\Qlb$-sheaves.
    \end{itemize}
\end{prop}
\begin{proof}
    (ii) is \cite[Th\'eor\`eme 1.2.2]{Lau}. Note that the complex $\mathscr R_{T,\eta}$ is the $[-\dim G]$ shift of  the corresponding perverse sheaf in  \select{loc. cit.}.

    (i) follows from \cite[(1.7.1) and (1.9.1)]{Sho} and \cite[Theorem 5.5]{Sho2}, which extend \cite[Theorem 1.14]{L5} in specific cases.
\end{proof}

\begin{rmk}\label{rm-FstableBchar}
    Suppose  that both $T$ and $B$ are $F$-stable, then the whole diagram (\ref{diag-charshf}) is defined over $\mathbb F_q$. Consequently, for characters $\eta:T^F\to \Qlb^\times$, we have a Weil structure $\phi^\circ_{T,\eta}:F^* \left(  \pi_!\rho^* \mathscr L_\eta   \right)\xrightarrow{\sim} \pi_! \rho^* \mathscr L_\eta$ given by Grothendieck's construction. It is elementary to see that $\phi_{T,\eta}^\circ$ coincide with $\phi_{T,\eta}$ (under the identification (ii) of Proposition \ref{prop-charactersheaf-main}) by the formal property of the intermediate extension.
\end{rmk}

\subsection{Sheaves encoding inductions}\label{subsec-shasheaf}In this subsection, we retain the notation in Section \ref{sec-doublecoset}.
 Recall the subgroup $\mathrm S$ of $G=\mathrm {GL}(V)$ and the morphism $\mathfrak c:\mathrm S\to \mathbb G_a$. We fix \textbf{once for all} a nontrivial character $\psi:\mathbb F_q\to \Qlb^\times$. Let $\mathscr L_\psi$ be the Artin–Schreier sheaf on $\mathbb G_a$ such that $K_{\mathscr L_\psi}^1=\psi$.

\begin{rmk}\label{rm-ssheafvalue}
Fix a positive integer $n$.
    For $x\in \mathbb G_a^{F^n}\cong\mathbb F_{q^n}$, we have 
    $$
    K^n_{\mathscr L_\psi}(x)=\psi\circ \mathrm N_{\mathbb G_a,n}(x),
    $$
    where $\mathrm {N}_{\mathbb G_a,n}:\mathbb G_a^{F^n}\to \mathbb G_a^F$ is given by $a\mapsto a+F(a)+\ldots+F^{n-1}(a)$.
\end{rmk}

We introduce the following diagram and fix the notation.
$$
\xymatrix{
G&\mathcal S\ar[l]_{\mathrm p}\ar[r]^{\mathrm q}&\mathbb G_a 
}
$$
where
\begin{itemize}
    \item $\mathcal S$ classifies pairs $(g,x\mathrm S)\in G\times G/\mathrm S$ such that $x^{-1}gx\in \mathrm S$;
    \item $\mathrm p$ sends $(g,x\mathrm S)\in \mathcal S$ to $g$;
    \item $\mathrm q$ sends $(g,x\mathrm S)\in \mathcal S$ to $\mathfrak c(x^{-1}gx)$. (We verify that $\mathrm q$ is well-defined.)
\end{itemize}

\begin{defn}\label{def-shasheaf}
    We set $\mathscr S:=\mathrm p_!\mathrm q^* \mathscr L_\psi$. Note that $\mathscr S$ has been endowed with a natural Weil structure.
\end{defn}

\begin{rmk}\label{rm-shasheafvalue}
    Fix a positive integer $n$. Let $g\in G^{F^n}$. 
   Let $\psi^{(n)}:\mathbb G_a^{F^n}\to \Qlb^\times$ be the composite  $\psi\circ \mathrm N_{\mathbb G_a,n}$ (see Remark \ref{rm-ssheafvalue} for the notation).
    By Grothendieck trace formula and Lang's theorem (note that $\mathrm S$ is a connected algebraic group), we see that 
    $$
    K_{\mathscr S}^n(g)=\mathrm {Ind}_{\mathrm S^{F^n}}^{G^{F^n}}(\psi^{(n)}\circ \mathfrak{c})(g)
    $$
    for $g\in G^{F^n}$,
    where $\mathrm {Ind}_{\mathrm S^{F^n}}^{G^{F^n}}(\psi^{(n)}\circ \mathfrak{c})$ denotes the character of the corresponding induced representation.
\end{rmk}

\subsection{Multiplicities and top cohomology spaces}\label{subsec-mul-top}In this subsection, we retain the notation in Subsection \ref{subsec-shasheaf}. The main result in this subsection is Proposition \ref{pro-cohointerpret-main}. Roughly speaking, Proposition \ref{pro-cohointerpret-main} interprets the multiplicities of Shalika models into the traces of Frobenius on some cohomology spaces.

\begin{defn}\label{def-psiS}
    For a positive integer $n$, let $\psi^{(n)}_{\mathrm S}:\mathrm S^{F^n}\to \Qlb^\times$ be the character given by
    $$
    s\mapsto \left(\psi^{(n)}\circ\mathfrak{c}\right)(s).
    $$
\end{defn}

\begin{defn}\label{def-mathfrakS}
    For an $F$-stable maximal torus $T$ of $G$, a character $\chi:T^F\to \Qlb^\times$ and a positive integer $n$, we set 
    $$
    \mathfrak S_{T,\chi}(n):=\langle R_{T,\chi\circ \mathrm N_T^n}^{(n)},\psi^{(n)}_{\mathrm S}\rangle_{\mathrm S^{F^n}}.
    $$
    In particular, we have $\mathfrak{S}_{T,\chi}(1)=\langle R_{T,\chi},\psi^{(1)}_{\mathrm S}\rangle_{\mathrm S^F}$.
\end{defn}

\begin{prop}\label{pro-cohointerpretation-pre}Fix an $F$ stable maximal torus $T$ of $G$ and a character $ \chi:T^F\to \Qlb^\times$. For positive integers $n$,
    we have 
    $$
    |G^{F^n}|\mathfrak{S}_{T,\chi}(n)=\mathrm {Tr}((F^n)^*,H^\bullet_c(G,\mathscr R_{T,\chi}\otimes\mathscr S)).
    $$
\end{prop}
\begin{proof}
    Combine Remark \ref{rm-shasheafvalue} and Proposition \ref{prop-charactersheaf-main}, our assertion follows from Grothendieck trace formula and Frobenius reciprocity.
\end{proof}

\begin{defn}\label{def-YTB}
    For a Borel pair $T\subset B$ of $G$, we set $\mathcal Y_{T,B}:=\wt G_{T,B}\times_{G}\mathcal S$. Let $\mathrm {pr}_1:\mathcal Y_{T,B}\to \wt G_{T,B}$ and $\mathrm {pr}_2 :\mathcal Y_{T,B}\to \mathcal S$ be the projections.
\end{defn}

\begin{prop}\label{pro-badmodel}Fix a Borel pair $T\subset B$ of $G$ with $T$ being $F$-stable. Let the notation be as in Definition \ref{def-YTB}.  Fix a character $ \chi:T^F\to \Qlb^\times$.
    We have a canonical identification in the derived category of $\Qlb$-sheaves (see Subsection \ref{subsec-charactershea} for the notation $\rho$, $\pi$ and Subsection \ref{subsec-shasheaf} for the notation $\mathrm q$)
    $$(\pi \circ\mathrm {pr}_1)_!(\mathrm {pr}_1^*\rho^*\mathscr L_\chi\otimes \mathrm{pr}_2^* \mathrm q^*\mathscr L_\psi)\cong \mathscr R_{T,\chi}\otimes \mathscr S.$$ Suppose further that $B$ is $F$-stable; then the left-hand side of the above identification has a natural Weil structure, and the above identification is compatible with the Weil structures (where on the right-hand side, the Weil structure is induced by those of $\mathscr{R}_{T,\chi}$ and $\mathscr S$).
\end{prop}

\begin{proof}
    The identification follows from (ii) of Proposition \ref{prop-charactersheaf-main}, the definition of $\mathscr S$ and the proper base change. The last statement follows from Remark \ref{rm-FstableBchar}.
\end{proof}

\begin{prop} \label{pro-weight-main}
We have the following:
\begin{itemize}
    \item[(i)] The cohomology space
    $H^i_c(G,\mathscr R_{T,\chi}\otimes\mathscr S)$ vanishes unless $0\leq i\leq 2\dim G$.
    \item[(ii)] For eigenvalues $\alpha$ of $F^*$ on $H^i_c(G,\mathscr R_{T,\chi}\otimes\mathscr S)$, we have $|\alpha|\leq q^{i/2}$.
    \item[(iii)] Let  $\beta$  be an eigenvalue of $F^*$ on $H^{2\dim G}_c(G,\mathscr R_{T,\chi}\otimes\mathscr S)$, then $\beta/q^{\dim G}$ is a root of unity.
\end{itemize}

\end{prop}

\begin{proof}
  The statement (i) is due to Proposition \ref{pro-badmodel} and the fact that $\dim \mathcal Y_{T,B}=\dim G$. (See Corollary \ref{cor-dimYTB} below.)
Fix a Borel subgroup $B$ of $G$ containing $T$.  In proving (ii) and (iii), we may replace $F$ by powers. Hence in the rest of this proof, we do assume that  $B$ is $F$-stable. 
    
    By Proposition \ref{pro-badmodel}, we have a canonical identification 
    $$
    H^i_c(G,\mathscr R_{T,\chi}\otimes \mathscr S) \cong H^i_c(\mathcal Y_{T,B},\mathrm {pr}_1^*\rho^*\mathscr L_\chi\otimes \mathrm{pr}_2^* \mathrm q^*\mathscr L_\psi)
    $$
    that is compatible with the action of the Frobenius endomorphism. The statement (ii) is a trivial instance of Deligne's weight theory \cite[Th\'eor\`eme 3.3.1]{D}. The statement (iii) follows from Poincar\'e duality (see Lemma \ref{lem-tpduality} below).
\end{proof}

\begin{lem}\label{lem-tpduality}
    Let $X_0$ be an algebraic scheme of dimension $\leq n$ over $\mathbb F_q$. %Let $K$ be the number of $n$-dimensional irreducible components of $X$.
    Let $X$ be the pullback of $X_0$ to $\mathrm k$. Let $\mathscr F$ be a Weil sheaf on $X$ that is pure of weight $0$ such that all characteristic functions $K_\mathscr F^m$ (for all positive integers $m$) take values in roots of unity. Suppose that $\alpha$ is an eigenvalue of $F^*$ on $H^{2 n}_c(X,\mathscr F)$, then the number $\alpha/q^{ n}$ is a root of unity.
\end{lem} 
\begin{proof}
    We may replace $X_0$ by its reduced subscheme with the same underlying topological space. Hence we assume $X_0$ is reduced. We choose an open subscheme $V_0$ of $X_0$ such that (Let $V$ be the pullback of $V_0$ to $\mathrm k$ and let $j:V\hookrightarrow X$ be the inclusion.)
    \begin{itemize}
        \item $V_0$ contains the generic points of all irreducible components  of $X_0$ that is of dimension $n$. 
        \item $V_0$ is smooth over $\mathbb F_q$.
        \item The restriction of $\mathscr F$ to $V$ is lisse.
    \end{itemize}
    We verify that such $V_0$ exists. Let $i:Z\hookrightarrow X$ be the complement of $j$, with $Z$ reduced. Note that $\dim Z\leq n-1$.
     We have the distinguished triangle
     $$
     j_!j^* \mathscr F\to \mathscr F\to i_*i^*\mathscr F\to.
     $$
     This gives rise to an identification $H^{2n}_c(V,j^*\mathscr F)\cong H^{2n}_c(X,\mathscr F)$ that is compatible with the action of the Frobenius. Note that the Frobenius permute the irreducible components of $V$. Since the characteristic functions of $\mathscr F$ take values in roots of unity, it is easy to see that the eigenvalues of the operator $F^*$ on $H^0(V,(j^*\mathscr F)^\vee)$ are all roots of unity, where $(j^*\mathscr F)^\vee$ is the local system on $V$ such that it is dual to the restriction of $j^*\mathscr F$ over each irreducible component of $V$.  Using Poincar\'e duality, we conclude the proof.
\end{proof}

Similarly, we have the following proposition. The proof is essentially identical to that of Proposition \ref{pro-weight-main}, and we omit it.

\begin{prop}\label{pro-weight-sub}
    Let $i:C\hookrightarrow G$ be an $F$-stable locally closed subscheme. Then we have the following:\begin{itemize}
    \item[(i)] The cohomology space
    $H^j_c\left(C,i^*(\mathscr R_{T,\chi}\otimes\mathscr S)\right)$ vanishes unless $0\leq j\leq 2\dim G$.
    \item[(ii)] For eigenvalues $\alpha$ of $F^*$ on $H^j_c\left(C,i^*(\mathscr R_{T,\chi}\otimes\mathscr S)\right)$, we have $|\alpha|\leq q^{j/2}$.
    \item[(iii)] Let  $\beta$  be an eigenvalue of $F^*$ on $H^{2\dim G}_c\left(C,i^*(\mathscr R_{T,\chi}\otimes\mathscr S)\right)$, then $\beta/q^{\dim G}$ is a root of unity.
\end{itemize}
\end{prop}

Let $\mathbb Z_+$ denote the set of positive integers.
We introduce the following definition to invoke Lemma \ref{lem-const}.
\begin{defn} \label{def-gtype}
Let  $\CP\subset \BZ_+$ be an arithmetic progression. A function $M: \CP \to \BC $ is said to be of geometric type if it is of the form
\[
M(\nu) = \frac{\sum^k_{i=1} a_i \alpha_i^\nu}{\sum^l_{j=1} b_j \beta_j^\nu } ,\quad \nu\in\CP,
\]
where $a_i, \alpha_i, b_j, \beta_j \in\BC$, and the denominator is nonzero for every $\nu\in \CP$. If the denominator is a nonzero constant, we say that $M$ is of trace type.  
\end{defn}
In the remainder of this paper, we fix an identification $\Qlb\cong \mathbb C$.
We have the following elementary lemma, see  \cite[Lemma 2.2]{LMS}. 

\begin{lem} \label{lem-const}
Let $M$ be a function of geometric type defined on an arithmetic progression $\CP\subset \mathbb Z_+$. If $M$ is integer-valued and  has a finite limit $L\in \mathbb C$ as $\nu\to\infty$ through $\CP$, then 
$M$ is a constant function taking the value $L$. 
\end{lem}

The following is a refinement of Proposition \ref{pro-cohointerpretation-pre}.
\begin{prop}\label{pro-cohointerpret-main}Fix an $F$ stable maximal torus $T$ of $G$ and a character $ \chi:T^F\to \Qlb^\times$. 
For each positive integer $n$, we have 
    $$
    \mathfrak{S}_{T,\chi}(n)=\frac{1}{q^{n\dim G}}\mathrm {Tr}((F^n)^*,H^{2\dim G}_c(G,\mathscr R_{T,\chi}\otimes\mathscr S)).
    $$
\end{prop}
\begin{proof}
   Note that $\mathfrak{S}_{T,\chi}(n)$ is an integer by definition. By Proposition \ref{pro-weight-main}, we may choose a integer $N$ such that: for all eigenvalues $\beta$ of $F^*$ on $H^{2\dim G}_c(G,\mathscr R_{T,\chi}\otimes\mathscr S))$, we have $\beta^N/q^{N\dim G}=1$.  By Proposition \ref{pro-weight-main} and Proposition \ref{pro-cohointerpretation-pre}, we have
   $$
   \lim_{i\to \infty} \mathfrak{S}_{T,\chi}(n+iN)=\frac{1}{q^{n\dim G}}\mathrm {Tr}((F^n)^*,H^{2\dim G}_c(G,\mathscr R_{T,\chi}\otimes\mathscr S)).
   $$
   We apply Lemma \ref{lem-const} to the function $f:\mathbb Z_+\to \mathbb C$ given by $j\mapsto \mathfrak{S}_{T,\chi}(n-N+jN)$ and complete the proof. (Note that, by Proposition \ref{pro-cohointerpretation-pre}, the function $f$ is indeed of geometric type in the sense of Definition \ref{def-gtype}.)
\end{proof}

\subsection{Digression: computing top cohomology}
In this subsection, we collect some elementary lemmas for later use.
\begin{lem}\label{lem-topcohogeneral}
    Let $X$ be an algebraic scheme of dimension $n$ over $\mathrm k$. Let $\mathscr F$ be a $\Qlb$-sheaf on $X$. Let $\{X_i\}_{i\in I}$ be the collection of $n$-dimensional irreducible components of $X$, where $I$ is a index set. Let $\mathscr F_i$ for $i\in I$ be the restriction of $\mathscr F$ to $X_i$. 
    Then we have a canonical identification 
$$
H^{2n}_c(X,\mathscr F)\cong \bigoplus_{i\in I}H^{2n}_c(X_i,\mathscr F_i).
$$
Moreover, in the above identification, we may replace $X_i$ by an arbitrary open dense subscheme $U_i$.
\end{lem}
\begin{proof}
    For each $i\in I$, we choose an open dense subscheme $V_i$ of $X_i$ so that $V_i$ is likewise open in $X$. We set $V$ to be the union of $V_i$ for $i\in I$. Let $j:V\hookrightarrow X$ be the inclusion and $i:Z\hookrightarrow X$ be the complement of $j$, where $Z$ is reduced. Note that $Z$ is of dimension $\leq n-1$.
    Applying $H_c^\bullet$ to the distinguished triangle
    $$
    j_!j^* \mathscr F\to \mathscr F\to i_*i^*\mathscr F\to,
    $$
    we see that $\bigoplus_{i\in I}H^{2n}_c(V_i,\mathscr F|{V_i})\cong H^{2n}_c(V,j^*\mathscr F)\cong H^{2n}_c(X,\mathscr F)$. A similar argument shows that $H^{2n}_c(V_i, \mathscr F|{V_i})\cong H^{2n}_c(X_i,\mathscr F_i)$. We verify that the yielding identification $H^{2n}_c(X,\mathscr F)\cong \bigoplus_{i\in I}H^{2n}_c(X_i,\mathscr F_i)$ is independent of the choice of $V_i$. This completes the proof.
\end{proof}

\begin{lem}\label{lem-topcoho-sliced}
    Let $f:X\to Y$ be a morphism between algebraic schemes over $\mathrm k$. Let $n$ be the dimension of $X$. Let $\mathscr F$ be a $\Qlb$-sheaf on $X$. Let $\{X_i\}_{i\in I}$ be the collection of $n$-dimensional irreducible components of $X$, where $I$ is a index set. Let $\mathscr F_i$ for $i\in I$ be the restriction of $\mathscr F$ to $X_i$.  Fix an arbitrary locally closed subscheme $C$ of $Y$, and let $C_X:=C\times_Y X$. Set $$I_C:=\{i\in I:\text{$C_X$ contains the generic point of $X_i$}\}.$$ Let $\mathscr F_i$ for $i\in I$ be the restriction of $\mathscr F$ to $X_i$.  Let $\mathscr F_C$ be the restriction of $\mathscr F$ to $C_X$.
    Then we have
    $$
    H^{2n}_c(C_X,\mathscr F_C)\cong \bigoplus_{i\in I_C} H^{2n}_c(X_i,\mathscr F_i).
    $$
    Moreover, in the above identification, we may replace $X_i$ by an arbitrary open dense subscheme $U_i$.
\end{lem}
\begin{proof}
    The proof is essentially identical to that of Lemma \ref{lem-topcohogeneral}, as we show in what follows.

   Let $C_i:=C_X\cap X_i$ for $i\in I$.
     We see from Lemma \ref{lem-topcohogeneral} that 
    $$
    H^{2n}_c(C_X,\mathscr F_C)\cong \bigoplus_{i\in I_C} H^{2n}_c(C_i,\mathscr F|C_i).
    $$
    Since $C_i$ is a locally closed subscheme of $X_i$ containing the generic point, we see that $C_i$ is indeed an open dense subscheme of $X_i$, yielding
    $$
    H^{2n}_c(C_i,\mathscr F|C_i)\cong H^{2n}_c(X_i,\mathscr F_i).
    $$
    The last statement follows similarly to Lemma \ref{lem-topcohogeneral}.
    This completes the proof.
\end{proof}

\begin{comment}
\begin{rmk}\label{rm-slices-Fstructure}
    In Lemma \ref{lem-topcohogeneral} (\ref{lem-topcoho-sliced}, resp.), suppose that the encountered schemes $X$ ($X$, $Y$ and $C$, resp.) and morphisms are defined over $\mathbb F_q$, and that $\mathscr F$ is a Weil sheaf, then the displayed isomorphism is compatible with the action of the Frobenius.
\end{rmk}
\end{comment}
\subsection{The scheme $\mathcal Y_{T,B}$}\label{subsec-geoYTB} We retain the notation in Subsection \ref{subsec-mul-top}. Throughout this subsection, we fix a Borel pair $T\subset B$ of $G$ with $F$-stable $T$.
In this subsection we study the geometry of $\mathcal Y_{T,B}$ (see Definition \ref{def-YTB}).

\begin{defn}\label{def-YTbw}
    For $w\in B(\mathrm k)\backslash G(\mathrm k)/\mathrm S(\mathrm k)$, we define
    $$
    \mathcal Y_{T,B}^w:=\{(g,xB,y\mathrm S)\in G\times G/B\times G/\mathrm S:\text{$x^{-1}gx\in B$, $y^{-1}gy\in \mathrm S$ and $x^{-1}y\in Bw\mathrm S$}\}.
    $$
    We set $\mathrm v_w:\mathcal Y_{T,B}^w\to G$ to be the projection to the first factor.
    
    Fix $w\in B(\mathrm k)\backslash G(\mathrm k)/\mathrm S(\mathrm k)$ in this paragraph.  More concretely, let  $\mathcal O_w$ be the orbit of $(eB,w_0\mathrm S)\in G/B\times G/\mathrm S$ under the action of $G$ given by $g\cdot (xB,y\mathrm S)\mapsto (gxB,gy\mathrm S)$, where $w_0\in G(\mathrm k)$ is a representative of $w$. (We endow $\mathcal O_w$ with the reduced scheme structure.)
    Then we have a pullback square  
    \begin{equation*}
        \xymatrix{
&\mathcal Y_{T,B}^w\ar[r]\ar[d]& \mathcal O_w\ar[d]\\
 &G\times \mathcal O_w\ar[r]& \mathcal O_w\times \mathcal O_w       
        }
    \end{equation*}
    where \begin{itemize}
        \item the left vertical map is the  inclusion;
        \item the lower horizontal map is given by $(g,r)\mapsto (g\cdot r,r)$; (the action is introduced above)
        \item the right vertical map is the diagonal embedding;
        \item the upper horizontal map is given by $(g,xB,y\mathrm S)\mapsto (xB,y\mathrm S)$.
    \end{itemize}
\end{defn}

   Unwinding the definition, we see that $\mathcal Y_{T,B}^w$ is a locally closed subscheme of $\mathcal Y_{T,B}$. (Since each $\mathcal O_w$ is a locally closed subscheme of $G/B\times G/\mathrm S$.) We have an obvious partition  (note that the index set is finite by Lemma \ref{lem-BSdoublecoset})\begin{equation}\label{eq-partitionYTB}
       \mathcal Y_{T,B}=\bigcup_{w\in  B(\mathrm k)\backslash G(\mathrm k)/\mathrm S(\mathrm k)}\mathcal Y_{T,B}^w.
   \end{equation}

\begin{lem}\label{lem-dim-YTBw}
    For $w\in B(\mathrm k)\backslash G(\mathrm k)/\mathrm S(\mathrm k)$, we have 
    $\dim \mathcal Y_{T,B}^w=\dim G$. Further, the scheme $\mathcal Y_{T,B}^w$ is irreducible and has dimension $\dim G$.
\end{lem}
\begin{proof}Fix $w\in B(\mathrm k)\backslash G(\mathrm k)/\mathrm S(\mathrm k)$ in this proof.
    Let $\mathfrak p_w:\mathcal Y_{T,B}^w\to G/B\times G/\mathrm S$ temporarily be the projection. We see that the image of $\mathfrak{p}_w$ is the orbit $\mathcal O_w$ of $(eB,w_0\mathrm S)\in G/B\times G/\mathrm S$ under the action of $G$ given by $g\cdot (xB,y\mathrm S)\mapsto (gxB,gy\mathrm S)$, where $w_0\in G(\mathrm S)$ is a representative of $w$. 

    Recall the following cartesian diagram introduced in Definition \ref{def-YTbw}
    \begin{equation*}
        \xymatrix{
&\mathcal Y_{T,B}^w\ar[r]\ar[d]& \mathcal O_w\ar[d]\\
 &G\times \mathcal O_w\ar[r]& \mathcal O_w\times \mathcal O_w       
        }
    \end{equation*}
   \begin{comment} where \begin{itemize}
        \item the left vertical map is given by $\mathcal Y_{T,B}^w\ni p=(g,xB,y\mathrm S)\mapsto \left(g,\mathrm p_w(p)\right)$;
        \item the lower horizontal map is given by $(g,r)\mapsto (g\cdot r,r)$; (the action is introduced in the previous paragraph)
        \item the right vertical map is the diagonal;
        \item the upper horizontal is $\mathrm p_w$.
    \end{itemize}\end{comment}
    The lower horizontal is smooth and surjective, of which each fibre is isomorphic to $B\cap w_0 \mathrm Sw_0^{-1}$. Note that $\mathcal O_w$ is smooth and irreducible of dimension $\dim G-\dim (B\cap w_0 \mathrm S w^{-1}_0)$.
    Consequently the upper horizontal map is likewise smooth, indicating that $\mathcal Y_{T,B}^w$ is smooth and irreducible of dimension $\dim G$. 
\begin{comment}
The image $\mathcal O_w$ of $\mathfrak{p}_w$ is irreducible and has dimension $\dim G-\dim B\cap w_0 \mathrm S w^{-1}_0$. It is clear that a nonempty fibre of $\mathfrak{p}_w$ is irreducible and of dimension $\dim (B\cap w_0\mathrm Sw^{-1}_0)$. 
    This shows that $\dim \mathcal Y_{T,B}^w=\dim G$ and $\mathcal Y_{T,B}^w$ has a unique irreducible component of dimension $\dim G$. Combining the fact that $\mathcal Y_{T,B}^w$ is equidimensional, we see that it is  irreducible, as desired.
\end{comment}
\end{proof}

\begin{cor}\label{cor-dimYTB}
 We have   $\dim \mathcal Y_{T,B}=\dim G$.
\end{cor}   
\begin{proof}
    This is a immediate consequence of  Lemma \ref{lem-dim-YTBw} and the partition (\ref{eq-partitionYTB}).
\end{proof}

Motivated by Lemma \ref{lem-BSdoublecoset} and Proposition \ref{pro-bijec-phi}, we make the following definition.

\begin{defn}\label{def-PhiB} We set
    $$\Phi_{B}:=\{w\in B(\mathrm k)\backslash G(\mathrm k)/\mathrm S(\mathrm k) : \text{ $w^{-1}Bw\cap N$ has trivial image under $\mathrm c$} \},$$
   where we abuse the notation by denoting some/any representative of $w\in B(\mathrm k)\backslash G(\mathrm k)/\mathrm S(\mathrm k)$ again by $w$.
\end{defn}
In particular, we have $\Phi_{\mathrm B^{op}}=\Phi$, where the right-hand side is defined in Proposition \ref{pro-bijec-phi} and $\mathrm B^{op}$ is defined in Definition \ref{def-various}.

\begin{rmk}\label{rm-phi-self-B}(We abuse the notation by denoting a representative of $w\in B(\mathrm k)\backslash G(\mathrm k)/\mathrm S(\mathrm k)$ again by $w$.)
    Parallel to Remark \ref{rm-phi-selfcen}, we may rephrase $\Phi_B$ as 
    $$\Phi_B=
  \{  w\in B(\mathrm k)\backslash G(\mathrm k)/\mathrm S(\mathrm k):C_{w^{-1} Bw\cap S}(U_w)=U_w\} ,
    $$
    where $U_w$ is some/any maximal torus of $w^{-1}Bw\cap \mathrm S$, and $C_{w^{-1}Bw\cap \mathrm S }(U_w)$ denotes the corresponding centralizer. Unraveling the definition, we may also rephrase $\Phi_B$ as
    $$\Phi_B=
     \{  w\in B(\mathrm k)\backslash G(\mathrm k)/\mathrm S(\mathrm k):\text{ $w^{-1}Bw\cap \mathrm S$ has trivial image under $\mathfrak c$}\}. 
    $$
    Fix $w\in \Phi_B$ and an maximal torus $U_w$ of $w^{-1}Bw\cap \mathrm S$, we also verify that $C_{w^{-1}Bw\cap \mathrm S }(s)=U_w$ for a general $s\in U_w(\mathrm k)$. Consequently, we have 
    $$
    \Phi_B=\{w\in B(\mathrm k)\backslash G(\mathrm k)/\mathrm S(\mathrm k):\text{ a general element in $w^{-1}Bw\cap \mathrm S$ is semisimple}\}.
    $$
\end{rmk}

\subsection{The locus $\mathfrak{V}$}\label{subsec-locus} We retain the notation in Subsection \ref{subsec-mul-top}.  Recall the torus $U$ introduced in Definition \ref{def-torusU}. Let $\kappa:G\times U\to G$ by the map given by $(g,u)\mapsto gug^{-1}$. Note that $\kappa$ is a finite-type morphism between noetherian schemes. By a theorem of Chevalley, we see that the (set-theoretic) image of $\kappa$ is a constructible subset of $G$. Since the source $G\times U$ of $\kappa$ is integral and the morphism $\kappa$ is defined over $\mathbb F_q$, we may choose an $F$-stable integral locally closed  subscheme $\mathfrak V$ of $G$ such that:
\begin{itemize}
    \item $\mathfrak{V}$ is contained in the scheme-theoretic image of $\kappa$;
    \item $\mathfrak{V}$ is dense in the (underlying topological space of the) scheme-theoretic image of $\kappa$.
\end{itemize}
Let $i_{\mathfrak{V}}:\mathfrak{V}\hookrightarrow G$ be the inclusion.

Roughly speaking, the rest of this section is to show `` $\mathfrak{V}$ is the locus contributing to the multiplicity of Shalika models".

In the rest of this subsection, we fix a Borel pair $T\subset B$ of $G$ with $F$-stable $T$.
\begin{prop}\label{pro-YTBw-toG}
    For $w\in B(\mathrm k)\backslash G(\mathrm k)/\mathrm S(\mathrm k)$, we have:
    (Recall the map $\mathrm v_w$ introduced in Definition \ref{def-YTbw})
    \begin{itemize}
        \item If $w\in \Phi_B$, then $\mathrm v_w$ maps a general point of $\mathcal Y_{T,B}^w$ to $\mathfrak{V}$. In particular, the scheme $\mathcal Y_{T,B}^w\times_{\mathrm v_w,G} \mathfrak{ V}$ contains the  generic point  of $\mathcal Y_{T,B}^w$. 
        \item If $w\notin \Phi_B$, then $\mathrm v_w$ maps a general point of $\mathcal Y_{T,B}^w$ to the complement of $\mathfrak{V}$ in $G$. In particular, the scheme $\mathcal Y_{T,B}^w\times_{\mathrm v_w,G} \mathfrak{ V}$ is of dimension $\leq \dim G-1$. 
    \end{itemize}
\end{prop}
\begin{proof}
%For each $w\in B(\mathrm k)\backslash G(\mathrm k)/\mathrm S(\mathrm k)$, set $B_w:= w_0^{-1}Bw_0\cap \mathrm S$  for convenience, where $w_0$ is a representative of $w$. 
Appropriately choosing a representative $w_0$ for each $w\in B(\mathrm k)\backslash G(\mathrm k)/\mathrm S(\mathrm k)$, we may assume that $U$ is a maximal torus of $w_0^{-1}Bw_0\cap \mathrm S$ by (an obvious variant of) Proposition \ref{pro-bijec-phi}. In the rest of this proof, we fix one such representative $w_0$ for each $w\in B(\mathrm k)\backslash G(\mathrm k)/\mathrm S(\mathrm k)$, and set $B_w:= w_0^{-1}Bw_0\cap \mathrm S$  for convenience.

We fix $w\in B(\mathrm k)\backslash G(\mathrm k)/\mathrm S(\mathrm k)$ in this paragraph. 
   % Let $\mathfrak{p}_w:\mathcal Y_{T,B}^w\to \wt G_{T,B}\times G/\mathrm S$ be the morphism introduced in the proof of Lemma \ref{lem-dim-YTBw}. %Let $p=(eB,w_0\mathrm S)$ be the point introduced in \select{loc. cit.}.
    We verify that the image of $\mathrm v_w$ is the image of $\mathfrak{i}_w:G\times B_w \to G$, where $\mathfrak{i}_w$ is given by $(g,h)\mapsto ghg^{-1}$.
 Hence we have (note that $\kappa$ is the restriction of $\mathfrak i_w$)
\begin{equation}\label{eq-imagecompar}
    \mathrm {Im}(\kappa) \subset \mathrm {Im}(\mathrm v_w),
\end{equation}
where $\mathrm {Im}$ denotes the set-theoretic image. Note that $\mathfrak{V}\subset \mathrm {Im}(\kappa)$ set-theoretically.
    
    Consequently, it suffices to show:
    \begin{itemize}
        \item[(1)] a general point of $B_w$ is conjugate (under $G$) to an element in $U$ if $w\in \Phi_B$;
        \item[(2)] a general point of $B_w$ is not in $\mathfrak{V}$ if $w\notin \Phi_B$.
    \end{itemize}

Let $N_w$ be the kernel of the obvious map $B_w\twoheadrightarrow U$. Let $\mathfrak{j}_w:N_w\times U\to B_w$  be the map given by $(n,s)\mapsto nsn^{-1}$. By  Remark \ref{rm-phi-self-B}, we have $C_{B_w}(U)=U$ if and only if $w\in \Phi_B$. This indicates
$$
\Phi_B=\{w\in B(\mathrm k)\backslash G(\mathrm k)/\mathrm S(\mathrm k):\text{$\mathfrak{j}_w$ has dense image} \}.
$$
%(Indeed, for $w\in \Phi_B$, we can verify that $C_{B_w}(s)=U$ for a general element $s\in U(\mathrm k)$.)

Consequently, we have:
\begin{itemize}
    \item [(i)] a general point of $B_w$ is conjugate (under $N_w$) to an element in $U$ if $w\in \Phi_B$;
    \item [(ii)] a general element of $B_w$ is not semisimple if $w\notin \Phi_B$; in particular, a general element of $B_w$ is not in $\mathfrak{V}$ if $w\notin \Phi_B$.
\end{itemize}
We see from the definition of $\mathfrak{V}$ that: (i) implies (1), and (ii) implies (2). This completes the proof.
\end{proof}

In the rest of this subsection, we fix a character $\chi:T^F\to \Qlb^\times$. Recall the identification of $\Qlb$-sheaves
$$(\pi \circ\mathrm {pr}_1)_!(\mathrm {pr}_1^*\rho^*\mathscr L_\chi\otimes \mathrm{pr}_2^* \mathrm q^*\mathscr L_\psi)\cong \mathscr R_{T,\chi}\otimes \mathscr S$$
introduced in Proposition \ref{pro-badmodel}. We set $\mathscr F_{T,B,\chi}:=\mathrm {pr}_1^*\pi^*\mathscr L_\chi\otimes \mathrm{pr}_2^* \mathrm q^*\mathscr L_\psi$ to be the sheaf on $\mathcal Y_{T,B}$ for convenience. For $w\in B(\mathrm k)\backslash G(\mathrm k)/\mathrm S(\mathrm k)$, let $\mathscr F_{T,B,\chi}^w$ be the restriction of $\mathscr F_{T,B,\chi}$ to $\mathcal Y_{T,B}^w$.

\begin{prop}\label{pro-vanishing-YTBw-ninPhi} Fix $w\in B(\mathrm k)\backslash G(\mathrm k)/\mathrm S(\mathrm k)$.
    Suppose that $w\notin \Phi_B$, then we have
    $$
    H^{i}_c(\mathcal Y_{T,B}^w,\mathscr F_{T,B,\chi}^w)=0
    $$for all integers $i$.
\end{prop}
\begin{proof}Fix one such $w$ throughout of this proof.
Note that $G$ acts on $\mathcal Y_{T,B}^w$ via $g\cdot(g',xB.y\mathrm S)=(gg'g^{-1},gxB,gy\mathrm S)$. We verify that the sheaf $\mathscr F_{T,B,\chi}^w$ is equivariant with respect to this action.
    Let $\mathfrak p_w:\mathcal Y_{T,B}^w\to G/B\times G/\mathrm S$ be the projection introduced in the proof of Lemma \ref{lem-dim-YTBw}. Note that $G$ acts on $G/B\times G/\mathrm S$ via $g\cdot(xB,y\mathrm S)=(gxB,gy\mathrm S)$. Then we verify that $\mathfrak{p}_w$ is equivariant with respect to $G$-action. 
    
    By a spectral sequence, it suffices to show $(\mathfrak{p}_w)_! \mathscr F_{T,B,\chi}^w$ vanishes; or equivalently, for any point $r\in G/B\times G/\mathrm S$, the stalk $\left((\mathfrak{p}_w)_! \mathscr F_{T,B,\chi}^w\right)_r$ vanishes. For $r\in G/B\times G/\mathrm S$, let $f_r$ denote the fibre of $\mathfrak p_w$ over $r$.  By the proper base change, it suffices to show $H^\bullet_c(f_r,\mathscr F_{T,B,\chi}^w|f_r)$ vanish for any $r\in G/B\times G/\mathrm S$.
    By the equivariance introduced in the previous paragraph and the fact that the image of $\mathfrak{p}_w$ forms a single $G$-orbit, it suffices to show $H^\bullet_c(f
    _p,\mathscr F_{T,B,\chi}^w|f_p)$ for $p=(w_0^{-1}B,e\mathrm S)$, where $w_0\in G(\mathrm k)$ is a representative of $w$.
    
    We observe that the fibre $f_p$ can be identified with $B_w:=w_0^{-1}Bw_0\cap \mathrm S$ in a natural way. Under this identification, the sheaf $\mathscr F_{T,B,\chi}^w|f_p$ is identified with $\mathrm d_w^*\mathscr L_\chi\otimes \mathfrak{c}^*_w\mathscr L_\psi$, where
    \begin{itemize}
        \item $\mathscr L_\chi$ is introduced in Subsection \ref{subsec-sheavesontori};
        \item $\mathrm d_w:B_w\to T$  is given by $b\mapsto \mathrm d_{T}^B(w_0bw_0^{-1})$; (Here, the morphism $\mathrm d_T^B$ is introduced in Subsection \ref{subsec-charactershea}.)
        \item $\mathscr L_\psi$ is the Artin–Schreier sheaf on $\mathbb G_a$;
        \item $\mathfrak{c}_w:B_w\to \mathbb G_a$ is the restriction of $\mathfrak{c}$  introduced above Proposition \ref{pro-bijec-phi}.
    \end{itemize}

    Consequently, it remains to show that $H^{\bullet}_c(B_w,\mathrm d_w^*\mathscr L_\chi\otimes \mathfrak{c}^*_w\mathscr L_\psi)$ vanish.

    Recall the quotient map $q_\mathrm P:\mathrm P\to L$ introduced below Definition \ref{def-various}. Let $q_w:B_w\to L$ be the restriction of $q_\mathrm P$. By a spectral sequence, it suffices to show $(q_w)_! (\mathrm d_w^*\mathscr L_\chi \otimes \mathfrak{c}_w^*\mathscr L_\psi)=0$. Passing to stalks, it suffices to show $H^\bullet_c(\mathrm {Ker}(q_w),c^*\mathscr L_\psi)=0$, where $c:\mathrm {Ker}(q_w)\to \mathbb G_a$ denotes the restriction of $\mathfrak{c}_w:B_w\to \mathbb G_a$. 
    We have $\mathrm {Ker}(q_w)=w_0^{-1}Bw_0\cap N$. The map $c$ is a surjective (by Definition \ref{def-PhiB}) group homomorphism with a connected unipotent kernel (see Remark \ref{rm-ker-connected-unipo}). Consequently, we may identify $c_!c^*\mathscr L_\psi$ with $\mathscr L_\psi$ up to shifts and Tate twists. Since $H^{\bullet}_c(\mathbb G_a,\mathscr L_\psi)=0$, we likewise have $H^{\bullet}_c(\mathrm {Ker}(q_w),c^*\mathscr L_\psi)=H^\bullet_c(\mathbb G_a,c_!c^*\mathscr L_\psi)=0$. This completes the proof.
\end{proof}

\begin{prop}\label{pro-coho-passingto-locus}
    We have a canonical identification of cohomology spaces that is compatible with the action of Frobenius $F^*$
    $$
    H^{2\dim G}_c(G,\mathscr R_{T,\chi}\otimes \mathscr S)\cong H^{2\dim G}_c\left(\mathfrak{ V},i_{\mathfrak{V}}^*(\mathscr R_{T,\chi}\otimes \mathscr S)\right).
    $$
\end{prop}
\begin{proof}
    This is a combination of Proposition \ref{pro-vanishing-YTBw-ninPhi} and Lemma \ref{lem-topcoho-sliced}, as we show in the following. 

During this proof, we will implicitly use Lemma \ref{lem-dim-YTBw} and the partition (\ref{eq-partitionYTB}). Note that each piece of the partition  (\ref{eq-partitionYTB}) is an irreducible locally closed subscheme of $\mathcal Y_{T,B}$ with dimension $\dim G$.
    
Set $\mathscr F:=\mathscr R_{T,\chi}\otimes \mathscr S$ in this proof for convenience.
    By Lemma \ref{cor-dimYTB}, Lemma \ref{lem-topcohogeneral} and Proposition \ref{pro-badmodel}, we have  identifications (some detail omitted)$$
    H^{2\dim G}_c(G,\mathscr F)\cong H^{2\dim G}_c(\mathcal Y_{T,B},\mathscr F_{T,B,\chi})\cong  \bigoplus_{w\in B(\mathrm k)\backslash G(\mathrm k)/\mathrm S(\mathrm k)} H^{2\dim G}_c(\mathcal Y_{T,B}^w,\mathscr F_{T,B,\chi}^w).
    $$
    By Proposition \ref{pro-vanishing-YTBw-ninPhi}, we see that
    $$H^{2\dim G}_c(\mathcal Y_{T,B}^w,\mathscr F_{T,B,\chi}^w)
    $$
    vanishes unless $w\in \Phi_B$.

By Proposition \ref{pro-YTBw-toG} and Lemma \ref{lem-topcoho-sliced} (applied to the morphism $\mathcal Y_{T,B}\xrightarrow[]{\pi\circ\mathrm {pr}_1} G$, the locally closed subset $\mathfrak{V}$ and the sheaf $\mathscr F_{T,B,\chi}$), we have     $$
H^{2\dim G}_c(\mathfrak{V},i_{\mathfrak{V}}^*\mathscr F)\cong \bigoplus_{w\in \Phi_B} H^{2\dim G}_c(\mathcal Y_{T,B}^w,\mathscr F_{T,B,\chi}^w).
$$

Putting all things together, we get the desired     $$
    H^{2\dim G}_c(G,\mathscr F)\cong H^{2\dim G}_c\left(\mathfrak{ V},i_{\mathfrak{V}}^*\mathscr F\right).
    $$

It remains to show this identification is compatible with the action of $F^*$. For this, we track the identification in a functorial way. Let $\bar {\mathfrak{V}}$ denote the closure of $\mathfrak{V}$ in $G$, with the reduced scheme structure. Let $i_{\bar {\mathfrak{V}}}: \bar {\mathfrak{V}}\hookrightarrow G$ be the closed immersion, and let $j_{\mathfrak{V}}:\mathfrak{V}\hookrightarrow \bar {\mathfrak{V}}$ be the open immersion. We have the following diagram
$$
\mathscr F\rightarrow (i_{\bar {\mathfrak{V}}})_*i_{\bar {\mathfrak{V}}}^* \mathscr F \leftarrow  (i_{\bar {\mathfrak{V}}})_*(j_{\mathfrak{V}})_!j_{\mathfrak{V}}^*i_{\bar {\mathfrak{V}}}^* \mathscr F,
$$
where the  arrow on the left is given by the adjunction $id\to (i_{\bar {\mathfrak{V}}})_*i_{\bar {\mathfrak{V}}}^*$, and the arrow on the right is given by the adjunction $(j_{\mathfrak{V}})_!j_{\mathfrak{V}}^*\rightarrow id$.
Note that these adjunctions are defined over $\mathbb F_q$. Hence we can apply $H^{2\dim G}_c(G,\_)$ to the above diagram, yielding morphisms that are compatible with the action of $F^*$:
$$
H^{2\dim G}_c(G,\mathscr F)\rightarrow H^{2\dim G}_c(G,(i_{\bar {\mathfrak{V}}})_*i_{\bar {\mathfrak{V}}}^* \mathscr F) \leftarrow H^{2\dim G}_c(G, (i_{\bar {\mathfrak{V}}})_*(j_{\mathfrak{V}})_!j_{\mathfrak{V}}^*i_{\bar {\mathfrak{V}}}^* \mathscr F).
$$
By Proposition \ref{pro-YTBw-toG}, Proposition \ref{pro-vanishing-YTBw-ninPhi} and Lemma \ref{lem-topcoho-sliced} (applied to the morphism $\mathcal Y_{T,B}\xrightarrow[]{\pi\circ\mathrm {pr}_1} G$, the locally closed subset $\mathfrak{V}$, $\bar{\mathfrak{V}}$ and the sheaf $\mathscr F_{T,B,\chi}$), the above arrows linking $H^{2\dim G}_c(G,\_)$ are all isomorphisms. We verify that the resulting isomorphism $H^{2\dim G}_c(G,\mathscr F)\cong H^{2\dim G}_c(G, (i_{\bar {\mathfrak{V}}})_*(j_{\mathfrak{V}})_!j_{\mathfrak{V}}^*i_{\bar {\mathfrak{V}}}^* \mathscr F)= H^{2\dim G}_c(\mathfrak{V},i^*_{\mathfrak{V}}\mathscr F)$ is precisely the isomorphism we constructed in the previous paragraph. This shows the compatibility of the actions of $F^*$ and completes the proof. 
\end{proof}

\subsection{Passing to semisimple elements}\label{subsec-passing-to-semisimple} We retain the notation in Subsection \ref{subsec-locus}. In particular, we fix a Borel pair $T\subset B$ of $G$, with $F$-stable $T$. And we fix a character $\chi:T^F\to\Qlb^\times$.

Recall the torus $\mathrm T$ of $G$ intrdoduced in Definition \ref{def-various}. Let $\mathrm s:G\times \mathrm T\to G$ be the map given by $(g,t)\mapsto gtg^{-1}$. Note that $\mathrm s$ is indeed defined over $\mathbb F_q$. The set-theoretic image $\mathrm {Im}(\mathrm s)$ of $\mathrm s$ is a constructible subset of $G$ containing $\mathfrak{V}$. Consequently, there exists a finite partition
$$\mathrm {Im}(\mathrm s)=\mathfrak{V}\sqcup \left(\bigsqcup_{j\in J} S_j\right),$$ where\begin{itemize}
    \item $J$ is a finite index set;
    \item for $j\in J$, the set $S_j$ is an $F$-stable locally closed subset of $G$. 
\end{itemize}
(Such a partition can be done over $\mathbb F_q$, then the pullback to $\mathrm k$ is a desired one.)

\textbf{We fix one such partition in the rest of this section.}

For $j\in J$, we endow $S_j$ with the reduced scheme structure and denote the corresponding scheme again by $S_j$ (by abusing the notation). Note that, the set of closed points of $\mathrm {Im}(\mathrm s)$ is precisely the set of semisimple element in $G(\mathrm k)$.  Let $\mathfrak{i}_j: S_j\hookrightarrow G$ denote the inclusion for $j\in J$. Then $\mathfrak{i}_j^*(\mathscr R_{T,\chi}\otimes \mathscr S)$ is a complex of Weil sheaves, where the Weil structure is induced by that of $\mathscr R_{T,\chi}\otimes \mathscr S$.

\begin{lem}\label{lem-vanishing-Sj}
    Let $J$ be as above.  Fix $j\in J$. Then:
    \begin{itemize}
        \item[(1)] The cohomology space $H_c^{i}(S_j,\mathfrak{i}_j^*(\mathscr R_{T,\chi}\otimes \mathscr S))$ vanishes unless $0\leq i\leq 2\dim G-1$;
        \item[(2)] For eigenvalues $\alpha$ of $F^*$ on $H_c^{i}(S_j,\mathfrak{i}_j^*(\mathscr R_{T,\chi}\otimes \mathscr S))$, we have $|\alpha|\leq q^{i/2}$.
    \end{itemize}
\end{lem}
\begin{proof}
  The statement  (2) follows from (ii) of Proposition \ref{pro-weight-sub}. 
    
    We will implicitly use the fact (Lemma \ref{lem-dim-YTBw} and Corollary \ref{cor-dimYTB}) that $\dim \mathcal Y_{T,B}=\dim G$ in the rest of this proof. Using Proposition \ref{pro-badmodel} and proper base change, it is elementary to show that the space $H_c^{i}(S_j,\mathfrak{i}_j^*(\mathscr R_{T,\chi}\otimes \mathscr S))$ vanishes unless $0\leq i\leq 2\dim G$. Hence it suffices to show the vanishing of $H_c^{2\dim G}(S_j,\mathfrak{i}_j^*(\mathscr R_{T,\chi}\otimes \mathscr S))$.
    By Lemma \ref{lem-topcoho-sliced} (applied to the morphism $\mathcal Y_{T,B}\to G$, the locally closed subset $S_j$ and the sheaf $\mathscr F_{T,B,\chi}=\mathrm {pr}_1^*\rho^*\mathscr L_\chi\otimes \mathrm{pr}_2^* \mathrm q^*\mathscr L_\psi$) and Proposition \ref{pro-badmodel}, we have $$H_c^{2\dim G}(S_j,\mathfrak{i}_j^*(\mathscr R_{T,\chi}\otimes \mathscr S))\cong \bigoplus\limits_{w\in K}H_w,$$ where %(Note that by Lemma \ref{lem-dim-YTBw}, each $\mathcal Y_{T,B}^w$ has a unique generic point of dimension $\dim G$.) 
    $$K:=\{w\in B(\mathrm k)\backslash G(\mathrm k)/\mathrm S(\mathrm k):\text{$\mathrm v_w$ maps the  generic point of $\mathcal Y_{T,B}^w$ into $S_j$}\}$$
    (see Definition \ref{def-YTbw} for the definition of $\mathrm v_w$)
    and $H_w:= H^{2\dim G}_c(\mathcal Y_{T,B}^w, \mathscr F_{T,B,\chi}^w)$.
    %(1) follows from Lemma \ref{lem-topcoho-sliced}, 
    By Proposition \ref{pro-YTBw-toG}, the set $K$ is disjoint from $\Phi_B$.
    Hence by Proposition \ref{pro-vanishing-YTBw-ninPhi}, for each $w\in K$, the space $H_w$ vanishes. This shows (1), as desired. 
     %wrong(Note that we can indeed show that ``$H_c^{i}(S_j,\mathfrak{i}_j^*(\mathscr R_{T,\chi}\otimes \mathscr S))$ vanishes unless $0\leq i\leq 2\dim G-2$'', but we do not require this.)
\end{proof}

%We set $\mathfrak K:=\left(|B(\mathrm k)\backslash G(\mathrm k)/\mathrm S(\mathrm k)|\right)!\cdot (2\mathfrak{n})!$.
\begin{thm}\label{thm-passing-to-semisimple} Fix a positive integer $n$.
    The following numbers coincide: (See the proof for the existence of the relevant limits.)
    \begin{itemize}
        \item[(1)]  $\mathfrak{S}_{T,\chi}(n)$;
        \item[(2)] $\frac{1}{q^{n\dim G}}\mathrm {Tr}((F^n)^*,H^{2\dim G}_c(G,\mathscr R_{T,\chi}\otimes\mathscr S))$;
        \item[(3)] $\frac{1}{q^{n\dim G}}\mathrm {Tr}((F^n)^*,H^{2\dim G}_c\left(\mathfrak{V},{i}_{\mathfrak{ V}}^*(\mathscr R_{T,\chi}\otimes\mathscr S))\right)$;
        \item[(4)] $
\lim\limits_{i\to \infty}\frac{1}{|G^{F^{n+i\mathfrak{K}}}|}\sum\limits_{s\in G^{F^{n+i\mathfrak{K}}}}\left( R_{T,\chi\circ \mathrm N_T^{n+i\mathfrak{K}}}^{(n+i\mathfrak{K})}\cdot \mathrm {Ind}_{\mathrm S^{F^{n+i\mathfrak{K}}}}^{G^{F^{n+i\mathfrak{K}}}}(\psi^{(n+i\mathfrak{K})}_{\mathrm{S}})\right)(s)
$ for sufficiently divisible $\mathfrak{K}$;     
\item [(5)]
$
\lim\limits_{i\to \infty}\frac{1}{|G^{F^{n+i\mathfrak{K}}}|}\sum\limits_{s\in \mathfrak{V}^{F^{n+i\mathfrak{K}}}}\left( R_{T,\chi\circ \mathrm N_T^{n+i\mathfrak{K}}}^{(n+i\mathfrak{K})}\cdot \mathrm {Ind}_{\mathrm S^{F^{n+i\mathfrak{K}}}}^{G^{F^{n+i\mathfrak{K}}}}(\psi^{(n+i\mathfrak{K})}_{\mathrm{S}})\right)(s)
$ for sufficiently divisible $\mathfrak{K}$;       
\item[(6)] $
\lim\limits_{i\to \infty}\frac{1}{|G^{F^{n+i\mathfrak{K}}}|}\sum\limits_{s}\left( R_{T,\chi\circ \mathrm N_T^{n+i\mathfrak{K}}}^{(n+i\mathfrak{K})}\cdot \mathrm {Ind}_{\mathrm S^{F^{n+i\mathfrak{K}}}}^{G^{F^{n+i\mathfrak{K}}}}(\psi^{(n+i\mathfrak{K})}_{\mathrm{S}})\right)(s)
$ for sufficiently divisible $\mathfrak{K}$, where  $s$ runs over  semisimple elements of $G^{F^{n+i\mathfrak{K}}}$ in the inner summation.
    \end{itemize}
    (The reader is advised to review Remark \ref{rm-shasheafvalue} and Subsection \ref{subsec-DL} for the notation.)
\end{thm}
\begin{proof}
 By  Proposition \ref{pro-cohointerpret-main}, we see that (1) and (2) are identical. By Proposition \ref{pro-coho-passingto-locus}, we see (2) and (3) coincide. 
 
   Proposition \ref{pro-weight-main} indicates that: for each eigenvalue $\alpha$ of $F^*$ on $H^{2\dim G}_c(G,\mathscr R_{T,\chi}\otimes\mathscr S)$, the number $\alpha/q^{\dim G}$ is a  root of unity. Hence (2) and (4) are identical by Grothendieck trace formula and Proposition \ref{pro-weight-main} (let the integer $\mathfrak{K}$ be such that $(\alpha/q^{\dim G})^{\mathfrak{K}}=1$ for each eigenvalue $\alpha$ of $F^*$ on $H^{2\dim G}_c(G,\mathscr R_{T,\chi}\otimes\mathscr S)$). 
Similarly, by Proposition \ref{pro-weight-sub} (ii) and Grothendieck trace formula, we see that (3) and (5) coincide.

 By Grothendieck trace formula and Lemma \ref{lem-vanishing-Sj},  for each $j\in J$, we have $$
\lim\limits_{i\to \infty}\frac{1}{|G^{F^{n+i\mathfrak{K}}}|}\sum_{s\in S_j^{F^{n+i\mathfrak{K}}}}\left( R_{T,\chi\circ \mathrm N_T^{n+i\mathfrak{K}}}^{(n+i\mathfrak{K})}\cdot \mathrm {Ind}_{\mathrm S^{F^{n+i\mathfrak{K}}}}^{G^{F^{n+i\mathfrak{K}}}}(\psi^{(n+i\mathfrak{K})}_{\mathrm{S}})\right)(s)=0.
$$ 
Subtracting (5) from (6),  it is immediate to show that (6) and (5) are identical given the partition introduced at the beginning of this subsection.  

This completes the proof.
\end{proof}

\begin{rmk}  We will not use this Remark in the rest of this paper.
    In Theorem \ref{thm-passing-to-semisimple} (4), (5) and (6), we may take $\mathfrak{K}=2(2\mathfrak n)!$. This can be seen as follows:
    \begin{itemize}
       \item[(i)] It suffices to show that: if $F=F_{\mathrm {GL}}$, then in Theorem \ref{thm-passing-to-semisimple} (4), (5) and (6), we may take $\mathfrak{K}=(2\mathfrak n)!$. In the following two items, we assume that $F=F_{\mathrm{GL}}$.
        \item[(ii)] Suppose that $T$  is  contained in an $F$-stable Borel $B$ in this item. Fix $w\in B(\mathrm k)\backslash G(\mathrm k)/\mathrm S(\mathrm k)$ in this item. Then  $\mathcal Y_{T,B}^{w}$ is $F$-stable and $\mathscr F_{T,B,\chi}^w$ is a Weil sheaf (with the natural Weil structure). We verify that the sheaf $\mathscr F_{T,B,\chi}^w$ is either the constant Weil sheaf or  is not geometrically constant. Consequently (by Poincar\'e duality), the operator $F^*$ acts as the multiplication by $q^{\dim G}$ on $H^{2\dim G}_c(\mathcal Y_{T,B}^w,\mathscr F_{T,B,\chi}^w)$. By Lemma \ref{lem-topcohogeneral} and Proposition \ref{pro-badmodel}, we see that $F^*$ acts as multiplication by $q^{\dim G}$ on $H^{2\dim G}_c(G,\mathscr R_{T,\chi}\otimes \mathscr L)$.

        \item[(iii)] For a general $F$-stable maximal torus $T$ of $G$, we see by the explicit classification of $F$-stable maximal tori of $G$ \cite[Corollary 1.14]{DL} that $T$ is contained in an $F^{(2\mathfrak{n})!}$-stable Borel subgroup of $G$. Hence by (ii), the operator $(F^{(2\mathfrak{n})!})^*$ acts as the multiplication by $q^{(2\mathfrak{n})!\dim G}$ on $H^{2\dim G}_c(G,\mathscr R_{T,\chi}\otimes \mathscr L)$, as desired.
        
    \end{itemize}
   
\end{rmk}

\begin{rmk}\label{rm-passing-to-semisimple}
    The coincidence of (1) and (6) in Theorem \ref{thm-passing-to-semisimple} can be roughly rephrased as follows: the multiplicity of Shalika models with respect to Deligne-Lusztig characters is the ``leading term" of the period summation over semisimple elements.
\end{rmk}

\section{A Computational Approach}\label{sec-computational-approach}

Throughout this section, let $G$ be as introduced in Section \ref{sec-doublecoset}, and $F$ be either $F_{\mathrm {GL}}$ or $F_{\mathrm U}$ introduced in Subsection \ref{subsec-intro-Fro}.

This section is dedicated to computing (6) of Theorem \ref{thm-passing-to-semisimple}; that is, $$
\lim\limits_{i\to \infty}\frac{1}{|G^{F^{n+i\mathfrak{K}}}|}\sum\limits_{s}\left( R_{T,\chi\circ \mathrm N_T^{n+i\mathfrak{K}}}^{(n+i\mathfrak{K})}\cdot \mathrm {Ind}_{\mathrm S^{F^{n+i\mathfrak{K}}}}^{G^{F^{n+i\mathfrak{K}}}}(\psi^{(n+i\mathfrak{K})})\right)(s)
$$ for sufficiently divisible $\mathfrak{K}$, where $s$ runs over semisimple elements of $G^{F^{n+i\mathfrak{K}}}$ in the inner sum. 

In computing this, we establish a Reeder formula for the Shalika model. The procedure is similar to that of \cite{Shi}.

For convenience, we denote $\mathcal X=G/\mathrm S$ (see Section \ref{sec-doublecoset} for the notation).

\subsection{Transporters of semisimple elements}
In this subsection, we retain the notation in Section \ref{sec-doublecoset}. The main reference is \cite[Section 2]{Shi}. 

For any $s\in G(\mathrm k)$, let $\mathfrak T_\circ(s,\mathrm S)$ be the closed subscheme of $G$ representing the  functor via Yoneda embedding sending an algebra $R$ over $\mathrm k$ to
$$
\{g\in G(R) :  g^{-1} s g \in \mathrm S(R) \},
$$
where we view $s$ as an $R$-point via the inclusion $G(\mathrm k) \hookrightarrow G(R)$. Hence, we have a pullback square
$$
\xymatrix{
\mathfrak T_\circ(s,\mathrm S) \ar[r]\ar[d]  & \mathrm S\ar[d]\\
G\ar[r]&G\\
}
$$
where the right vertical arrow is the inclusion and the bottom arrow sends $g$ to $g^{-1} s g$. Let $\mathfrak T(s,\mathrm S)$ be the closed subscheme of $\mathfrak T_\circ(s,\mathrm S)$ with the same underlying topological space and the reduced scheme structure.

\begin{rmk}
    We will not use this Remark in the rest of this paper. In this remark, we fix a semisimple $s\in G(\mathrm k)$. Then we can show that $\mathfrak T_\circ(s,H)$ is smooth. (Consequently, the schemes $\mathfrak T_\circ(s,\mathrm S)$ and $\mathfrak T(s,\mathrm S)$ coincide.)
    We note that the following commutative diagram is cartesian: % (Recall $\mathcal X=G/\mathrm S$)
    $$
\xymatrix{
\mathfrak T_\circ(s,\mathrm S) \ar[r]\ar[d]  & \mathcal X ^{s}\ar[d]\\
G\ar[r]&\mathcal X\\
}
$$
where
\begin{itemize}
    \item $\mathcal X^s$ denotes the fixed-point scheme of $\mathcal X$ under s;
    \item the horizontal maps are the natural quotients;
    \item the vertical maps are the natural inclusions.
\end{itemize}
By \cite[Theorem 13.1]{Mi}, the scheme $\mathcal X^s$ is smooth. The lower horizontal map is visibly smooth, hence the upper horizontal map is likewise smooth. Consequently, we see that $\mathfrak T_\circ(s,\mathrm S)$ is smooth.
\end{rmk}

\subsection{$\mathrm k$-points of $\mathfrak{T}(s,\mathrm S)$} We retain the notation as in the previous subsection. In this subsection, let $T$ be a (not necessarily $F$-stable) maximal torus of $G$. Recall the maxmial torus $U$ of $\mathrm S$ introduced in Definition \ref{def-various}.

For $t\in T(\mathrm k)$, we set $N(t,U,T):=\{g\in G(\mathrm k):t \in gUg^{-1}(\mathrm k), gUg^{-1}\subset T\}$.% Let $G_t$ be the identity component of the centralizer $C_G(t)$.

For semisimple elements $r\in G(\mathrm k)$, let $G_r$ denotes the centralizer $C_G(r)$ for convenience.

The following lemma is a trivial instance of \citep[Lemma 2.1]{Shi}.
\begin{lem}\label{kpt} Let $s\in T(\mathrm k)$.
The set of the $\mathrm k$-points of $\mathfrak T(s,\mathrm S)$ is the image of $G_s(\mathrm k)\times N(s,U,T)\times \mathrm S(\mathrm k)$ under the multiplication of $G(\mathrm k)$.
\end{lem}
\begin{proof}
    For completeness we elaborate on the proof.

It is clear that the image of $G_s(\mathrm k)\times N(s,U,T)\times \mathrm S(\mathrm k)$ under the multiplication is contained in the set of the $\mathrm k$-points of $\mathfrak T(s,\mathrm S)$.
    
   Fix $g\in G(\mathrm k)$. Suppose that we have $g^{-1}sg\in \mathrm S(\mathrm k)$. Since $U$ is a maximal torus of $\mathrm S$, there exists $h\in \mathrm S(\mathrm k)$ such that $h^{-1}g^{-1}sgh\in  U(\mathrm k)$. Hence we have $ghUh^{-1}g^{-1}\subset G_s$. Since $T$ is a maximal torus in $G_s$, there exists $\widetilde{g}\in G_s(\mathrm k)$ such that $\widetilde{g}ghUh^{-1}g^{-1}\widetilde{g}^{-1}\subset T$. (Note that $gh Uh^{-1}g^{-1}$ is a torus in $G_s$.) It is easy to verify that $\wt g g h\in N(s,U,T)$, as desired.
\end{proof}

\subsection{A partition of $T$}\label{subsec-partition-T}
We retain the notation as in the previous subsection. In particular, we fix a (not necessary $F$-stable) maximal torus $T$ of $G$.

Let $J(T)$ be the set of subtori $K$ of $T$ such that $K$ is conjugate to $U$ under $G(\mathrm k)$-conjugation.
\begin{defn}
     For any nonempty subset $\jmath \subset J(T)$, let $T_\jmath$ be the following locally closed subscheme of $T$ $$\bigcap\limits_{K\in\jmath}K-\bigcup_{R\in  J(T) -\jmath}R$$
equipped with the reduced scheme structure. If $\jmath$ is the empty set, we set $T_\jmath=T-\bigcup\limits_{K \in  J(T)} K$.
\end{defn}

We record the following trivial lemma for future use. We denote the power set of $ J(T)$ by $2^{ J(T)}$.
\begin{lem}\label{lem-jus-singleton}
    For nonempty $\jmath\in 2^{J(T)}$, we have $\dim T_\jmath\leq \mathfrak{n}$. The equation holds if and only if $\jmath$ is a singleton set.
\end{lem}
\begin{proof}
    By the obvious symmetry, we may assume $T$ is the torus $\mathrm T$ introduced in Definition \ref{def-various}. The rest of this proof is due to elementary linear algebra, which we omit.
\end{proof}

\begin{defn}\label{def-barnugs}
%Recall that for $s\in T(\mathrm k)$, we denote the identity component of the centralizer $C_G(s)$ by $G_s$.
For $r\in G(\mathrm k)$, we set $N(r,T)$ to be the reduced subscheme of $G$ whose set of  $\mathrm k$-points is $\{g\in G(\mathrm k):g^{-1}r g \in T(\mathrm k)\}$. 
For a semisimple element $s\in G(\mathrm k)$, we define $\bar N(s,T):=G_s\backslash N(s,T)$, and view it as a finite set. 
% Let $\mathcal U(G_s^F)$ be the set of $\Ad(G_s^F)$-orbits of unipotent elements in $G^F_s$. 
 \end{defn}

\begin{defn}\label{def-I(T)}
    Let $I(T)$ be an index set for the set of  subgroups $\{G_t: t \in T(\mathrm k)\}$. For $\iota \in I(T)$, let $G_\iota$ be the corresponding connected centralizer, and we set $T_\iota$ to be the locally closed reduced  subscheme of $T$ whose set of $\mathrm k$-points is
$$\{t\in T(\mathrm k):G_t =G_\iota\}.
$$
\end{defn}

The scheme $N(r,T)$ introduced in Definition \ref{def-barnugs} remain unchanged as $r$ runs over $T_\iota(\mathrm k)$ for a fix $\iota\in I(T)$, and we denote $N(\iota,T):=N(r,T)$ for some/all $r\in T_\iota(\mathrm k)$. Similarly, we set $\bar N(\iota,T):=\bar N(r,T)$ for some/all $r\in T_{\iota}(\mathrm k)$. %Similarly, we set $G_\iota:=G_r$ for some/all $r\in T_\iota (\mathrm k)$.

\begin{defn} For $\jmath\in 2^{J(T)}$ and $\iota\in I(T)$, we set $T_{\jmath,\iota}=T_\jmath \cap T_\iota$ with the reduced scheme structure. 
\end{defn}

The following proposition rephrases \cite[Proposition 2.3]{Shi}.
\begin{prop}\label{pro-T-partition}
The set $\{T_{\jmath,\iota}  \}_{(\jmath,\iota)}$ of reduced locally closed subschemes of $T$ indexed by the set $2^{ J(T)} \times I(T)$ forms a finite partition of $T$. For each pair $(\jmath,\iota)\in 2^{ J(T)}\times I(T)$, the following statement holds:\\
For any $t_1,t_2 \in T_{\jmath,\iota}(\mathrm k)$, we have $$\mathfrak T(t_1,\mathrm S)=\mathfrak T(t_2,\mathrm S).$$
\end{prop}
\begin{proof}
    For completeness, we elaborate on the proof. The first statement is clear. We prove the last statement.

Fix $(\jmath,\iota)\in 2^{J(T)}\times I(T)$ and $t_1,t_2\in T_{\jmath,\iota}$.   
Since $\mathfrak T(t_1,\mathrm S)$ and $\mathfrak T(t_2,\mathrm S)$ are reduced subschemes of $G$, we need to show
$$\mathfrak T(t_1,\mathrm S) (\mathrm k)=\mathfrak T(t_2,\mathrm S)(\mathrm k);
$$
that is, their sets of $\mathrm k$-points coincide.

Thanks to Lemma \ref{kpt}, it remains to show $G_{t_1}=G_{t_2}$ and $N(t_1,U,T)=N(t_2,U,T)$. By definition we have $G_{t_1}=G_{t_2}=G_\iota$.  We see $N(t_1,U,T)=\bigcup\limits_{K\in \jmath} N_K=N(t_2,U,T)$, where $N_K:=\{g\in G(\mathrm k):g Ug^{-1}=K\}$. This completes the proof.
\end{proof}

\begin{rmk}\label{rm-F-stable-parition}
Suppose in this remark  that $T$ is $F$-stable.
Then the Frobenius $F$ acts naturally on $ J(T)$, which gives rise to an action on $2^{ J(T)}$. This action satisfies that $F(T_\jmath)=T_{F(\jmath)}$ for $\jmath \in 2^{ J(T)}$. Similarly, the Frobenius $F$ acts naturally on $I(T)$.  We have $F(T_\iota)=T_{F(\iota)}$ and $F(T_{\jmath,\iota})=T_{F(\jmath),F(\iota)}$ for $\iota \in I(T)$ and $\jmath \in  J(T)$. The set of $F$-invariant elements in $2^{ J(T)}$ (resp. $I(T)$) is denoted by $2^{\mathcal J(T),F}$ (resp. $I(T)^F$).
\end{rmk}

\begin{defn}\label{def-X-jmath}
   Fix $(\jmath,\iota)\in 2^{J(T)}\times I(T)$. We set $\mathfrak{T}(\jmath,\mathrm S):=\mathfrak{T}(t,\mathrm S)$ for some/all $t\in T_\jmath(\mathrm k)$. Similarly, we set $\mathcal X^{\jmath}:=\mathcal X^s$  for some/all $s\in T_\jmath(\mathrm k)$, where $\mathcal X^s$ denotes the fixed-point scheme of $\mathcal X$ by $s$. (To justify the notation $\mathcal X^{\jmath}$, we note: for $s
   _1,s_2\in T_\jmath(\mathrm k)$, the schemes $\mathcal X^{s_1}$ and $\mathcal X^{s_2}$ are smooth by \cite[Theorem 13.1]{Mi}, and their set of $\mathrm k$-points coincide by Proposition \ref{pro-T-partition}; hence $\mathcal X^{s_1}$ and $\mathcal X^{s_2}$ coincide.)
  % We verify that they are well-defined by Proposition \ref{pro-T-partition}.
\end{defn}

\subsection{Character values at semisimple elements}
In this subsection we retain the notation as in the previous subsection. We further assume that $T$ is $F$-stable throughout this subsection.

\begin{defn}\label{def-Green-function}
    Fix a reductive group $K_0$ over $\mathbb F_q$ and let $K$ be its pullback to $\mathrm k$. For an $F$-stable maximal torus $S$ of $K$, let $Q_{S}^K$ denote the Green function in the sense of \cite{DL}. Hence $Q_S^K$ is by definition  the restriction of the Deligne-Lusztig character $R_{S,1}^K$ to the set of unipotent elements in $K^F$.

    Fix a positive integer $n$. We set $Q_{S}^{K,(n)}$ to be the restriction of $R_{S,1}^{(n)}$ to the set of unipotent elements in $K^{F^n}$, where $R_{S,1}^{(n)}$ is the virtual character of $K^{F^n}$ introduced at the end of Subsection \ref{subsec-DL} for the reductive group $G_0=K_0$ and the trivial character $\chi=1$.
\end{defn}

The following proposition rephrases \cite[Theorem 4.2]{DL}.
\begin{prop}\label{pro-value-DL}Let $\chi:T^F\to \Qlb^\times$ be a character. The following equation holds for a Jordan decomposition $g=su\in G^F$, where $s$ (resp. $u$) is semisimple (resp. unipotent). (In the following, the element $\gamma$ is a representative of the corresponding $\bar \gamma$.)
    \begin{equation}
R_{T,\chi}^G(su)= \sum_{\bar \gamma \in \bar N(s,T)^F} \chi(\gamma^{-1}s \gamma) Q_{\gamma T\gamma^{-1}}^{G_s}(u).
\end{equation}
In particular, for any semisimple element $t\in G^F$, we have 
\begin{equation}
R_{T,\chi}^G(t)= \sum_{\bar \gamma \in \bar N(t,T)^F} \chi(\gamma^{-1}t \gamma) Q_{\gamma T\gamma^{-1}}^{G_s}(e).
\end{equation}
\end{prop}

Recall the character $\mathrm {Ind}_{\mathrm S^F}^{G^F}(\psi_{\mathrm S}^{(1)}):G^F\to \Qlb$ introduced in Remark \ref{rm-shasheafvalue}. 
\begin{lem}\label{lem-shasheaf-value-semisimple}
    Let $t\in G^F$ be a semisimple element. Then we have the following equation.
    $$
    \mathrm {Ind}_{\mathrm S^F}^{G^F}(\psi_{\mathrm S}^{(1)})(t)=|\mathfrak{T}(t,\mathrm S)^F|/|\mathrm S^F|=|(\mathcal X^t)^F|,
    $$
    where $\mathfrak{T}(t,\mathrm S)^F$ denotes the fixed-point set of $\mathfrak{T}(t,\mathrm S)$ under $F$.
\end{lem}
\begin{proof}
Combined with Lang's theorem (note that $\mathrm S$ is connected), an immediate computation reduces us to show: for each $x\in \mathfrak{T}(t,\mathrm S)$, the image of $x^{-1}tx$ under $\mathfrak{c}:\mathrm S\to \mathbb G_a$ is trivial. We fix $x\in \mathfrak{T}(t,\mathrm S)$ in the rest of this proof. 

    Since $x^{-1}tx$ is semisimple, it is contained in some maximal torus $U'$ of $\mathrm S$. The restriction of $\mathfrak{c}$ to $U'$ must be trivial by general nonsense (this restriction is an algebraic group homomorphism from a torus to a unipotent group). This in particular shows that $x^{-1}tx$ has trivial image under $\mathfrak{c}$, as desired.
\end{proof}

\subsection{Summation over $G^F$}
In this subsection, we reformulate Section 5.1 of \cite{R}. 

Suppose we have a function $f: G^F \to \mathbb C$, invariant  under conjugation by $G^F$, with the property that
\begin{itemize}\label{fp}
\item[(*)] if $g \in G^F$ has Jordan decomposition $g=su$, then $f(g)=0$ unless the conjugacy class $\Ad(G^F)\cdot s$ meets $T^F$. (Here, the element $s\in G$ is semisimple and $u\in G$ is unipotent.)
\end{itemize}

 For $s\in T^F$, we see that the fixed-point set  of $\bar N(s,T)$ (see Definition \ref{def-barnugs}) under the Frobenius $F$ is $G_s^F \bs N(s,T)^F$. The following is (5.2) of \cite{R}.
\begin{prop}\label{prop-summationf}Let $f: G^F \to \mathbb C$ be a class function satisfying (*) required above.  For semisimple $s\in G^F$, let $\mathcal U(G_s^F)$ be the set of $\Ad(G_s^F)$-orbits of unipotent elements in $G^F_s$.
The following equation holds:
 $$
\frac{1}{|G^F|}\sum_{g\in G^F}f(g)=\sum_{s\in T^F}\frac{1}{|\bar N(s,T)^F|}\sum_{[u]\in \mathcal U(G_s^F)} \frac{1}{|C_{G_s}(u)^F|}f(su).
$$
\end{prop}

The same argument yields the following corollary.

\begin{cor}\label{cor-summation-semisimple}
    Let $f$ be as in  Proposition \ref{prop-summationf}. The following equation holds:
 $$
\frac{1}{|G^F|}\sum_{t\in G^F ~\text{ss.}}f(t)=\sum_{s\in T^F}\frac{1}{|\bar N(s,T)^F|\cdot|G_s^F|}f(s),
$$
where the subscript of $\sum\limits_{t\in G^F ~\text{ss.}}$ means that $t$ runs over the set of semisimple elements in $G^F$. 
\end{cor}

\begin{rmk}\label{rm-f-class-property}
We see that 
the function $f_0 \cdot R_{T,\chi}^G$ has the property (*) required at the beginning of this subsection  for any $F$-stable maximal torus $T$ of $G$, any character $\chi: T^F \to \Qlb^\times$ and any class function $f_0$ of $G^F$. This can be seen from Proposition \ref{pro-value-DL}.
\end{rmk}

\subsection{The summation in question}\label{subsec-summation-in-ques}In this subsection, we retain the notation as in Subsection \ref{subsec-passing-to-semisimple}. In particular, we fix an $F$-stable maximal torus $T$ of $G$ and a character $\chi:T^F\to \Qlb^\times$.

\begin{thm}\label{thm-summation-main}
    Fix a positive integer $n$. We have the following equation. (The reader is advised to review Subsection \ref{subsec-partition-T} and Subsection \ref{subsec-passing-to-semisimple} for the notation.)
    \begin{align*}
         &\frac{1}{|G^{^{F^n}}|} \sum_{t\in G^{F^n} \text{ss.}}\left(R_{T,\chi\circ \mathrm N_{T}^n}^{(n)}\cdot \mathrm {Ind}_{\mathrm S^{F^n}}^{G^{F^n}}(\psi_{\mathrm S}^{(n)})\right)(t)\\
   &=\sum_{(\jmath,\iota)\in 2^{J(T),F^n}\times I(T)^{F^n}}\sum_{\bar \gamma \in \bar N(\iota,T)^{F^n}}\sum_{s\in T_{\jmath,\iota}^{F^n}}\frac{\chi\circ \mathrm N_T^n(\gamma^{-1}s \gamma)\cdot Q_{\gamma T\gamma^{-1}}^{G_\iota,(n)}(e)\cdot |(\mathcal X^{\jmath})^{F^n}|}{|\bar N(\iota,T)^{F^n}|\cdot|G_\iota^{F^n}|},
    \end{align*}
    where the subscript of $
    \sum\limits_{t\in G^{F^n} \text{ss.}}
    $ means that $t$ runs over the set of semisimple elements in $G^{F^n}$.
\end{thm}
\begin{proof}
    In view of Proposition \ref{pro-value-DL} and Lemma \ref{lem-shasheaf-value-semisimple} (and their corresponding versions for $F$ replaced by powers), the current theorem follows from Corollary \ref{cor-summation-semisimple} (and its corresponding version for $F$ replaced by powers) and the partition of $T$ introduced in Subsection \ref{subsec-partition-T}.
\end{proof}

\begin{rmk}
    Theorem \ref{thm-summation-main} is a priori complicated. The beneficial point is that: it is adaptive to the limit procedure (Theorem \ref{thm-passing-to-semisimple}), as we will show in the subsequent subsections. 
\end{rmk}

\subsection{Dimension estimation}In this subsection, we retain the notation as in Subsection \ref{subsec-partition-T}. In particular, we fix a (not necessarily $F$-stable) maximal torus $T$ of $G$.

Let $\mathcal B$ be the flag variety of $G$. For each Borel subgroup $B$ of $G$, there exists a canonical isomorphism $\mathcal B\cong G/B$. In the rest of this subsection, we fix a Borel subgroup $B$ containing the torus $T$.

\begin{defn}\label{def-flag-fixed-s}
    For semisimple $s\in G(\mathrm k)$, let $\mathcal B_s$ be the subscheme of $\mathcal B$ fixed by $s$.
\end{defn}

\begin{rmk}\label{rm-qcomponent}
    We may rephrase Proposition 4.4 of \cite{DL} as follows. (See Definition \ref{def-barnugs} for the definition of $\bar N(s,T)$)
    \begin{itemize}
        \item For each  $\bar \gamma \in \bar N(s,T)$, there exists an irreducible component $\mathcal Q_{\bar \gamma}^s$ (defined over $\mathrm k$) of $\mathcal B_{s}$ consisting of points of the form $gB\in G/B\cong \mathcal B$, such that $g\in N(s,T)$ is a representative of $\bar \gamma$.
        \item The scheme $\mathcal B_{s}$ (as a scheme over $\mathrm k$) is the disjoint union of its irreducible components $\mathcal Q_{\bar \gamma}^s$ parameterized by $\bar \gamma \in \bar N(s,T)$.
    \end{itemize}
\end{rmk}

Consequently, the following is well-defined.

\begin{defn}\label{def-flag-fixed-s}
    Fix $\iota\in I(T)$.\begin{itemize}
        \item Let $\mathcal B_\iota$ be the subscheme of $\mathcal B$ fixed by some/all $s\in T_\iota(\mathrm k)$;
        \item For each  $\bar \gamma \in \bar N(s,T)$,  let $\mathcal Q_{\bar \gamma}^\iota$ be the irreducible component $\mathcal Q_{\bar \gamma}^s$ introduced in  Remark \ref{rm-qcomponent} for some/all $s\in T_\iota(\mathrm k)$.
    \end{itemize}   
\end{defn}

\begin{defn}\label{def-conjugacy-map}
    Fix $(\jmath,\iota)\in 2^{J(T)}\times I(T)$. Let $\mathrm c_{\jmath,\iota}:G\times T_{\jmath,\iota}\times  \mathcal B_\iota \times \mathcal X^\jmath\to \mathcal Y_{T,B}$ be the map given by $(g,t,xB,y\mathrm S)\mapsto(gtg^{-1},gxB,gy\mathrm S)$. 
    (See Definition \ref{def-X-jmath} for the definition of $\mathcal X^\jmath$, see Definition \ref{def-YTB} for $\mathcal Y_{T,B}$.)
\end{defn}

\begin{lem}\label{lem-fibre-dim-conj}
    Fix $(\jmath,\iota)\in 2^{J(T)}\times I(T)$. The fibre of $\mathrm c_{\jmath,\iota}$ is either empty or of dimension $\dim G_\iota$. (See Definition \ref{def-I(T)} for the definition of $G_\iota$.)
\end{lem}
\begin{proof}
    Fix $p=(s,xB,y\mathrm S)\in \mathcal Y_{T,B}(\mathrm k)$. It is immediate from Definition \ref{def-conjugacy-map} that the fibre $\mathrm c_{\jmath,\iota}^{-1}(\{p\})$ is isomorphic to the disjoint union of $m$ copies of $G_\iota$, where $m$ is the cardinality of the set
    $$
    \{t\in T_{\jmath,\iota}(\mathrm k):\text{$t$ is $G(\mathrm k)$-conjugate to $s$}\}.
    $$
    This completes the proof. (Note that $m$ must be finite.)
\end{proof}

\begin{prop}\label{pro-index-injection}
    Fix $(\jmath,\iota)\in 2^{J(T)}\times I(T)$. Then:
    \begin{itemize}
        \item[(i)]  we have $\dim T_{\jmath,\iota}+\dim \mathcal B_\iota+\dim \mathcal X^\jmath \leq\dim G_\iota$;
        \item[(ii)] if the equation in (i) holds, then $\jmath=\{R\}$ is a singleton set  and $T_{\jmath,\iota}$ is an open  dense subscheme of $R$.  (Here, the torus $R$ is a $G(\mathrm k)$-conjugation of $U$ contained in $T$ by the definition of $J(T)$.)
    \end{itemize}
\end{prop}

\begin{proof}
    By Lemma \ref{lem-fibre-dim-conj}, the fibre of $\mathrm c_{\jmath,\iota}$ is either empty or of dimension $\dim G_\iota$. The target of $\mathrm c_{\jmath,\iota}$ is of dimension $\dim G$ by Corollary \ref{cor-dimYTB}. Consequently, we have $$\dim G+\dim G_\iota \geq \dim G+\dim T_{\jmath,\iota}+\dim \mathcal B_\iota +\dim \mathcal X^\jmath,$$ yielding (i).

Before proceeding, we introduce some notation.
For $w\in B(\mathrm k)\backslash G(\mathrm k)/\mathrm S(\mathrm k)$, recall the scheme $\mathcal Y_{T,B}^w$ introduced in Definition \ref{def-YTbw}. We set $\mathrm d_{\mathcal Y}:\mathcal Y_{T,B}\to T$ to be the morphism given by $(g,xB,y\mathrm S)\mapsto \mathrm d_T^B(x^{-1}gx)$. For $w\in B(\mathrm k)\backslash G(\mathrm k)/\mathrm S(\mathrm k)$, let $\mathrm d_w:\mathcal Y_{T,B}^w\to T$ be the restriction of $\mathrm d_{\mathcal Y}$. 
It is easy to verify that $\mathrm d_w$ maps $\mathcal Y_{T,B}^w$ onto $\mathrm d_T^B(B\cap w\mathrm Sw^{-1})$. Note that $\mathrm d_T^B(B\cap w\mathrm Sw^{-1})$ is a subtorus of $T$ that is conjugate to $U$: that is, we have $\mathrm d_T^B(B\cap w\mathrm Sw^{-1})\in J(T)$. Note that $\dim \left( \mathrm{d}_T^B (B\cap w\mathrm Sw^{-1})\right)=\mathfrak{n}.$

    By the  proof of (i), we see that the equation in (i) holds if and only if the image of $\mathrm c_{\jmath,\iota}$ is of dimension $\dim G$. Therefore, by the partition (\ref{eq-partitionYTB}) and Lemma \ref{lem-dim-YTBw}, the image of $\mathrm c_{\jmath,\iota}$  contains the generic point of $\mathcal Y_{T,B}^w$ for some $w\in B(\mathrm k)\backslash G(\mathrm k)/ \mathrm S(\mathrm k)$. We fix one such $w$ in the rest of this proof.
    Combining Lemma \ref{lem-dim-YTBw} and the fact that $\mathcal Y_{T,B}^w$ has the image of dimension $\mathfrak{n}$ under $\mathrm d_w$, we see that the image of the composite $\mathrm d_\mathcal Y\circ \mathrm c_{\jmath,\iota}$ is of dimension $\geq\mathfrak{n}$. Unwinding the definition, we see that $\mathrm d_\mathcal Y\circ \mathrm c_{\jmath,\iota}:G\times T_{\jmath,\iota}\times \mathcal B_\iota\times \mathcal X^\jmath\to T$ is given by $(g,t,xB,y\mathrm S)\mapsto t$. Consequently, we see that $\dim T_{\jmath}\geq \dim T_{\jmath,\iota}\geq \mathfrak{n}$. Note that $\mathcal X^\jmath=\mathfrak{T}(\jmath,\mathrm S)/\mathrm S$ is nonempty, indicating that $\jmath$ is nonempty by Lemma \ref{kpt}. We see that $\mathfrak{n}\geq \dim T_{\jmath}$ by Lemma \ref{lem-jus-singleton}. Hence $\dim T_\jmath=\dim T_{\jmath,\iota}=\mathfrak{n}$.
    Again by Lemma \ref{lem-jus-singleton}, we see that $\jmath$ is a singleton. Say that $\jmath=\{R\}$, where $R$ is a $G(\mathrm k)$-conjugate of $U$ contained in $T$. Since we have that $\dim R=\mathfrak{n}$, and that $T_{\jmath,\iota}$ is a locally closed subscheme of $R$, we conclude that $T_{\jmath,\iota}$ is a dense open subscheme of $R$.
\end{proof}

\begin{defn}\label{def-Delta-T}
    Let $\Delta_T$ be the set consisting of triples $(\jmath,\iota,C)$ satisfying the following conditions:
    \begin{itemize}
        \item $(\jmath,\iota)\in 2^{J(T)}\times I(T)$;
        \item $C$ is an irreducible component of $\mathcal X^\jmath$ such that $\dim T_{\jmath,\iota}+\dim B_\iota+\dim C=\dim G_\iota$.
    \end{itemize}
    If we assume further that $T$ is $F$-stable, then $\Delta_T$ has been endowed with a natural action of the Frobenius, given by $(\jmath,\iota,C)\mapsto (F(\jmath),F(\iota),F(C))$. Under such circumstance, we denote the fixed-point set by $\Delta_T^F$. (See Remark \ref{rm-F-stable-parition}.)
\end{defn}

\begin{rmk}Fix $(\jmath,\iota,C)\in \Delta_T$.
    It is immediate from Proposition \ref{pro-index-injection} that the closure of $T_{\jmath,\iota}$ in $T$ is a $G(\mathrm k)$-conjugate of $U$.  
\end{rmk}

\begin{defn}
    For $\delta=(\jmath,\iota,C)\in \Delta_T$, let $T_{\delta}$ be the closure of $T_{\jmath}$. By Proposition \ref{pro-index-injection} and the above remark, we have $T_\delta\in J(T)$ and $\jmath=\{T_\delta\}$.
\end{defn}

 %For $w\in B(\mathrm k)\backslash G(\mathrm k)/\mathrm S(\mathrm k)$, let $\mathrm o(w)$ be the $B$-orbit $Bw\mathrm S$ on $\mathcal X=G/\mathrm S$. 
 For $\delta=(\jmath,\iota,C)\in \Delta_T$, let $\mathrm I_\delta$ denote the set-theoretic image of $K_\delta:=G\times T_{\jmath,\iota}\times \mathcal Q_1^\iota\times C$ under $\mathrm c_{\jmath,\iota}$. Recall the set $\Phi_B$ introduced in Definition \ref{def-PhiB}.
\begin{cor}\label{cor-index-injection}Let the notation be as above.
    We have a injection $\mathrm b_{T,B}:\Delta_T\hookrightarrow \Phi_B$ characterized by the following property:
    \begin{itemize}
        \item $\delta$ maps to $w$ if and only if $\mathrm I_\delta\cap \mathcal Y_{T,B}^w$ is dense in $\mathcal Y_{T,B}^w$.
    \end{itemize}
\end{cor}
\begin{proof}
We first show that such a map is well-defined. Fix $\delta =(\jmath,\iota,C)\in \Delta_T$ in this paragraph. By Lemma \ref{lem-fibre-dim-conj} and the Definition of $\Delta_T$, we see that the closure of $\mathrm I_\delta$ is irreducible and of dimension $\dim G$. Consequently by Lemma \ref{lem-dim-YTBw}, there exists a unique $w\in B(\mathrm k)\backslash G(\mathrm k)/\mathrm S(\mathrm k)$ such that $\mathrm I_\delta\cap \mathcal Y_{T,B}^w$ is dense in $\mathcal Y_{T,B}^w$. We want to show that $w\in \Phi_B$. Take a $\mathrm k$-point $p=(g,xB,y\mathrm S)\in \mathrm I_\delta\cap \mathcal Y_{T,B}^w$ such that $I_\delta\cap \mathcal Y_{T,B}^w$ contains an open neighborhood $V\subset \mathcal Y_{T,B}^w$ of $p$. Given the definition of $\mathrm c_{\jmath,\iota}$, we see that a general point $b\in xBx^{-1}\cap y\mathrm S y^{-1}$ is $G(\mathrm k)$-conjugate to an element in $T_{\jmath,\iota}$. By Remark \ref{rm-phi-self-B}, we see that a general point of $xBx^{-1}\cap y\mathrm S y^{-1}$ is not semisimple if $w\notin \Phi_B$. This shows that $w\in \Phi_B$, as desired. In particular, we have the well-defined map $\mathrm b_{T,B}$ characterized by the property displayed in our corollary.

It remains to show that $\mathrm b_{T,B}$ is an injection. Fix $\delta_1 =(\jmath_1,\iota_1,C_1)\in \Delta_T$, and $\delta_2 =(\jmath_2,\iota_2,C_2)\in \Delta_T$ in this paragraph. Assume that $\mathrm b_{T,B}(\delta_1)=\mathrm b_{T,B}(\delta_2)$. Let $q=(h,iB,j\mathrm S)\in \mathrm I_{\delta_1}\cap I_{\delta_2}$. 
Say, that $q=\mathrm c_{\jmath,\iota}((g_1,a_1,x_1B,y_1\mathrm S))=\mathrm c_{\jmath,\iota}((g_2,a_2,x_2B,y_2\mathrm S))$ for $(g_1,a_1,x_1B,y_1\mathrm S)\in K_{\delta_1}$ (resp. $(g_2,a_2,x_2B,y_2\mathrm S)\in K_{\delta_2}$).
Then it is immediate to see that $a_1=a_2=\mathrm d_T^B(i^{-1}hi)\in T_{\jmath_1,\iota_1}\cap T_{\jmath_2,\iota_2}$. This yields $(\jmath_1,\iota_1)=(\jmath_2,\iota_2)$ since $\{T_{\jmath,\iota}\}$ forms a partition of $T$ by Proposition \ref{pro-T-partition}. This in turns show that we may assume $x_1,x_2\in G_{\iota_1}=G_{\iota_2}$ by the definition of $\mathcal Q_1^\iota$ (Definition \ref{def-flag-fixed-s}). (We denote $G_\iota:=G_{\iota_1}=G_{\iota_2}$ in the rest of this proof.)
Consequently, we see that: 
\begin{itemize}
    \item $g_1x_1B=g_2x_2B$;
    \item $g_1g_2^{-1}$ centralizes $a_1=a_2$.
\end{itemize}
    This shows that $g_1g_2^{-1}\in G_\iota$ by Remark \ref{rm-qcomponent}. As $g_2^{-1}g_1y_1\mathrm S=y_2\mathrm S$, we see that $C_1=C_2$, since it is clear by definition that $C_1$ and $C_2$ are stable under the action of the connected algebraic group $G_\iota$.
\end{proof}

\begin{rmk}\label{rm-injectT}
Fix distinct $w_1,w_2\in \Phi_B$. Recall the map $\rho\circ \mathrm{pr}_1:\mathcal Y_{T,B}\to T$ used in Proposition \ref{pro-badmodel}, which is explicitly given by $(g,xB,y\mathrm S)\mapsto \mathrm d_{T}^B(x^{-1}gx)$.
   We claim that the images of $\mathcal Y_{T,B}^{w_1}$ and $\mathcal Y_{T,B}^{w_2}$ under $\pi\circ \mathrm{pr}_1$ are distinct. 

   In proving this claim, we may assume that $B$ is the Borel group $\mathrm B^{op}$ introduced in Definition \ref{def-various} by the obvious symmetry. Then this claim follows from Proposition \ref{pro-bijec-phi} and a straightforward calculation.
\end{rmk}

\begin{cor}\label{cor-Tdelta-vs-delta}
    Let $\delta_1$ and $\delta_2$ be distinct elements in $\Delta_T$. Then the tori $T_{\delta_1}$ and $T_{\delta_2}$ are likewise distinct.
\end{cor}
\begin{proof}Let $\bar {\mathrm I}_{\delta_1}$ be the closure of $\mathrm I_{\delta_1}$, which is the underlying set of the scheme-theoretic image of $K_{\delta_1}$ under $\mathrm c_{\jmath,\iota}$. Consequently $\bar {\mathrm I}_{\delta_1}$ is irreducible and of dimension $\dim G$. Corollary \ref{cor-index-injection} shows that $\mathcal Y_{T,B}^{\mathrm b_{T,B}(\delta_1)}$ contains a nonempty open subset $V_1$ of $\bar {\mathrm I}_{\delta_1}$.

Let $w\in G(\mathrm k)$ be a representative of $\mathrm b_{T,B}(\delta_1)$. Let $\mathrm f:V_1\to T$ be the map given by $(g,xB,y\mathrm S)\mapsto \mathrm d_T^B(x^{-1}gx)$, which is the restriction of $\rho\circ \mathrm {pr}_1$.
Note that the composite $K_\delta\xrightarrow[]{\mathrm c_{\jmath,\iota}}\mathcal Y_{T,B}\xrightarrow{\rho\circ \mathrm{pr}_1}T$ equals the projection $K_\delta=G\times T_{\jmath,\iota}\times \mathcal Q_1^\iota\times C\to T_{\jmath,\iota}\hookrightarrow T$.

By the above argument, the image of $V_1$ under the map $\mathrm f$ is dense in $T_{\jmath,\iota}$. This procedure ensures that $T_{\delta_1}$ is indeed determined by $\mathrm b_{T,B}(\delta_1)$.  
In particular, this manifests that $T_{\delta_1}$ and $T_{\delta_2}$ are distinct by Remark \ref{rm-injectT}. 
\end{proof}

\begin{rmk}
    Latter (in Proposition \ref{pro-bij-delta} below) we will see that $\mathrm b_{T,B}$ is indeed a bijection.
\end{rmk}

\subsection{The limit in question}In this subsection, we retain the notation as in Subsection \ref{subsec-summation-in-ques}. In particular, we fix an $F$-stable maximal torus $T$ of $G$ and a character $\chi:T^F\to \Qlb^\times$. Note that $T$ is the pullback to $\mathrm k$ of a torus $T_0$ over $\mathbb F_q$.

The following lemma is trivial.
\begin{lem}\label{lem-the-gap}
    There exists a positive integer $\mathfrak{N}$ satisfying the following properties:
    \begin{itemize}
        \item $F^{\mathfrak{N}}$ acts trivially on $J(T)$ and $I(T)$;
        \item $F^{\mathfrak{N}}$ acts trivially on $\bar N(\iota,T)$ for all $\iota \in I(T)^F$;
        \item $F^{\mathfrak{N}}$ acts trivially on the set of irreducible components of $\mathcal X^\jmath$ for all $\jmath \in 2^{J(T),F}$;
        \item $T$ is $F^{\mathfrak{N}}$-split: that is, the torus $T_0\times_{\Spec \mathbb F_q}\Spec \mathbb F_{q^\mathfrak{N}}$ is split over $\mathbb F_{q^\mathfrak{N}}$;
        \item Fix any $\iota \in I(T)^F$. The number $\sigma_{1+i\mathfrak{N}}(G_\iota)$ remains constant as $i$ varies in the set of positive integers, where  $\sigma_{1+i\mathfrak{N}(G_\iota)}$ denote the $F^{1+i\mathfrak{N}}$-rank of $G_\iota$.
    \end{itemize}
\end{lem}

\begin{rmk}\label{rm-green-limit}
    Fix a positive integer $\mathfrak{N}$ as in Lemma \ref{lem-the-gap}. Fix $\iota \in I(T)^F$ and $\bar \gamma \in \bar N(\iota,T)^F$. It follows from \cite[Theorem 7.1]{DL} that we have the following equation:
    $$
    \lim_{i\to \infty} \frac{Q_{\gamma T\gamma^{-1}}^{G_\iota,(1+i\mathfrak{N})}(e)}{q^{(1+i\mathfrak{N})\dim \mathcal B_\iota}}=(-1)^{\sigma(G_\iota)+\sigma(T)}.
    $$
\end{rmk}

The following two lemmas are trivial instances of Grothendieck trace formula and Deligne's theory on weight (apply to the constant sheaves ).

\begin{lem}\label{lem-lim-X}
    Fix a positive integer $\mathfrak{N}$ as in Lemma \ref{lem-the-gap}. Fix $\iota \in I(T)^F$   and $\jmath\in 2^{J(T),F}$. Then we have
 \[
    \lim_{i\to \infty} \frac{|(\mathcal X^\jmath)^{F^{1+i\mathfrak{N}}}|}{q^{(1+i\mathfrak{N})\dim \mathcal X^j}}=\mathcal C(\mathcal X^{\jmath}),\] where $\mathcal C(\mathcal X^{\jmath})$ denote the number of $F$-stable irreducible components of $\mathcal X^\jmath$ with dimension $\dim \mathcal X^\jmath$.
\end{lem}

\begin{lem}\label{lemn-lim-Giota}
    Fix $\iota \in I(T)^F$. Then we have
    \[
    \lim_{i\to \infty} \frac{|G_\iota^{F^i}|}{q^{i\dim G_\iota}}=1.
    \]
\end{lem}

%We fix one such $\mathfrak{ N}$ as in Lemma \ref{lem-the-gap} in the rest of this subsection.

\begin{lem}\label{lem-lim-chara}
     Let $C$ be an $F$-stable open dense subscheme of an $F$-stable subtorus $R$ of $T$. For each positive integer $i$, we set $M_{i}:=\sum\limits_{s\in C^{F^i}}\chi\circ \mathrm N_T^i(s)$.
      Then we have
    \begin{equation*}
        \lim\limits_{i\to \infty}\frac{M_{i}}{q^{i\dim R}}=\left\{ \begin{aligned}
&0& \text{if the~restriction~of~$\chi$~to ~$R^F$~ is~ nontrivial.}\\
&1 &\text{otherwise.}
\end{aligned}\right.
    \end{equation*}       
\end{lem}
\begin{proof}
    Recall the sheaf $\mathscr L_\chi$ introduced in Subsection \ref{subsec-sheavesontori}. It follows from Grothendieck trace formula that:
    $$
    M_{i}=\mathrm {Tr}((F^{i})^*,H_c^\bullet(C,i_{C}^*\mathscr L_\chi)),
    $$
    where $i_{C}:C\hookrightarrow T$ is the inclusion.
    Since $C$ is an open dense subscheme of the $\dim R$-dimesional torus $R$ by assumption, the cohomology spaces $H^{2\dim R}_c(R,i_{R}^*\mathscr L_\chi)$ and $H_c^{2\dim R}(C,i_{C}^*\mathscr L_\chi)$ are canonically identified, where $i_R:R\hookrightarrow T$ is the inclusion. In view of Grothendieck trace formula and Deligne's theory on weight \cite[Th\'eor\`eme 3.3.1]{D}, we see that for all positive integer $i$
    \[
    \lim\limits_{i\to \infty}\frac{M_{i}}{q^{i\dim R}}=\lim\limits_{i\to \infty}\frac{N_{i}}{q^{i\dim R}},
    \]
    where $N_{i}:=\sum\limits_{s\in R^{F^i}}\chi\circ \mathrm N_T^i(s)$. By the explicit description of $\mathscr L_\chi$ and $K_{\mathscr L}^n$ introduced in Subsection \ref{subsec-sheavesontori}, we see that for each positive integer $i$
    \begin{equation*}
      N_{i}  =\left\{ \begin{aligned}
&0& \text{if the~restriction~of~$\chi$~to ~$R^F$~ is~ nontrivial.}\\
&|R^{F^i}| &\text{otherwise.}
\end{aligned}\right.
    \end{equation*}  
    This completes the proof.
\end{proof}

\begin{defn}
    For $\delta=(\jmath,\iota,C)\in \Delta_T^F$, we set $\mathrm t_{\delta,\chi}=\langle\chi_{|T_{\delta}^F},1_{T_\delta^F}\rangle_{T_\delta^F}$, where $\chi_{|T_{\delta}^F}$ denotes the restriction of $\chi$ to $T_{\delta}^F$.
\end{defn}

For a positive integer $n$, let $\mathcal P_ n$ be the arithmetic progression $\{1+(i-1)n\}_{i\in \mathbb Z_+}$.

\begin{thm}\label{thm-limit-main}Fix a positive integer $\mathfrak{N}$ as in Lemma \ref{lem-the-gap}. We have
    $$\lim_{\mathcal P_{\mathfrak{N}}\ni n\to \infty}\frac{1}{|G^{^{F^n}}|} \sum_{t\in G^{F^n} \text{ss.}}\left(R_{T,\chi\circ \mathrm N_{T}^n}^{(n)}\cdot \mathrm {Ind}_{\mathrm S^{F^n}}^{G^{F^n}}(\psi_{\mathrm S}^{(n)})\right)(t)=\sum_{\delta\in \Delta_T^F}(-1)^{\sigma(T)+\sigma(G)}t_{\delta,\chi},$$
    where $\lim\limits_{\mathcal P_{\mathfrak{N}}\ni n\to \infty}$ means that $n\in \mathcal P_{\mathfrak{N}}$ tends to infinity, and $\sum\limits_{t\in G^{F^n} \text{ss.}}$ means that $t$ runs over the set of semisimple elements of $G^{F^n}$.
\end{thm}

\begin{proof}
Fix $\iota \in I(T)^F$   and $\jmath\in 2^{J(T),F}$ in this paragraph. Set $M_{\jmath,\iota,\chi,i}:=\sum\limits_{s\in T_{\jmath,\iota}^{F^i}}\chi\circ \mathrm N_T^i(s)$. Then it is elementary to see that $\limsup\limits_{i\to \infty}\frac{M_{\jmath,\iota,\chi,i}}{q^{i\dim T_{\jmath,\iota}}}$ and $\liminf\limits_{i\to \infty}\frac{M_{\jmath,\iota,\chi,i}}{q^{i\dim T_{\jmath,\iota}}}$  are well-defined  by Lemma \ref{lem-tpduality}.
    By  our choice of $\mathfrak{N}$, Lemma \ref{lem-lim-X}  \ref{lemn-lim-Giota} \ref{lem-tpduality} and Remark \ref{rm-green-limit}, Proposition \ref{pro-index-injection} (i), the inner sum in Theorem \ref{thm-summation-main} $$
    \sum_{s\in T_{\jmath,\iota}^{F^n}}\frac{\chi\circ \mathrm N_T^n(\gamma^{-1}s \gamma)\cdot Q_{\gamma T\gamma^{-1}}^{G_\iota,(n)}(e)\cdot |(\mathcal X^{\jmath})^{F^n}|}{|\bar N(\iota,T)^{F^n}|\cdot|G_\iota^{F^n}|}$$ has a well-defined limit superior (resp. inferior) as $\mathcal P_\mathfrak{N} \ni n\to \infty$, and the  limit superior (resp. inferior) is nonzero  only if $\dim T_{\jmath,\iota}+\dim B_\iota+\dim \mathcal X^\jmath=\dim G_\iota$ (we assume this equation in the rest of this paragraph). By Proposition \ref{pro-index-injection}, we see that $T_{\jmath,\iota}$ is an open dense subscheme of some torus  $R_{\jmath,\iota}\in J(T)^F$. By Lemma \ref{lem-lim-chara} and various lemmas introduced above, we see that (for $\gamma$ a representative of $\bar \gamma \in\bar N(\iota,T)^F$)  $$\lim_{\mathcal P_{\mathfrak{N}}\ni n\to \infty}
    \sum_{s\in T_{\jmath,\iota}^{F^n}}\frac{\chi\circ \mathrm N_T^n(\gamma^{-1}s \gamma)\cdot Q_{\gamma T\gamma^{-1}}^{G_\iota,(n)}(e)\cdot |(\mathcal X^{\jmath})^{F^n}|}{|\bar N(\iota,T)^{F^n}|\cdot|G_\iota^{F^n}|}=\frac{\langle ^{\gamma}\chi_{\jmath,\iota} ,1_{R_{\jmath,\iota}^F}\rangle_{R_{\jmath,\iota}^F}\cdot (-1)^{\sigma (G_\iota)+\sigma(T)}\cdot \mathcal C(\mathcal X^\jmath) }{|\bar N(\iota,T)^{F}|},$$
    where $^{\gamma}\chi_{\jmath,\iota}$ denotes the character $R_{\jmath,\iota }^F\to\Qlb^\times$ given by $t\mapsto \chi(\gamma^{-1}t\gamma)$.
    
  Consequently, by Theorem \ref{thm-summation-main} we have (note that: by our choice of $\mathfrak N$, the index sets of the outer two sum remain unchanged while $n$ varies in $\mathcal P_{\mathfrak N}$)  $$\lim_{\mathcal P_{\mathfrak{N}}\ni n\to \infty}\frac{1}{|G^{^{F^n}}|} \sum_{t\in G^{F^n} \text{ss.}}\left(R_{T,\chi\circ \mathrm N_{T}^n}^{(n)}\cdot \mathrm {Ind}_{\mathrm S^{F^n}}^{G^{F^n}}(\psi_{\mathrm S}^{(n)})\right)(t)=\sum_{(\jmath,\iota)}\sum_{\bar \gamma \in \bar N(\iota,T)^{F}}\frac{\langle ^\gamma\chi _{\jmath,\iota},1_{R_{\jmath,\iota}}\rangle_{R_{\jmath,\iota}^F}\cdot (-1)^{\sigma (G_\iota)+\sigma(T)}\cdot \mathcal C(\mathcal X^\jmath) }{|\bar N(\iota,T)^{F}|},$$
  where $(\jmath,\iota)$ runs over the pairs in $2^{J(T),F}\times I(T)^F$ such that  $\dim T_{\jmath,\iota}+\dim B_\iota+\dim \mathcal X^\jmath=\dim G_\iota$.  Unwinding the definition, we have 
  $$ \sum_{(\jmath,\iota)}\sum_{\bar \gamma \in \bar N(\iota,T)^{F}}\frac{\langle ^\gamma\chi _{\jmath,\iota},1_{R_{\jmath,\iota}^F}\rangle_{R_{\jmath,\iota}^F}\cdot (-1)^{\sigma (G_\iota)+\sigma(T)}\cdot \mathcal C(\mathcal X^\jmath) }{|\bar N(\iota,T)^{F}|}=\sum_{\delta=(\jmath,\iota,C)\in\Delta_T^F}(-1)^{\sigma (T)+\sigma(G_\iota)} t_{\delta,\chi}.
  $$
  We win by using Lemma \ref{lem-sigma-cent}  below. (Note that $G_\iota$ is by definition the centralizer of some/all elements $t\in T_{\jmath,\iota}(\mathrm k)$. Hence $G_\iota$ equals the centralizer $C_{G}(R_{\jmath,\iota})$.)
\end{proof}

We use the following elementary lemma in the proof of Theorem \ref{thm-limit-main}, and we omit the proof.
\begin{lem}\label{lem-sigma-cent}
    Let $R\in J(T)^F$: that is, the $F$-stable subtorus $R$ of $T$ is $G(\mathrm k)$-conjugate to $U$. Then we have $(-1)^{\sigma(C_G(R))}=(-1)^{\sigma(G)}$.
\end{lem}

\begin{thm}\label{thm-mul-main}We have the following equation.
    $$\mathfrak{S}_{T,\chi}(1)=\langle R_{T,\chi},\psi^{(1)}_{\mathrm S}\rangle_{\mathrm S^F}=\sum_{\delta\in \Delta_T^F}(-1)^{\sigma (T)+\sigma(G)}\mathrm t_{\delta,\chi}$$
\end{thm}

\begin{proof}
    Combine Theorem \ref{thm-limit-main} and \ref{thm-passing-to-semisimple}: in Theorem \ref{thm-limit-main}, we take the limit for $\mathcal P_{\mathfrak{NK}}\ni n\to \infty$, where:
    \begin{itemize}
    \item $\mathfrak{K}$ is the integer introduced in (6) of Theorem \ref{thm-passing-to-semisimple};
    \item $\mathcal P_{\mathfrak{NK}} $ is the arithmetic progression $\{1+i\mathfrak{NK}\}_{i\in \mathbb N}$.
     \end{itemize}
    
\end{proof}

\begin{prop}\label{pro-bij-delta}
    The map $\mathrm b_{T,B}:\Delta_T\to \Phi_B$ is a bijection.  The assignment $\delta\mapsto T_\delta$ yields a bijection $\mathrm z_T:\Delta_T\to J(T)$. Moreover, the bijection $\mathrm z_T$ is compatible with the action of Frobenius; in particular, it induces a bijection $\Delta^{F}_T\cong J(T)^F$.
\end{prop}
\begin{proof}
    We have already seen that $\mathrm b_{T,B}$ is an injection Corollary \ref{cor-index-injection}. At present stage, we note that:
    \begin{itemize}
        \item $\mathrm z_T$ is a well-defined injection by Corollary \ref{cor-Tdelta-vs-delta};
        \item the cardinality of $J(T)$ agrees with that of $\Phi_B$.
    \end{itemize}
    Consequently, in proving the first two statements of this proposition, it suffices to show the claim: the cardinality of $\Delta_T$ is no less than that of $\Phi_B$.

    In proving the above claim, by obvious symmetry, it suffices to deal with the case  that $T\subset B$ is the pair $\mathrm T\subset \mathrm B^{op}$ introduced in Definition \ref{def-various} (and we do assume so in the rest of this proof), and that the Frobenius structure $F=F_{\mathrm {GL}}$. Then by Theorem \ref{thm-mul-main}, we have 
    $$
    \mathfrak{S}_{T,1}(1)=|\Delta_T^F|.
    $$
    On the other hand, by Proposition \ref{pro-bijec-phi} and Mackey formula, 
    we have $$\mathfrak{S}_{T,1}(1)=|\Phi|=|\Phi_{\mathrm B^{op}}|.$$
(Note that $R_{\mathrm T,1}^G$ is the character of $\mathrm{Ind}_{ B^F}^{G^F}1_{ B^F}$.)

 This proves the above claim.
 
    The bijection $\mathrm z_T$ is obviously compatible with the action of Frobenius.
    
    This completes the proof.
\end{proof}

\begin{rmk}
    The bijectivity of $\mathrm b_{T,B}$ can be deduced using brutal force.
\end{rmk}

\begin{defn}
     Fix $R\in J(T)^F$, that is, the $F$-stable subtorus $R$ of $T$ is $G(\mathrm k)$-conjugate to $U$. We set $\mathrm t_{R,\chi}:=\langle \chi_{|R^F},1_{R^F}\rangle_{R^F}$.
\end{defn}

\begin{thm}\label{thm-mul-refined}
    We have the following equation.
      $$\mathfrak{S}_{T,\chi}(1)=\langle R_{T,\chi},\psi^{(1)}_{\mathrm S}\rangle_{\mathrm S^F}=\sum_{R\in J(T)^F}(-1)^{\sigma (T)+\sigma(G)}\mathrm t_{R,\chi}.$$
\end{thm}
\begin{proof}
    Combine Theorem \ref{thm-mul-main} and  Proposition \ref{pro-bij-delta}.
\end{proof}

\begin{rmk}
    In view of the bijections introduced in Proposition \ref{pro-bij-delta} and \ref{pro-bijec-phi}, Theorem \ref{thm-mul-refined} can be viewed as a variant of Mackey formula.
\end{rmk}

\section{multiplicity for irreducible representations}\label{sec-mul-general-irreducible}

Throughout this section, let $G$ be as introduced in Section \ref{sec-doublecoset}, and $F:G\to G$ be either $F_{\mathrm {GL}}$ or $F_{\mathrm U}$ introduced in Subsection \ref{subsec-intro-Fro}.

By Lusztig map and a theorem in \cite{LS}, we see that (the characters of) all irreducible representations of $G^F$ are indeed a linear combination of Deligne-Lusztig characters. In this section, we deduce a formula concerning the multiplicity of Shalika models for all irreducible representations of $G^F$, using the strategy above and Theorem \ref{thm-mul-refined}.

\subsection{Dual group and dual torus}

Let $G^*$ be the dual group of $G$ defined as in \cite[Definition 5.21]{DL}. By abuse of the notation, we denote the geometric Frobenius of $G^*$ again by $F$.

Fix an identification  $\mathrm{Cont}\left(\mathscr{T}(\mathbb G_m(\mathrm k)),\mathbb Q/\mathbb Z\right)\cong \mathrm k^\times$, where $\mathscr{T}(\mathbb G_m(\mathrm k))$ denotes the Tate module of $\mathbb G_m(\mathrm k)$ and $\mathrm{Cont}(\cdot,\cdot)$ denotes the group of continuous homomorphisms. Fix an inclusion of groups $\mathbb {Q}/\mathbb{Z}\hookrightarrow \Qlb^\times$.
By \cite[(5.21.5)]{DL}, there exist a natural bijection between $G^F$-conjugacy classes of pairs $(T,\theta)$ and $(G^*)^F$-conjugacy classes of pairs $(T',s)$, where:
\begin{itemize}
    \item $T$ is an $F$-stable maximal torus of $G$, and $\theta:T^F\to \Qlb^\times$ is a character;
    \item $T'$ is an $F$-stable maximal torus of $G^*$, and $s\in (T')^F$.
\end{itemize}

%\begin{rmk}
 For $F$-stable maximal tori $T$ and characters $\chi:T^F\to \Qlb^\times$, the virtual characters $R_{T,\chi}$ depend only on the $G^F$-conjugacy classes of the pairs $(T,\chi)$. 
 
 Consequently, for an $F$-stable maximal torus $T'$ of $G^*$ and $s\in (T')^F$, we may denote $R_{T',s}:=R_{T,\chi_s}$, where the $G^F$-conjugacy class of $(T,\chi_s)$ corresponds to the $(G^*)^F$-conjugacy class of $(T',s)$ under the bijection formulated in \cite[(5.21.5)]{DL}.
%\end{rmk}

The rest of this subsection is dedicated to reformulating Theorem \ref{thm-mul-refined} in terms of these dual data. We need some preliminaries.

Recall the $F$-stable torus $\mathrm T$ of $G$ introduced in \ref{def-various} and its subtorus $U$ introduced in Definition \ref{def-torusU}. Let $U^\heartsuit$ be the quotient torus of $\mathrm T$ fitting into the following exact sequence of tori:
$$
e\rightarrow U\rightarrow \mathrm T\rightarrow U^\heartsuit \rightarrow e.
$$
Let $U^\perp$ be the dual torus of $U^\heartsuit$. Passing to the dual, we get:
$$
e\rightarrow U^\perp \rightarrow \mathrm T^* \rightarrow U^*\rightarrow e.
$$

Note that $\mathrm T^*$ is a maximal torus of $G^*$. 

\begin{defn}\label{def-J*}
    Fix a maximal torus $T$ of $G^*$. Let $J^*(T)$ be the set of subtori $Z$ of $T$ such that $Z$ is $G^*(\mathrm k)$-conjugate to $U^\perp$. For $s\in T(\mathrm k)$, let $J^*_s(T):=\{Z\in J^*(T): s\in Z(\mathrm k)\}$. 

    Suppose that $T$ is $F$-stable and $s\in T^F$, then the sets $J^*(T)$ and $J^*_s(T)$ have been endowed with the natural action of Frobenius given by $Z\mapsto F(Z)$. We denote the fixed-point sets by $J^{*,F}(T)$ and $J^{*,F}_s(T)$ respectively. 
\end{defn}

It is easy to see that the following theorem rephrases Theorem \ref{thm-mul-refined}. 

\begin{thm}\label{thm-dual-main}
    Fix an $F$-stable maximal torus $T$ of $G^*$. Let $s\in T^F$. Then we have
    $$
    \langle R_{T,s},\psi_{\mathrm S}^{(1)}\rangle_{\mathrm S^F}=(-1)^{\sigma(T)+\sigma(G)}|J^{*,F}_s(T)|.
    $$
\end{thm}

\subsection{A multiplicity free representation}\label{subsec-canonical-rep} We retain the notation as in the previous subsection.
Throughout this subsection, we fix a semisimple element $s\in (G^*)^F$. Let $G^*_s$ denote the centralizer of $s$ in $G^*$. Note that $G^*_s$ is a connected reductive group defined over $\mathbb F_q$.
In this subsection, we define a representation of ``the" Weyl group of $G^*_s$, which will play important role in the desired multiplicity formula.

At first, we introduce some notions, which essentially reformulates \cite[Section 1.1]{DL}.

\begin{defn}\label{def-Cs}
    Let $\mathfrak C_s$ be the groupoid (indeed, a setoid) characterized by the following data:
    \begin{itemize}
        \item The set of objects consists of Borel pairs $(B,T)$ of $G^*_s$;
        \item For any two Borel pairs $p_1=(B_1,T_1)$ and $p_2=(B_2,T_2)$, there exists a \textbf{unique} morphism $m(p_1,p_2)$ from $p_1$ to $p_2$. 
    \end{itemize}
\end{defn}

\begin{rmk}\label{rm-canonical-Ws}
Let $\mathcal {GP}$ denote the category of groups.
    There exists a functor $WG:\mathfrak{C}_s\to \mathcal{GP}$, characterized by the following conditions:
    \begin{itemize}
        \item For a Borel pair $p=(B,T)\in \mathfrak{C}_s$, we set $WG(p)=W_s(T)$, where $W_s(T)=N_{G^*_s}(T)/T$ denotes the Weyl group of $T$ in $G^*_s$;
        \item For Borel pairs $p_1=(B_1,T_1)$ and $p_2=(B_2,T_2)$, the functor $WG$ transforms the morphism $m(p_1,p_2)$ into the isomorphism $W(T_1)\to W(T_2)$ given by $w T_1\mapsto g^{-1}wg T_2$, where $g\in G^*_s(\mathrm k)$ satisfies $g^{-1}T_1g=T_2$ and $g^{-1}B_1g=B_2$. (We verify that this isomorphism is independent of the choice of such $g$, since such $g$ form a $T_1$-torsor under the left action of $T_1$.) 
    \end{itemize}
    
\end{rmk}

\begin{defn}
    Let the notation be as in Remark \ref{rm-canonical-Ws}. We say that the limit $\mathbb W_s$ of $WG$ is the Weyl group of $G^*_s$. (This coincides with the definition adopted in \cite[Section 1.1]{DL}. In particular, the Weyl group $\mathbb W_s$ has been endowed with an action of the Frobenius $F$.)
\end{defn}

\begin{rmk}\label{rm-canonical-J*}
    Let $\mathcal {Set}$ denote the category of sets. There exists a functor $\mathcal J^*_s:\mathfrak{C}_s\to \mathcal{Set}$, characterized by the following conditions:
    \begin{itemize}
        \item For a Borel pair $p=(B,T)$ of $G^*_s$, we have $\mathcal J^*_s(p)=J^*_s(T)$, where $J^*_s(T)$ is defined in Definition \ref{def-J*};
        \item For Borel pairs $p_1=(B_1,T_1)$ and $p_2=(B_2,T_2)$, the functor $\mathcal J^*_s$ transforms the morphism $m(p_1,p_2)$ into the isomorphism $J_s^*(T_1)\to J^*_s(T_2)$ given by $Z\mapsto g^{-1} Zg$ for $Z\in J^*_s(T_1)$, where $g\in G^*_s(\mathrm k)$ satisfies $g^{-1}T_1g=T_2$ and $g^{-1}B_1g=B_2$. (Again, we verify that this isomorphism is independent of the choice of such $g$.)
    \end{itemize}
\end{rmk}

\begin{defn}\label{def-caonical-J*}
    Let the notation be as in Remark \ref{rm-canonical-J*}, we define the set $\mathbb J^*_s$ to be the limit of the functor $\mathcal{J}_s^*$.
\end{defn}

\begin{rmk}\label{rm-action-canonical}
    By the construction of $\mathbb J^*_s$, it has been endow with an action of the Weyl group $\mathbb W_s$, and an action of the Frobenius, characterized by the following conditions:
    \begin{itemize}
        \item Fix an $F$-stable Borel pair $(B,T)$ of $G^*_s$. The following diagram commutes
        \begin{equation}
            \xymatrix{
\mathbb W_s \times \mathbb J^*_s\ar[r]\ar[d]&\mathbb J^*_s\ar[d]\\
      W_s(T)\times  J^*_s(T)\ar[r]& J^*_s(T)      
            }
        \end{equation}  where the vertical arrows are the projections (such projections are all isomorphisms) given by the construction of $\mathbb W_s$ and $\mathbb J_s$, and the horizontal arrows are the actions.
    \item %A similar but more simple diagram commutes for the action of the Frobenius.
    Fix an $F$-stable Borel pair $(B,T)$ of $G^*_s$. The following diagram commutes
        \begin{equation}
            \xymatrix{
 \mathbb J^*_s\ar[r]\ar[d]&\mathbb J^*_s\ar[d]\\
       J^*_s(T)\ar[r]& J^*_s(T)      
            }
        \end{equation}
        where the  vertical arrows are the projections given by the construction of $\mathbb J_s^*$ and the horizontal arrows are the actions of the Frobenius.
    \end{itemize}
\end{rmk}

\begin{defn}\label{def-ext-caonical-weyl}Let $\langle \mathfrak F\rangle$ denote the free abelian group with the generator $\mathfrak F$.
    Let $\wt {\mathbb W}_s:=\mathbb W_s \rtimes \langle \mathfrak F\rangle$ be the semidirect product, where we have $ \mathfrak F w\mathfrak F^{-1}=F(w)$ for $w\in \mathbb W_s$.
\end{defn}

It is clear that the two actions introduced in Remark \ref{rm-action-canonical} extend, or combine, to give an action of $\wt{\mathbb W}_s$ on $\mathbb J^*_s$, with $\mathfrak{F}$ acting as the Frobenius.

\begin{defn}\label{def-caonnical-rep-pis}
    Let $\wt \Pi_s$ be the representations of $\wt{\mathbb W}_s$, given by the $\Qlb$-span of the set $\mathbb J^*_s$. Note that $\mathfrak{F}$ acts as an automorphism of finite order.
\end{defn}

We have the following obvious lemma.
\begin{lem}
    Let $T$ be an $F$-stable maximal torus of $G^*_s$: that is, the torus $T$ is an $F$-stable maximal torus of $G^*$ and $s\in T^F$. Then we have $$|J^{*,F}_s(T)|=\mathrm {Tr}(w\mathfrak{F},\wt \Pi_s),$$ where $w\in \mathbb W_s$ is a representative of the $F$-conjugacy class corresponding to $T$ in the sense of \cite[Corollary 1.14]{DL}.
\end{lem}

Consequently, we have the following by Theorem \ref{thm-dual-main}.
\begin{thm}\label{thm-dual-canonical-main}
    Let $T$ be an $F$-stable maximal torus of $G^*_s$: that is, the torus $T$ is an $F$-stable maximal torus of $G^*$ and $s\in T^F$. Then we have
    $$
    \langle R_{T,s},\psi_{\mathrm S}^{(1)}\rangle_{\mathrm S^F}=\mathrm {sgn}(w)(-1)^{\sigma(G^*_s)+\sigma(G)}\mathrm
    {Tr}(w\mathfrak{F},\wt \Pi_s),
    $$
    where $w\in \mathbb W_s$ is a representative of the $F$-conjugacy class corresponding to $T$ in the sense of \cite[Corollary 1.14]{DL}, and $\mathrm {sgn}$ is (the character of) the sign representation of $\mathbb W_s$. %(see Remark \ref{rm-sign-res} below).
\end{thm}

\begin{comment}
    
\begin{rmk}\label{rm-sign-res}
    Let $\mathbb W=\mathbb W_1$ be the Weyl group of $G^*$. Fix a Borel pair $(B_s,T)$ of $G^*_s$. Let $B$ be a Borel subgroup of $G^*$ such that $B\cap G^*_s=B_s$. Then we have an inclusion of $i_s:\mathbb W_s\hookrightarrow \mathbb W$ given by the composition $\mathbb W_s \xrightarrow[(B_s,T)]{\sim} W_{G^*_s}(T)\hookrightarrow W_{G^*}(T)\xrightarrow[(B,T)]{\sim}\mathbb W$,
    where $W_{G^*_s}(T)$ and $W_{G^*}(T)$ denote the Weyl groups of $T$ in the corresponding reductive groups, the middle arrow is the natural inclusion given by definition, and the first (resp. the last) arrow  is the isomorphism given by the projection (resp. the inverse of the projection) at the corresponding Borel pair. We verify that the pullback of the sign representation of $\mathbb W$ along $i_s$ is precisely the sign representation of $\mathbb W_s$.
\end{rmk}
\end{comment}

\begin{comment}
\begin{rmk}\label{rm-F-rank-nonempty-J}
Fix semisimle $s\in (G^*)^F$.
    Suppose that $\mathbb J^*_s$ is nonempty, we verify that $\sigma(G^*_s)$ is even.
\end{rmk}
\end{comment}

\begin{rmk}
   We will see later in Remark \ref{rm-mul-one-induced} that the restriction of $\wt \Pi_s$ to $\mathbb W_s$ is a multiplicity free representation.
\end{rmk}
\subsection{The reductive group $G^*_s$}\label{subsec-G*s-structure;mul-one-canonical-rep}
We retain the notation as in the previous subsection. In particular, we fix a semisimple element $s\in (G^*)^F$. 

Recall the torus $\mathrm T$ introduced in Definition \ref{def-various} and the basis $\mathbb B$ of $V$. A typical element $t=(t_1,\ldots,t_{2\mathfrak{n}})\in \mathrm T(\mathrm k)$ sends $v_i$ to $t_i\cdot v_i$ for $1\leq i\leq 2\mathfrak{n}$, where $t_i\in \mathrm k^\times$. For $1\leq i\leq 2\mathfrak{n}$, let $x_i:\mathrm T \to \mathbb G_m$ be the character of $\mathrm T$ given by $t\mapsto t_i$. 

Since $\mathrm T^*$ is the dual torus of $\mathrm T$, each $x_i$ introduced above induces a cocharacter $y_i:\mathbb G_m\to \mathrm T^*$. 

Let $V'$ be a $2\mathfrak{n}$ dimensional vector space over $\mathrm k$ with basis $\mathbb B'=\{u_i\}_{1\leq i\leq 2\mathfrak{n}}$. In the rest of this paper, we fix an isomorphism $\mathrm i_{G^*}:G^*\xrightarrow[]{\sim}\mathrm {GL}(V')$ of algebraic groups defined over $\mathrm k$ satisfying the following properties:
\begin{itemize}
    \item $\mathfrak{i}_{G^*}(\mathrm T^*)$ is the maximal torus of $\mathrm {GL}(V')$ that possesses $u_i$ for $1\leq i\leq 2\mathfrak{n}$ as eigenvectors;
    \item Fix an integer $1\leq i \leq 2\mathfrak{n}$. For $z\in \mathbb G_m(\mathrm k)=\mathrm k^\times$, the automorphism  $\mathrm{i}_{G^*} \left(y_i(z)\right)$ fixes $u_j$ for $j\neq i$ and sends $u_i$ to $z\cdot u_i$.
\end{itemize}
It is not hard to see that such an identification exists. In the rest of this paper, we will implicitly identify $G^*$ with $\mathrm {GL}(V')$ via $\mathrm i_{G^*}$. In particular, for the semisimple element $s\in (G^*)^F$, it makes sense to speak of its eigenvalues.

The rest of this subsection is dedicated to an explicit description of the various notions introduced in the previous subsections.

\begin{rmk}\label{rm-Uperp-explicit}
    Under the identification exhibited above, we may rephrase the algebraic subgroup $U^\perp$ of $G^*$ as follows:
    \begin{itemize}
        \item It is isomorphic to $\mathbb G_m^{\mathfrak n}$ (the $\mathfrak n$-fold copy of $\mathbb G_m$), of which   we denote by $t=(t_1,\ldots,t_\mathfrak n)$ a typical element;
        \item $t\cdot u_i=t_i u_i$, if $i\leq \mathfrak{n}$; and $t\cdot u
        _j=t_{j-\mathfrak{n}}^{-1}u_j$ if $\mathfrak{n}+1\leq j\leq 2\mathfrak{n}$.
    \end{itemize}
   For $1\leq i\leq \mathfrak{n}$, let $y^U_i:\mathbb G_m\to U^\perp$ be the cocharacter given by $k\mapsto (k_1,\ldots,k_{\mathfrak{n}})$, where $k_i=k$ and $k_j=1$ for $j\neq i$.
\end{rmk}

\begin{rmk}\label{rm-s-effective} Fix an $F$-stable maximal torus $R$ of $G^*_s$. 
    By Remark \ref{rm-Uperp-explicit}, it is elementary to see the following. 
    
    The set $J^*_s(R)$ is nonempty if and only if the following conditions hold:
    \begin{itemize}
        \item the multiplicities of the eigenvalues $\pm 1$ are even (possibly being $0$);
        \item if $k\in \mathrm k^\times$ is an eigenvalue of $s$, then $k^{-1}$ is likewise an eigenvalue of $s$.
    \end{itemize} 
\end{rmk}

For each eigenvalue $r\in \mathrm k^\times$ of $s$, let $V'_r$ be the eigenspace of $s$ corresponding to $r$. Note that $V'_r$ is not necessarily defined over $\mathbb F_q$.

\begin{rmk}\label{rm-G*s-decomp-expli}
    We have a natural identification $$G^*_s\cong\prod_{r} \mathrm {GL}(V'_r)$$ of algebraic groups over $\mathrm k$, where $r$ ranges over the set of eigenvalues of $s$.
\end{rmk}

\begin{prop}\label{pro-mul-one-induced}Suppose in this proposition that $\mathbb J^*_s$ is nonempty. Then, via the action introduced in Remark \ref{rm-action-canonical}, the Weyl group $\mathbb W_s$ acts transitively on the set $\mathbb J^*_s$. Fix any $x\in \mathbb J^*_s$. Let $\mathrm{St}_x$ be the stabilizer of $x$ under the action of $\mathbb W_s$. Then $(\mathbb W_s, \mathrm {St}_x)$ is a Gelfand pair: that is, the induced representation $\mathrm {Ind}_{\mathrm {St}_x}^{\mathbb W_s}(1_{\mathrm {St}_x})$ is multiplicity-free.
\end{prop}
\begin{proof}Let $s$ be as in the proposition. We do not exclude the case that the characteristic is $2$, but assume $1\neq -1$ for simplicity. The proof essentially consists of linear algebra computation.

  In this proof, we retain the notation in Remark \ref{rm-G*s-decomp-expli}.  Let $\mathrm e_s$ be the set of eigenvalues of $s$. By Remark \ref{rm-s-effective}, we have an involution on $\mathrm e_s$ given by $r\mapsto r^{-1}$. Let $[\mathrm e_s]$ denote the set of orbits with respect to this involution.
For each $o\in[\mathrm e_s]$, we define the group $G^*_{s,o}$ to be the product of $\mathrm {GL}(V'_r)$ and $\mathrm {GL}(V'_{r^{-1}})$, if $o=\{r,r^{-1}\}$, and to be $\mathrm {GL}(V'_1)$ (resp. $G^*_{s,-1}=\mathrm {GL}(V'_{-1})$) if $o=\{1\}$ (resp. $o=\{-1\}$). Hence we have 
\begin{equation}\label{eq-decompo-G*s}
G^*_s \cong \prod_{o\in [\mathrm e_s]} G^*_{s,o}.
\end{equation}

 In proving this proposition, we have nothing to do with the Frobenius action. Consequently, we may assume that $G^*_s$ contains a $(G^*)^F$-conjugate of $\mathrm T^*$ by replacing $\mathbb F_q$ by an extension. We may in turn assume that $s\in \mathrm (\mathrm T^*)^F$ and $U^\perp \in J^*_s(\mathrm T^*)$ by replacing $\mathbb F_q$ by an extension and applying a conjugation. We do assume so in the rest of this proof. Fix a Borel subgroup $B_s$ of $G^*_s$ that contains $\mathrm T^*$, we identify $\mathbb W_s$ (resp. $\mathbb J^*_s$) with $W_{G^*_s}(\mathrm T^*)$ (resp. $J^*_s(\mathrm T^*)$).  

Fix $Z\in J^*_s(\mathrm T^*)$ in this paragraph. By definition, there exists $g\in G^*(\mathrm k)$ such that $g U^\perp g^{-1}=Z$. Let $y_i^Z:\mathbb G_m\to Z$ be the cocharacter of $Z$ given by $\mathrm k^\times\ni k\mapsto gy_i^U(k)g^{-1}$, where $y_i$ is the cocharacter of $U^\perp$ introduced in the Remark \ref{rm-Uperp-explicit}. Since $Z\subset \mathrm T^*$, we may assume $g\in N_{G^*}(\mathrm T^*)$ (fix this $g$ in the rest of this paragraph). This shows that the image of $y_i^Z$ has the fixed vector space of dimension $2\mathfrak{n}-2$, and has at most $2$ eigenvalues that $\neq 1$ (such two eigenvalues must be mutual inverse of each other). By definition, the semisimple element $s$ can be written as a product of elements $s_i\in y^Z_i(\mathrm k^\times)$ for $1\leq i\leq \mathfrak{n}$. This shows that each $y_i^Z$ factors through a group of the form $G^*_{s,o}$ for some orbit $o\in[\mathrm e_s]$. Counting the dimension, we see that:
\begin{itemize}
    \item if $o=\{r,r^{-1}\}$, there are precisely $\dim V'_r$ many $i$ %$1\leq i \leq \mathfrak{n}$
    such that $y^{Z}_i$ factors through $G^*_{s,o}$;
    \item if $o=\{1\}$ (resp. $o=\{-1\}$), there are precisely $\frac{\dim V'_{1}}{2}$ (resp. $\frac{\dim V'_{-1}}{2}$) many $i$ % $1\leq i \leq \mathfrak{n}$ 
    such that $y^{Z}_i$ factors through $G^*_{s,o}$.
\end{itemize}
Hence we have 
\begin{equation}\label{eq-decomp-Z}
    Z=\prod_{o\in [\mathrm e_s]}Z_o;
\end{equation}
where $Z_o$ is the subgroup of both $Z$ and $G^*_{s,o}$ generated by the cocharacters $y_i^Z$ that factor through $G^*_{s,o}$. Observe by the above argument that $\dim Z_0=\frac{\rank (G^*_{s,o})}{2}$, where $\rank(\cdot)$ denotes the absolute rank.

The decomposition in Remark \ref{rm-G*s-decomp-expli} exhibits a decomposition of the Weyl group \begin{equation}\label{eq-decomp-Weyl}
    \mathbb W_s \cong \prod_{o\in [\mathrm e_s]} \mathbb W_{o},
\end{equation}
where $\mathbb W_{o}$ denotes the Weyl group of $G^*_{s,o}$. 
Hence we have $\mathbb W_o\cong S_{\dim V'_r}\times S_{\dim V'_{r}}$ if $o=\{r,r^{-1}\}$ and $\mathbb W_o \cong S_{\dim V'_r}$ if $o=\{r\}$.

By the decomposition (\ref{eq-decomp-Z}) and (\ref{eq-decomp-Weyl}), it is easy to see that the set $\mathbb J^*_s\cong J^*_s(\mathrm T^*)$ consists of the subtori of $\mathrm T^*$ of the form $wU^\perp w^{-1}$ for $w\in W_{G^*_s}(\mathrm T^*)$. The conjugation action of $W_{G_s^*}(\mathrm T^*)$ coincide with the action of $\mathbb W_s$ on $\mathbb J^*_s$ by construction (see Remark \ref{rm-action-canonical}). Hence the transitivity is clear. It remains to show that $(\mathbb W_s,\mathrm {St}_x)$ is a Gelfand pair, for every $x\in \mathbb J^*_s$.

Without loss of generality, we may assume that $x=U^\perp$ (we assume so in the rest of this proof). By the decomposition (\ref{eq-decomp-Z}), we have the corresponding 
$$
U^\perp \cong \prod_{o\in [\mathrm e_s]}Z^U_o.
$$

Let $\mathbb J_{s,o}^*$ be the set of subtori of the form $w Z_o^Uw^{-1}$ for $w\in \mathbb W_o$.
In view of the decompositions (\ref{eq-decompo-G*s}), (\ref{eq-decomp-Z}) and the transitivity assertion, we see that $$\mathbb J^*_s\cong \prod \mathbb J^*_{s,o}.$$
This identification is compatible with the decomposition (\ref{eq-decomp-Weyl}).
For $o\in [\mathrm e_s]$, let $\mathrm {St}_{o}^x$ be the stabilizer of $Z_o^U$ with respect to the action of $\mathbb W_o$ on $\mathbb J^*_{s,o}$. It suffices to show that $(\mathbb W_o, \mathrm {St}_{o}^x)$ is a Gelfand pair for each $o\in[\mathrm {e}_s]$. %Note that we have $ \mathrm {St}_{s,o} \cong S_{\dim V'_r}\xrightarrow[]{diag}S_{\dim V'_r}\times S_{\dim V'_{r}} \cong \mathbb W_o$ if $o=\{r,r^{-1}\}$, and 

In down to earth terms, we are reduced to show:
\begin{itemize}
    \item[(1)] The pair $(S_{\dim V'_r}\times S_{\dim V'_r},S_{\dim V'_r})$ is a Gelfand pair, where the inclusion is given by the diagonal map; (These correspond to the case $o=\{r,r^{-1}\}$)
    \item[(2)] The pair $(S_{2n}, (\mathbb Z/2\mathbb Z)^{n}\rtimes S_n)$ is a Gelfand pair, where $n$ is a positive integer. (These correspond to the case $o=\{1\}$ or $o=\{-1\}$.)
\end{itemize}
The pair in (1) is clearly a symmetric Gelfand pair. By \cite[Proposition 2.11]{GM}, we see that the pair in (2) is a Gelfand pair. This completes the proof.
\end{proof}

\begin{rmk}\label{rm-mul-one-induced}
    In view of Definition \ref{def-caonnical-rep-pis}, the restriction of $\wt \Pi_s$ to $\mathbb W_s$ is isomorphic to the induced representation $\mathrm {Ind}_{\mathrm {St}_x}^{\mathbb W_s}(1_{\mathrm {St}_x})$ introduced in Proposition  \ref{pro-mul-one-induced}, which is multiplicity-free and self-dual.
\end{rmk}

\begin{rmk}\label{618}
Let the assumptions and the notation be as in Proposition \ref{pro-mul-one-induced}. Following the proof of Proposition \ref{pro-mul-one-induced}, we may informally describe the stabilizer $\mathrm {St}_x$ as follows. Note that all algebraic groups mentioned below are defined over $\mathrm k$.

Recall the decomposition (\ref{eq-decompo-G*s}). We may identify $G^*_{s,o}$ with $\mathrm{GL}_{\dim V'_r}$ when $o=\{r\}$, and with $\mathrm{GL}_{\dim V'_{r}}\times \mathrm {GL}_{\dim V'_r}$ when $o=\{r,r^{-1}\}$. Note that the group $\mathbb W_o$ introduced in (\ref{eq-decomp-Weyl}) is the Weyl group of $G^*_{s,o}$. For each $o\in [\mathrm e_s]$ (see the proof of \ref{pro-mul-one-induced} for notation), we attach the subgroup $\mathrm{St}_{o}^x$ of $\mathbb W_o$:
\begin{itemize}
    \item When $o=\{r\}$, the group $\mathbb W_o$ is identified with the Weyl group of $\mathrm {GL}_{\dim V'_r}$. Note that we have $r=\pm 1$ and the dimension $\dim V'_r$ is even by Remark \ref{rm-s-effective}. The subgroup $\mathrm{St}_{o}^x$ of $\mathbb W_o$ is the Weyl group of the symplectic group $\mathrm {Sp}_{\dim V'_r}$;
    \item When $o=\{r,r^{-1}\}$, the group $\mathbb W_o$ is identified with the Weyl group $S_{\dim V'_r}\times S_{\dim V'_r}$ of $\mathrm{GL}_{\dim V'_{r}}\times \mathrm {GL}_{\dim V'_r}$. The subgroup $\mathrm{St}_{o}^x$ is $S_{\dim V'_r}$ embedded diagonally into $\mathbb W_o\cong S_{\dim V'_r}\times S_{\dim V'_r}$.
\end{itemize}
Then we have $$\mathrm {St}_x\cong \prod_{o\in [\mathrm e_s]}\mathrm {St}_{o}^x,$$
up to conjugation.
In particular, by Remark \ref{rm-s-effective}, there exists a subgroup $\mathrm {Sp}_{2\mathfrak n}$ of $G^*$ through which $s$ factors. (We view $s$ as an element of $\mathrm{Sp}_{2\mathfrak{n}}$ in what follows.) And the stabilizer $\mathrm {St}_x$ is conjugate to the Weyl group of the centralizer $C_{\mathrm {Sp}_{2\mathfrak{n}}}(s)$.
\end{rmk}

\subsection{Recollection of \cite{LS}}\label{subsec-review-LS}
We retain the notation as in Subsection \ref{subsec-G*s-structure;mul-one-canonical-rep}. In particular, we fix a semisimple $s\in (G^*)^F$. By Remark \ref{rm-G*s-decomp-expli}, the Coxeter diagram of $G^*_s$ is a disjoint union of connected ones of type A.
Recall the group $\wt{\mathbb W}_s=\mathbb W_s\rtimes \langle\mathfrak{F}\rangle$ of $G^*_s$ defined in Definition \ref{def-ext-caonical-weyl}, which contains the Weyl group $\mathbb W_s$ of $G^*_s$ as a normal subgroup.

\begin{defn}\label{def-Ws-dual}
    We define $\mathbb W_s^\vee$ to be the set of (isomorphism classes of) irreducible representations of $\mathbb W_s$. Since the Frobenius $F$ acts on $\mathbb W_s$, it permutes $\mathbb W_s^\vee$. Let $(\mathbb W_s^\vee)^F$ be the fixed-point set.
\end{defn}

For $w\in \mathbb W_s$, let $R^s_w$ denote the unipotent Deligne-Lusztig character of $(G^*_s)^F$ corresponding to $w$. The following theorem rephrases \cite[Theorem 2.2]{LS}.

\begin{thm}\label{thm-G*s-unipotente-rep}
    For $\rho\in (\mathbb W_s^\vee)^F$, there exists a unique irreducible representation $\wt \rho$ (up to isomorphism) of $\wt{\mathbb W}_s$, characterized by the following conditions:
    \begin{itemize}
        \item the restriction of $\wt \rho$ to $\mathbb W_s$ is isomorphic to $\rho$;
        \item the element $\mathfrak{F}$ acts as an automorphism of finite order;
        \item the following virtual character of $(G^*_s)^F$
        \begin{equation}\label{eq-unipotent-linear-uniform}
            |\mathbb W_s|^{-1}\sum_{w\in \mathbb W_s} \mathrm{Tr}(w\mathfrak{F},\wt\rho) R_w^s
        \end{equation}
        is the character of a unipotent irreducible representation $\hat{\pi}_{s,\rho}$.
    \end{itemize}
    Moreover, the assignment $\rho \mapsto \hat{\pi}_{s,\rho}$ exhibits a bijection between $(\mathbb W_s^\vee)^F$ and the set of (isomorphism classes of) unipotent irreducible representations of $(G^*_s)^F$.
\end{thm}

\subsection{Lusztig map}\label{subsec-Lusztigmap}

For semisimple $s\in (G^*)^F$, let $\mathbb{L}_s$ denote the Lusztig map.

For $\rho\in (\mathbb W_s^\vee)^F$, let $\pi_{s,\rho}$ be such that $\mathbb L_s({\pi}_{s,\rho})=\hat{\pi}_{s,\rho}$. \begin{comment}We have
\begin{equation}\label{0eq-general-irreducible-linear-uniform-0}
    \pi_{s,\rho}=(-1)^{\sigma(G^*_s)+\sigma(G)}|\mathbb W_s|^{-1}\sum_{w\in \mathbb W_s} \mathrm{Tr}(w\mathfrak{F},\wt\rho) R_{T_w^*,s}.
\end{equation}
%Suppose (in this paragraph) that $\mathbb J^*_s$ is nonempty, then $\sigma(G^*_s)$ is even. %Hence by Equation (\ref{eq-unipotent-linear-uniform}) and the property of $\mathbb L_s$, 
\end{comment}
We have
\begin{equation}\label{eq-general-irreducible-linear-uniform}
   (-1)^{\sigma(G)+\sigma(G^*_s)} \pi_{s,\rho}=|\mathbb W_s|^{-1}\sum_{w\in \mathbb W_s} \mathrm{Tr}(w\mathfrak{F},\wt\rho) R_{T_w^*,s},
\end{equation}
where $\wt \rho$ is the representation of $\wt{\mathbb W}_s$ corresponding to $\rho$ in the sense of Theorem \ref{thm-G*s-unipotente-rep}, and $T_w^*$ is the $F$-stable maximal torus of $G^*_s$ corresponding to $w$ in the sense of \cite[Corollary 1.14]{DL}. 

Recall the representation $\wt \Pi_s$ introduced in Definition \ref{def-caonnical-rep-pis}.
\begin{thm}\label{thm-general-irreducible-main}
    We have
    $$
   S(\pi_{s,\rho})= \langle \pi_{s,\rho},\psi_{\mathrm S}^{(1)}\rangle_{\mathrm S^F}=\langle \rho,\sgn \otimes\Pi_s\rangle_{\mathbb W_s},
    $$
    where $\Pi_s$ is the restriction of $\wt \Pi_s$ to $\mathbb W_s$, and $\sgn$ is the sign representation of $\mathbb W_s$.
\end{thm}
\begin{proof} Fix a semisimple $s\in (G^*)^F$ and a representation $\rho\in (\mathbb W_s^\vee)^F$ throughout this proof.
    By definition, the number $ S(\pi_{s,\rho})= \langle \pi_{s,\rho},\psi_{\mathrm S}^{(1)}\rangle_{\mathrm S^F}$ is a nonnegative integer. 

In what follows, we assume that $\mathbb J^*_s$ is nonempty. (The case for empty $\mathbb J^*_s$  is similar but easier.)

    In view of Equation (\ref{eq-general-irreducible-linear-uniform}) and Theorem \ref{thm-dual-canonical-main}, we have
    \begin{equation}\label{eq-F-twisted-summation}
      \langle \pi_{s,\rho},\psi_{\mathrm S}^{(1)}\rangle_{\mathrm S^F}=|\mathbb W_s|^{-1}\sum_{w\in \mathbb W_s}\mathrm {Tr}(w\mathfrak{F},\wt \rho\otimes \sgn \otimes \wt \Pi_s)=|\mathbb W_s|^{-1}\sum_{w\in \mathbb W_s}\mathrm {Tr}(\mathfrak{F}w,\wt \rho\otimes \sgn \otimes \wt \Pi_s),
    \end{equation}
    where:
    \begin{itemize}\item $\sgn$ denotes the sign representation on which $\mathfrak{F}$ acts trivially;
    \item $\wt \rho$ is the representation of $\wt{\mathbb W}_s$ corresponding to $\rho$ in the sense of Theorem \ref{thm-G*s-unipotente-rep};
        %\item $(\wt \rho)^\vee$ is the dual representation of $\wt \rho$;
        \item we view $\wt \rho\otimes \sgn \otimes \wt \Pi_s$ as a representation that $\wt{\mathbb W}_s$ acts via the diagonal.
    \end{itemize}
   % (Note that $\sigma(G)=2\mathfrak{n}$ is even.)

    Let $V^\tau$ be the underlying vector space of $\tau:=\wt \rho\otimes \sgn \otimes \wt \Pi_s$.
     Let $$
     V^\tau= \bigoplus_{\kappa} V_{\kappa},
     $$
     be the decomposition of $\wt \rho\otimes \sgn \otimes \wt \Pi_s$ into isotypic parts with respect to $\mathbb W_s$, where $\kappa$ ranges over the isomorphism classes of irreducible representations of $\mathbb W_s$.
     By Remark \ref{rm-mul-one-induced}, for $\kappa=1_{\mathbb W_s}$ the trivial representation, the vector space $V_{1_{\mathbb W_s}}$ is of dimension $\langle \rho^\vee,\sgn \otimes\Pi_s\rangle_{\mathbb W_s}=\langle \rho,\sgn \otimes\Pi_s\rangle_{\mathbb W_s}\leq 1$, where $\rho^\vee$ denotes the dual representation of $\rho$ (the equality is due to the fact that $\Pi_s$ is self-dual by Remark \ref{rm-mul-one-induced}).  
     Let $P_1:V^\tau\to V^\tau$ be the projection to $V_{1_{\mathbb W_s}}$. The right-hand side of Equation (\ref{eq-F-twisted-summation}) can be written as
     $$
     \mathrm{Tr}(\mathfrak{F}\circ P_1,V^\tau).
     $$
     Since $\mathfrak{F}$ normalizes $\mathbb W_s$, it stabilizes $V_{1_{\mathbb W_s}}$. By Theorem \ref{thm-G*s-unipotente-rep} and the last statement of Definition \ref{def-caonnical-rep-pis}, the automorphism $\mathfrak{F}$ acts on $V_{1_{\mathbb W_s}}$ as a multiplication by a root of unity. (Note: if $V_{1_{\mathbb W_s}}=0$ is the trivial vector space, this statement trivially holds.)

     Consequently, we have that $
     \mathrm{Tr}(\mathfrak{F}\circ P_1,V^\tau)
     $ equals $\langle \rho,\sgn \otimes\Pi_s\rangle_{\mathbb W_s}$ up to multiplication by a root of unity. Since both $\mathrm{Tr}(\mathfrak{F}\circ P_1,V^\tau)=\langle \pi_{s,\rho},\psi_{\mathrm S}^{(1)}\rangle_{\mathrm S^F}$ and $\langle \rho,\sgn \otimes\Pi_s\rangle_{\mathbb W_s}$ are nonnegative integers (by the first paragraph of this proof), we have $
     \mathrm{Tr}(\mathfrak{F}\circ P_1,V^\tau)=\langle \rho,\sgn \otimes\Pi_s\rangle_{\mathbb W_s}
     $ as desired. This completes the proof.
\end{proof}

\begin{cor}\label{shalika-one}
    Shalika model has multiplicity-one: that is, for every irreducible representation $\pi$ of $G^F$, we have $S(\pi)=\langle \pi,\psi_{\mathrm S}^{(1)}\rangle_{\mathrm S^F}\leq 1$.
\end{cor}

\begin{proof}
     Combine Remark \ref{rm-mul-one-induced} and Theorem \ref{thm-general-irreducible-main}.
\end{proof}

\end{document}